              \def\version{25 May 2021}	        	%

\documentclass[reqno,11pt]{amsart} 
\usepackage[utf8]{inputenc}
\usepackage{amsmath} 
\usepackage[mathscr]{eucal}
\usepackage{amssymb}
\usepackage{srcltx} 
\usepackage{dsfont}
\usepackage{hyperref}
\usepackage{color}
\usepackage{enumerate}
\usepackage{tikz,pgfplots}
\usetikzlibrary{arrows}
\usepackage{comment}
\usepackage{lscape}
\usepackage{graphicx}
\usepackage{tabularx}
\usepackage{caption}
\usepackage{diagbox}
\usepackage{subcaption}
\pgfplotsset{compat=1.11}
\usepgfplotslibrary{fillbetween}
\usetikzlibrary{intersections}
\pgfdeclarelayer{bg}
\pgfdeclarelayer{ft}
\pgfsetlayers{bg,main,ft}

\numberwithin{equation}{section}
 
\newcommand{\ProofEnde}{\hfill {$\square$}
\renewcommand{\thesubfigure}{\alph{subfigure}}

}

\def\d{{\rm d}} 
\def\e{\varepsilon} 
 
\newfam\Bbbfam 
\font\tenBbb=msbm10 
\font\sevenBbb=msbm7 
\font\fiveBbb=msbm5 
\textfont\Bbbfam=\tenBbb 
\scriptfont\Bbbfam=\sevenBbb 
\scriptscriptfont\Bbbfam=\fiveBbb

\def\2{\mathbf 2}

\newcommand{\R}     {\mathbb{R}} 
 
\newcommand{\N}     {\mathbb{N}} 
\renewcommand{\P}   {\mathbb{P}} 
 
\newcommand{\E}     {\mathbb{E}}

\newcommand{\smfrac}[2]{\textstyle{\frac {#1}{#2}}}

\def\1{{\mathchoice {1\mskip-4mu\mathrm l}      
{1\mskip-4mu\mathrm l} 
{1\mskip-4.5mu\mathrm l} {1\mskip-5mu\mathrm l}}} 
 
\def\comment#1{} 
\newtheoremstyle{thm}{2ex}{2ex}{\itshape\rmfamily}{} 
{\bfseries\rmfamily}{}{1.7ex}{} 
 
\newtheoremstyle{rem}{1.3ex}{1.3ex}{\rmfamily}{} 
{\itshape\rmfamily}{}{1.5ex}{}

\renewcommand{\theequation}{\thesection.\arabic{equation}} 
 
\newtheorem{theorem}{Theorem}[section] 
\newtheorem{lemma}[theorem]{Lemma} 
\newtheorem{prop}[theorem] {Proposition} 
\newtheorem{cor}[theorem]  {Corollary}

\theoremstyle{definition}

\newtheorem{remark}[theorem]{Remark}

\renewcommand{\d}{{\rm d}} 
 
\newcommand{\eps}{\varepsilon}

\newcommand{\Tr}{{\operatorname {Tr}\,}}

\newcommand{\Bcal}  {{\mathcal B}}
 
\newcommand{\Dcal}   {{\mathcal D }}

\newcommand{\Lcal}   {{\mathcal L }} 
 
\newcommand{\Ncal}   {{\mathcal N }}

\newcommand\numberthis{\addtocounter{equation}{1}\tag{\theequation}}
\renewcommand{\e}   {{\operatorname e }}

\definecolor{Red}{rgb}{1,0,0}

\begin{document} 
 
\title[A stochastic population model with dormancy and transfer]
{The interplay of dormancy and  transfer in bacterial populations:
Invasion, fixation and coexistence regimes}
\author[Jochen Blath and András Tóbiás]{}
\maketitle
\thispagestyle{empty}
\vspace{-0.5cm}

\centerline{\sc Jochen Blath and András Tóbiás{\footnote{TU Berlin, Straße des 17. Juni 136, 10623 Berlin, {\tt blath@math.tu-berlin.de, tobias@math.tu-berlin.de}}}}
\renewcommand{\thefootnote}{}
\vspace{0.5cm}
\centerline{\textit{TU Berlin}}

\bigskip

\centerline{\small(\version)} 
\vspace{.5cm} 
 
\begin{quote} 
{\small {\bf Abstract:}} 
In this paper we investigate the interplay between two fundamental mechanisms of microbial population dynamics and evolution, namely {\em dormancy} and {\em horizontal gene transfer}. The corresponding traits come in many guises and are ubiquitous in microbial communities, affecting their dynamics in important ways. Recently, they have each moved (separately) into the focus of stochastic individual-based modelling
 (Billiard et\ al\ 2016, 2018; Champagnat, Méléard and Tran, 2021;  Blath and Tóbiás 2020). 
Here, we investigate their combined effects in a unified model. Indeed, we consider the (idealized) scenario of two sub-populations, respectively carrying `trait 1' and `trait 2', where trait 1 individuals are able to switch (under competitive pressure) into a dormant state, and trait 2 individuals are able to execute horizontal gene transfer, which in our case means that they can turn trait 1 individuals into trait 2 ones, at a rate depending on the density of individuals.

In the large-population limit, we examine the fate of (i) a single trait 2 individual (called `mutant') arriving in a trait 1 resident population living in equilibrium, and (ii) a trait 1 individual (`mutant') arriving in a trait 2 resident population. We analyse the invasion dynamics in all cases where the resident population is individually fit and the behaviour of the mutant population is initially non-critical. This leads to the identification of parameter regimes for the invasion and fixation of the new trait, stable coexistence of the two traits, and `founder control' (where the initial resident always dominates, irrespective of its trait). 

One of our key findings is that horizontal transfer can lead to stable coexistence even if trait 2 is unfit on its own. In the case of founder control, the limiting dynamical system also exhibits a coexistence equilibrium, which, however, is unstable, and with overwhelming probability none of the mutant sub-populations is able to invade. In all cases, we observe the classical (up to three) phases of invasion dynamics à la Champagnat (2006).
\end{quote}

\bigskip\noindent 
{\it MSC 2010.} 60J85, 92D25.

\medskip\noindent
{\it Keywords and phrases.} Dormancy, horizontal gene transfer, stochastic population model, large population limit, coexistence, founder control.

\setcounter{tocdepth}{3}


\setcounter{section}{0}
\begin{comment}{
This is not visible.}
\end{comment}

\section{Introduction }\label{sec-introductionHGT}

{\bf Motivation.} An essential feature in bacterial population biology is the ability of individuals to engage in {\em horizontal} (or {\em lateral}) {\em gene transfer} (HGT), which we understand here in an abstract sense as the exchange of heritable genetic material resp.\ traits between contemporary microorganisms (eg.\ as the result of bacterial conjugation, \cite{LT46}). The ability to execute horizontal transfer is widespread among microbial populations, appears in several different forms, and has a crucial impact on the ecological, evolutionary, and also pathogenic (eg.\ regarding antibiotic resistance) properties of microbial communities, cf.\  eg.\ \cite{OLG00}, \cite{KW12}, \cite{GB14}. Horizontal gene transfer has been introduced into a number of classical population genetics models, see eg.~\cite{BP14} and the references therein. Recently, stochastic individual-based modelling of the effects of HGT (in idealized models) has gained substantial interest in the mathematical biology community. Important contributions in this direction are for example the population dynamics models and results of Billiard et al.\ \cite{BCFMT16,BCFMT18} 
and the analysis in an extended adaptive dynamics framework, including a rather novel mutation regime, by Champagnat, Méléard, and Tran \cite{CMT19}. In these papers, the authors study 
the dynamics of two or more populations competing for resources and horizontally exchanging traits, in case of the last paper additionally considering mutations and the resulting changes in relative fitnesses. It is shown that HGT can drastically change the dynamics and evolution of the underlying populations: Depending on the frequency of mutations, it can lead to coexistence of traits that would not coexist without HGT, evolutionary cyclic behaviour, temporary extinction of sub-populations, and evolutionary suicide.


Another biologically essential feature of microbial populations is {\em dormancy}. 
This trait, which is again ubiquitous in microbial communities (see eg.\ \cite{RC87}, \cite{W04}, \cite{SD73}), allows individuals to switch reversibly into a metabolically (almost) inactive state (eg.\ by forming an endospore or cyst, or by phenotypically switching into a `persister cell') to withstand unfavourable conditions \cite{B04}, \cite{KL+05}, thus forming a seed bank comprised of dormant individuals. Traditionally, dormancy and seed banking have been modelled in the context of seeds and plants, where they can be understood as a bet-hedging strategy in fluctuating environments \cite{Co66}, \cite{B84}, \cite{E85}.  However, as is the case with horizontal transfer, dormancy also influences  the evolutionary, ecological, and pathogenic character of bacterial populations in complex ways (see eg.\ \cite{LJ11, SL18, L10, LdHWB20} for overviews). While in the last two decades there has been significant progress in the mathematical understanding of the role of dormancy and seed banks in  population genetics (cf.\ eg.\ \cite{KKL01, T11, BEGKW15, BGKW16, BGKW20}), there seem to be relatively few rigorous results in a stochastic 
 population dynamic framework involving direct competition. 
In our previous paper \cite{BT19}, we introduced a stochastic individual-based model 
for the invasion analysis of a dormancy trait in a resident population lacking this trait. We showed that under suitable assumptions on the model parameters, with asymptotically positive probability, a newly arriving (`mutant') individual having the dormancy trait is able to invade the resident population, even if it has a significantly lower reproductive rate than the residents, and in case of a successful invasion, it will reach fixation (i.e., frequency one in the population) with high probability. This shows that dormancy can come with a selective advantage even in the presence of a rather severe reproductive trade-off, providing some conceptual explanation for the ubiquity of dormancy despite the high energy costs for the maintenance of the underlying trait. 


In the present paper we aim to gain some conceptual understanding of the {combined effects} of both horizontal transfer and dormancy (at least in a basic scenario based on the previous works mentioned above). To this end we analyse competitive individual-based models with two traits, one of them being one-sided HGT involving a donor and a recipient (similar to \cite[Section 7]{BCFMT18} or \cite{CMT19}), and the other one 
being dormancy (as in \cite{BT19}), and their large population limits. For concreteness, let us say that trait 1 allows for dormancy, where switching into a dormant state is triggered by competitive pressure, while switching back (resuscitation) happens spontaneously. Trait 2 individuals have no dormant state, but they may be able to impose their trait  on trait 1 individuals via HGT (thus effectively turning trait 1 individuals into trait 2 individuals). See Section \ref{sec-modeldef} below for details of the model.
We investigate both the situation where at time zero there is a trait 1 resident population close to its equilibrium population size and a single newly arriving individual (mutant) of trait 2, and vice versa, with the roles of 1 and 2 exchanged.

\medskip

 The {\bf main questions} that we wish to answer in this paper are a) under which conditions an invasion of the mutant trait is possible (in either scenario) in a way that the probability of invasion stays  positive in the large-population limit, and b)  whether a successful mutant trait will reach fixation (ie.\ frequency one in the population) and make the resident trait go extinct, or whether both will coexist (again, for either scenario). Finally, we are also interested in c) how long it takes until an invading mutant reaches fixation, goes extinct, or reaches a coexistence equilibrium.

\medskip

{\bf Outline of the paper.} 
In Section~\ref{sec-main} we introduce our stochastic individual-based model and derive its large-population limit, given by a dynamical system. This scaling limit approximation works well as soon/as long as all population sizes (of the residents and the mutants) are of the same order, that is, after successful invasion, and before potential extinction/fixation of either population. 
%
We then describe the main properties of the dynamical system (existence of a coexistence equilibrium, stability etc.).
To understand the invasion and fixation properites for the original stochastic model, we  introduce couplings to approximating branching processes for the early invasion / late extinction phase, when either the mutant or the resident populations are small.  At the end of the section we present our main results, explain the heuristics behind them, and explicitly answer our main questions.

In Section~\ref{sec-discussion12} we conclude the main body of the text with a discussion of the rigorous methods required to obtain our results, the different parameter regimes, and the boundary cases where either HGT or dormancy is absent from the model. The analysis of the boundary cases clarifies which properties are due to HGT (and present also in absence of dormancy), which ones are due to dormancy (and observable also without HGT), and which emerge only if both evolutionary forces are simultaneously included in the model. 

The full proofs of our results can be found in the Appendix: In Section~\ref{sec-preliminaries}, we verify our deterministic preliminary results about the limiting dynamical system. Section~\ref{sec-generator}  contains the description of the infinitesimal generator of our population process. In Section~\ref{sec-HGTproofs}, we analyse the invasion of trait 2 against trait 1, whereas in Section~\ref{sec-dormantproofs}  we investigate the one of trait 1 against trait 2. The main methods and ideas of the proof are summarized in Section~\ref{sec-methods}. 

\section{Model definition, heuristics, and main results}\label{sec-main}

\subsection{The stochastic population model and its limiting dynamical system}
\label{sec-modeldef}

\tikzstyle{1a}=[circle,draw=blue!50,fill=blue!20,thick,minimum size=5mm]
\tikzstyle{1d}=[circle,draw=black!50,fill=black!20,thick,minimum size=5mm]
\tikzstyle{2}=[circle,draw=green!50,fill=green!20,thick,minimum size=5mm]
\tikzstyle{D}=[rectangle,draw=black!50,fill=white!20,thick,minimum size=5mm]


In our stochastic model, the population evolves as a continuous time Markov chain 
$$\mathbf N_t := (N_{1a,t},N_{1d,t},N_{2,t}), \quad t \geq 0,$$ 
%
on the state space $(\N_0)^3$.  The three quantities on the right-hand sides describe the number of individuals at time $t$ of trait $1$ that are active (index $1a$), those of trait 1 that are dormant (index $1d$), and finally those of trait 2. 
We further introduce a parameter $K>0$ called the \emph{carrying capacity} of the population. 
To obtain a large population limit, we will later also consider the rescaled population process 
$$
\mathbf N_t^K := \frac 1K \mathbf N_t, \quad t \geq 0.
$$
Besides $K$, our process involves the following parameters: 
\begin{itemize}
\item $\lambda_1>0$, the birth rate of active trait 1 individuals,
\item $\lambda_2>0$, the birth rate of trait 2 individuals,
\item $\mu >0$, the death rate of both active trait 1 and trait 2 individuals,
\item $C>0$, the competition strength,
\item $p \in (0,1)$, the dormancy initiation probability under competitive pressure,
\item $\kappa \geq 0$, with $\kappa \mu$ the death rate of dormant individuals,
\item $\sigma >0,$ the resuscitation rate of dormant individuals,
\item $\tau >0$, the horizontal transfer rate.
\end{itemize}
The dynamics of our process $(\mathbf N_t)$ is then given by the following transitions with corresponding rates, see Figure \ref{fig:hgt} for a visualization (for the interested reader, we also display the infinitesimal Markov generator of the process $(\mathbf N_t )_{t\geq 0}$ in Appendix~\ref{sec-generator}).
\begin{itemize}
    \item Active trait 1 individuals give birth to another such individual at rate $\lambda_1$ and die at rate $\mu$. 
    \item Trait 2 individuals give birth to another such individual at rate $\lambda_2$ and die at rate $\mu$. 
    
 
    \item For any ordered pair of active individuals, competitive events happen at rate $C/K$, affecting the first individual from the pair. If the affected individual is of trait 2, it dies immediately. If it is of trait 1, then it dies with probability $1-p$, but with probability $p$ it switches into dormancy.

    \item Dormant trait 1 individuals become active at rate $\sigma$ and die at rate $\kappa\mu$.
    
    
    \item Trait 2 individuals transfer their trait to active trait 1 individuals, i.e., turn them into trait 2, at rate $\tau/K$.


 \end{itemize}

Note that the death rate $\mu$ is the same for all active individuals. This assumption simplifies the notation, and it can be made without loss of generality since the crucial parameters are the net reproduction rates of the traits $1a$ and $2$ (here, $\lambda_1-\mu$ and $\lambda_2-\mu$), as we will see below. We typically assume that $\kappa \in [0,1]$, ie.\ the death rate is lower for dormant individuals than for active ones.

Further, note that the carrying capacity $K$ enters both the competition and the horizontal transfer events. This is a classical scaling that leads to non-trivial limits for $(\mathbf N_t^K)$ as $K \to \infty$. In particular, in the limit, the competition events will lead to a {\em logistic growth term} with pre-factor $C$, and the transfer events will lead to {\em density-dependent transfer} of rate $\tau$ (an alternative modelling option would be {\em frequency-dependent transfer}, see eg.~\cite[Remark 3.1]{BCFMT18} for a discussion of both concepts), as can be seen below. 
Regarding the concrete mechanism of competition, we chose a micro-model in which the individual death rates are increased due to crowding. An alternative way to incorporate competition could have been via decreased birth rates, which, if suitably chosen, would lead to the same scaling limit. However, increasing death rates blends in nicely with competition-induced switches into dormancy and is consistent with earlier work in \cite{BT19}. Our resuscitation mechanism for dormancy leads to an immigration term from the dormant pool of individuals into the active trait 1 population, and our HGT term is realized by an increased death rate of trait 1 individuals in the presence of trait 2, combined with a simultaneously increased birth rate for trait 2. 


To formally identify the large population limit, assume that we have the convergence of the initial population sizes of the rescaled process $(\mathbf N_t^K)$  at time 0, i.e.
$$
\frac 1K \mathbf N_0 = \mathbf N_0^K \overset{K \to \infty}{\longrightarrow}  {\bf n}(0)=(n_{1a}(0),n_{1d}(0),n_{2}(0)) \, \, \in \, [0, \infty)^3.
$$
Then, by classic convergence theory for Markov processes (see \cite[Theorem 11.2.1, p.~456]{EK}), as $K \to \infty$, for $T>0$, we obtain the weak convergence (uniformly on $[0,T])$ of our population process
$$
(\mathbf N_t^K)_{t \in [0,T]} \Rightarrow (\mathbf n(t))_{t \geq 0}=((n_{1a}(t),n_{1d}(t),n_{2}(t)))_{t \in [0,T]}
$$
to the solution $(\mathbf n(t))_{t \in [0,T]}$ of the deterministic dynamical system given by 
\begin{equation}\label{3dimHGT}
\begin{aligned}
\frac{\d n_{1a}(t)}{\d t} & = n_{1a}(t)\big( \lambda_1-\mu-C (n_{1a}(t)+n_{2}(t))-\tau n_{2}(t) \big)  + \sigma n_{1d}(t), \\
\frac{\d n_{1d}(t)}{\d t} & = p C n_{1a}(t)(n_{1a}(t)+n_{2}(t)) - (\kappa\mu+\sigma) n_{1d}(t), \\
\frac{\d n_{2}(t)}{\d t} & = n_{2}(t) \big( \lambda_2-\mu-C (n_{1a}(t)+n_{2}(t)) + \tau n_{1a}(t) \big),
\end{aligned}
\end{equation}
with the same initial condition ${\bf n}(0)$ and parameters as above. 

The result in particular shows that if all populations are comparably large at the same initial time, their future dynamics can be well-approximated by their law-of-large number scaling limits given by the above system. The Gaussian fluctuations around this scaling limit are described by \cite[Theorem 11.2.3]{EK}.

It is easy to see that the positive orthant $[0, \infty)^3$ is positively invariant under this system. 
In Section~\ref{sec-coex2d} we will provide a short discussion of the relation of \eqref{3dimHGT} to competitive Lotka--Volterra equations in the dormancy-free case $p=0$ and some related remarks corresponding to the case $p>0$. 


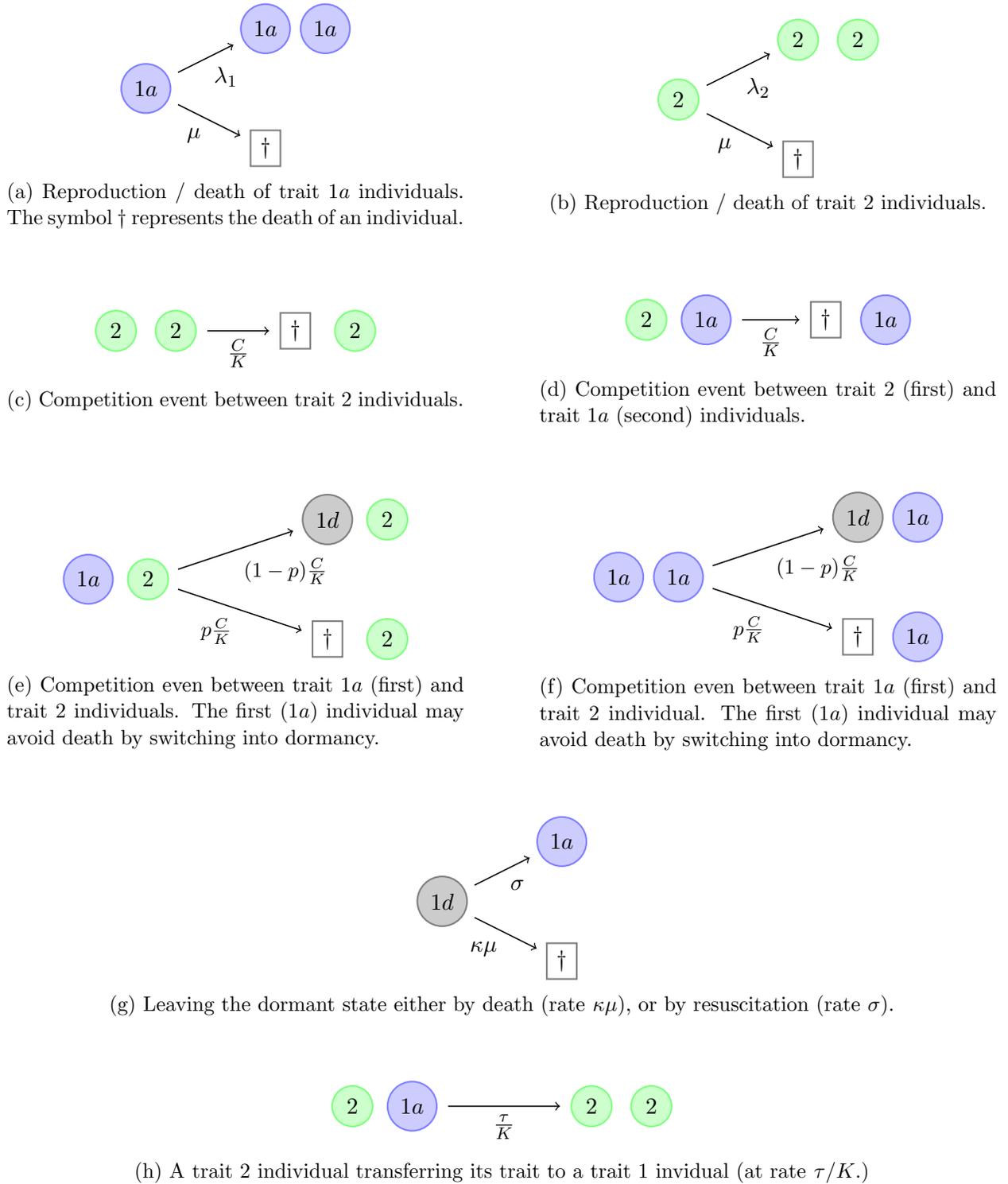
\begin{figure}
\vspace{1cm}
\begin{subfigure}{.45\textwidth}
  \centering
   \begin{tikzpicture}[node distance=2cm, semithick, shorten >=5pt, shorten <=5pt]
  \node  [1a] 	(1a_rep)	 {$1a$};
  \node  [1a] 	(1a_child_1)	[right of=1a_rep, yshift=1cm]	 {$1a$};    
  \node  [1a] 	(1a_child_2) 	[right of=1a_child_1, xshift=-1cm] 	{$1a$};
  \node  [D] 		(death) 	[below of=1a_child_1] 	{$\dagger$}; 
  	 \draw[->] (1a_rep) to node[auto, swap] {$\lambda_1$} (1a_child_1);
 	 \draw[->] (1a_rep) to node[auto, swap] {$\mu$} (death);  
 \end{tikzpicture}
  \caption{Reproduction / death of trait $1a$ individuals. The symbol $\dagger$ represents the death of an individual.}
  \vspace{1cm}
\end{subfigure}
\hspace{1cm}
\begin{subfigure}{.45\textwidth}
  \centering
  \begin{tikzpicture}[node distance=2cm, semithick, shorten >=5pt, shorten <=5pt]
  \node  [2] 	(2_rep)	 {$2$};
  \node  [2] 	(2_child_1)	[right of=2_rep, yshift=1cm]	 {$2$};    
  \node  [2] 	(2_child_2) 	[right of=2_child_1, xshift=-1cm] 	{$2$};
  \node  [D] 		(death) 	[below of=2_child_1] 	{$\dagger$}; 
  	 \draw[->] (2_rep) to node[auto, swap] {$\lambda_2$} (2_child_1);
 	 \draw[->] (2_rep) to node[auto, swap] {$\mu$} (death);  
 \end{tikzpicture}
  \caption{Reproduction / death of trait  2 individuals.}
  \vspace{1cm}
\end{subfigure}
 \begin{subfigure}{.45\textwidth}
   \centering
     \begin{tikzpicture}[node distance=2cm, semithick, shorten >=5pt, shorten <=5pt]
  \node  [2] 	(2_comp_1)	 {$2$};  
  \node  [2] 	(2_comp_2)	 [right of=2_comp_1, xshift=-1cm]{$2$};
  \node  [D] 	(2_child_1)	[right of=2_comp_2,]	 {$\dagger$};    
  \node  [2] 		(death) 	[right of=2_child_1, xshift=-1cm] 	{$2$}; 
  	 \draw[->] (2_comp_2) to node[auto, swap] {$\frac CK$} (2_child_1);
 \end{tikzpicture}
  \caption{Competition event between trait 2 individuals.}
   \vspace{1cm}
      \end{subfigure} 
 \hspace{1cm}     
 \begin{subfigure}{.45\textwidth}
  \centering
  \begin{tikzpicture}[node distance=2cm, semithick, shorten >=5pt, shorten <=5pt]
  \node  [2] 	(2_comp_1)	 {$2$};  
  \node  [1a] 	(2_comp_2)	 [right of=2_comp_1, xshift=-1cm]{$1a$};
  \node  [D] 	(2_child_1)	[right of=2_comp_2,]	 {$\dagger$};    
  \node  [1a] 		(death) 	[right of=2_child_1, xshift=-1cm] 	{$1a$}; 
  	 \draw[->] (2_comp_2) to node[auto, swap] {$\frac CK$} (2_child_1);
 \end{tikzpicture}
     \caption{Competition event between trait 2 (first) and trait $1a$ (second) individuals.}
      \vspace{1cm}
      \end{subfigure}
\vspace{1cm}
\begin{subfigure}{.45\textwidth}
 \centering
     \begin{tikzpicture}[node distance=2cm, ,semithick, shorten >=5pt, shorten <=5pt]
  \node  [1a] 	(2_comp)	 {$1a$};  
  \node  [2] (1a_comp)	 [right of=2_comp_1, xshift=-1cm]{$2$};
  \node  [1d] 	(2_child)	[right of=2_comp_2,yshift=1cm, xshift=1cm]	 {$1d$};    
  \node  [2] 		(dormant) 	[right of=2_child, xshift=-1cm] 	{$2$}; 
  	 \draw[->] (1a_comp) to node[auto, swap] {\small $(1-p)\frac{C}{K}$} (2_child);
  \node  [D] 	(2_child_2)	[below of=2_child]	 {$\dagger$};    
  \node  [2] 		(dormant) 	[right of=2_child_2, xshift=-1cm] 	{$2$}; 
  	 \draw[->] (1a_comp) to node[auto, swap] {\small $p\frac{C}{K}$} (2_child_2);	 
 \end{tikzpicture}
  \caption{Competition even between trait $1a$ (first) and trait 2 individuals. The first ($1a$) individual may avoid death by switching into dormancy.}
   \end{subfigure} 
\hspace{1cm}
 \begin{subfigure}{.45\textwidth}
 \centering
  \begin{tikzpicture}[node distance=2cm, semithick, shorten >=5pt, shorten <=5pt]
  \node  [1a] (2_comp)	 {$1a$};  
  \node  [1a] (1a_comp)	 [right of=2_comp_1, xshift=-1cm]{$1a$};
  \node  [1d] 	(2_child)	[right of=2_comp_2,yshift=1cm, xshift=1cm]	 {$1d$};    
  \node  [1a] 		(dormant) 	[right of=2_child, xshift=-1cm] 	{$1a$}; 
  	 \draw[->] (1a_comp) to node[auto, swap] {\small $(1-p)\frac{C}{K}$} (2_child);
  \node  [D] 	(2_child_2)	[below of=2_child]	 {$\dagger$};    
  \node  [1a] 		(dormant) 	[right of=2_child_2, xshift=-1cm] 	{$1a$}; 
  	 \draw[->] (1a_comp) to node[auto, swap] {\small $p\frac{C}{K}$} (2_child_2);	 
 \end{tikzpicture} 
 \caption{Competition even between trait $1a$ (first) and trait 2 individual. The first ($1a$) individual may avoid death by switching into dormancy.}
      \end{subfigure}
 \begin{subfigure}{.9\textwidth}
  \centering
  \begin{tikzpicture}[node distance=2cm, semithick, shorten >=5pt, shorten <=5pt]
  \node  [1d] (dormant)	 {$1d$};  
  \node  [1a] (active)	 [right of=dormant, yshift=1cm]{$1a$};
  \node  [D] 	(death)	[below of=active]	 {$\dagger$};    
  	 \draw[->] (dormant) to node[auto, swap] {$\sigma$} (active);
  	 \draw[->] (dormant) to node[auto, swap] {$\kappa \mu$} (death);	 
 \end{tikzpicture} 
 \caption{Leaving the dormant state either by death (rate $\kappa \mu$), or by resuscitation (rate $\sigma$).}   
  \vspace{1cm}
\end{subfigure}
    \begin{subfigure}{.9\textwidth}
  \centering
   \begin{tikzpicture}[node distance=2cm, semithick, shorten >=5pt, shorten <=5pt]
  \node  [2] 	(2_comp)	 {$2$};  
  \node  [1a] (1a_comp)	 [right of=2_comp_1, xshift=-1cm]{$1a$};
  \node  [2] 	(2_child)	[right of=1a_comp, xshift=1cm]	 {$2$};    
  \node  [2] 		(transfer) 	[right of=2_child, xshift=-1cm] 	{$2$}; 
  	 \draw[->] (1a_comp) to node[auto, swap] {$\frac{\tau}{K}$} (2_child);
 \end{tikzpicture}
 \caption{A trait 2 individual transferring its trait  to a trait 1 invidual (at rate $\tau/K$.)}
\end{subfigure}
\caption{Graphical representation of the transitions of the stochastic model from Section \ref{sec-modeldef}. 
}\label{fig:hgt}
\end{figure}



Before we analyse the behaviour of the full system, we briefly collect some preliminary results in special cases. Indeed, in absence of trait 1, the equilibrium population size of trait 2 is given as
\[ \bar n_{2} = \frac{(\lambda_2-\mu) \vee 0}{C}, \numberthis\label{n2bardef} \]
, since it is easy to see that started from an initial condition $\mathbf n(0)=(0,0,a)$, $a>0$, $\mathbf n(t)$ converges to $(0,0,\bar n_{2})$ as $t \to \infty$. (Throughout the paper, $\vee$ stands for maximum and $\wedge$ for minimum.)
Further, in the absence of trait 2, the active equilibrium population size of trait 1 is
\[ \bar n_{1a} = \frac{(\lambda_1-\mu) \vee 0}{C} \frac{\kappa\mu+\sigma}{\kappa\mu+(1-p)\sigma} \numberthis\label{n1abardef} \]
and the dormant equilibrium population size of (in other words, the size of the seed-bank for) trait 1 equals
\[ \bar n_{1d}=\frac{((\lambda_1-\mu)^2 \vee 0) p (\kappa\mu+\sigma)}{C (\kappa\mu+(1-p)\sigma)^2}, \numberthis\label{n1dbardef} \]
where $(\bar n_{1a},\bar n_{1d},0)$ is given as the limit of $\mathbf n(t)$ as $t\to\infty$ if $\mathbf n(0)=(a,b,0)$ for some $a,b>0$; see \cite{BT19} for further details. 

Throughout the paper we will usually assume that trait 1 is \emph{individually fit}, i.e., $\lambda_1>\mu$, which ensures that $\bar n_{1a}$ and $\bar n_{1d}$ are both positive. 
However, the condition that trait 2 is individually fit, i.e., $\lambda_2>\mu$, will not always be required; in some cases, an individually unfit trait 2 will still be able to coexist with an individually fit trait 1, as we will see later. Yet, the same is not true with the two traits interchanged, and the reason for this asymmetry is that trait 2 benefits from HGT, creating an inflow and thus decreasing the trait 1 population. 




\subsection{Equilibria of the dynamical system and their stability} 
\subsubsection{Existence of a coexistence equilibrium}\label{sec-coexistence}
Regarding the existence and number of coexistence equilibria of our system \eqref{3dimHGT} where both dormancy and HGT are present, we have the following lemma, the proof of which is simply checking the zeros of the system. 
\begin{lemma}\label{lemma-coexistence}
The system \eqref{3dimHGT} has a coordinate-wise positive coexistence equilibrium if 
\[ \lambda_2-\mu >\frac{Cp\sigma}{\tau(\kappa\mu+\sigma)} (\lambda_2-\mu)+\frac{C}{\tau}(\lambda_1-\lambda_2) > \lambda_1-\mu \numberthis\label{mutantfitter2ineq} \]
or
\[  \lambda_2-\mu <\frac{Cp\sigma}{\tau(\kappa\mu+\sigma)} (\lambda_2-\mu)+\frac{C}{\tau}(\lambda_1-\lambda_2) < \lambda_1-\mu. \numberthis\label{mutantlessfit2ineq} \]
Given that a coexistence equilibrium exists, it is also unique and given as $(n_{1a},n_{1d},n_2)$ where
\begin{equation}\numberthis\label{coexeq}
    \begin{aligned}
    n_{1a} & = \frac{C(\kappa\mu+\sigma)(\lambda_2-\lambda_1)+Cp\sigma(\mu-\lambda_2)+(\kappa\mu+\sigma)\tau(\lambda_2-\mu)}{\tau(Cp\sigma-(\kappa\mu+\sigma)\tau)}, \\
    n_{1d} & = \frac{pC(\lambda_2-\lambda_1) \big(C(\kappa\mu+\sigma)(\lambda_2-\lambda_1)+Cp\sigma(\mu-\lambda_2)+(\kappa\mu+\sigma)\tau(\lambda_2-\mu)\big)}{\tau(Cp\sigma-(\kappa\mu+\sigma)\tau)^2}, \\
    n_{2} & = \frac{C(\kappa\mu+\sigma)(\lambda_2-\lambda_1)+Cp\sigma(\mu-\lambda_2)+(\kappa\mu+\sigma)\tau(\lambda_1-\mu)}{-\tau(Cp\sigma-(\kappa\mu+\sigma)\tau)}.
    \end{aligned}
\end{equation}
\end{lemma}
The proof of Lemma~\ref{lemma-coexistence} (as well as proofs of all other assertions in this section) can be found in the Appendix, see Section~\ref{sec-coexistenceproof}. There, we will point out that condition~\eqref{mutantfitter2ineq} implies that $\tau<\frac{Cp\sigma}{\kappa\mu+\sigma}$, ie.\ the transfer intensity $\tau$ is relatively small, whereas condition~\eqref{mutantlessfit2ineq} implies that $\tau>\frac{Cp\sigma}{\kappa\mu+\sigma}$, ie.\ HGT is relatively strong, see Lemma~\ref{lemma-lessconditions}.

\begin{cor}\label{cor-1fit}
Condition \eqref{mutantfitter2ineq} implies $\lambda_1>\mu$, and the same holds for condition \eqref{mutantlessfit2ineq}. 
\end{cor}

The corollary shows that coexistence with an unfit  HGT-receiving trait is not possible. However, note that coexistence for $\lambda_1>\mu \geq \lambda_2$ is not excluded: condition~\eqref{mutantlessfit2ineq} can still be satisfied with $\lambda_2 \leq \mu$ if the other parameters are suitably chosen; a necessary condition for this is $\tau \geq C$.

\subsubsection{Stability of equilibria}\label{sec-stabilityHGT}
After understanding under what conditions the coexistence equilibrium exists, let us analyse the stability of the equilibria of the dynamical system \eqref{3dimHGT} in case $\lambda_1>\mu$.  Our main result regarding this is the following proposition, the proof of which can be found in Section~\ref{sec-stabilityproof}. 
\begin{prop}\label{prop-stability3d}
Assume that $\max \{ \lambda_1,\lambda_2\}>\mu$ and that none of the inequalities of \eqref{mutantfitter2ineq} (or equivalently, \eqref{mutantlessfit2ineq}) is satisfied with an equality. Then, the stability of the equilibria of the dynamical system \eqref{3dimHGT} is given according to Table~\ref{table-stability}.
\end{prop}

%
%
%

\begin{table}[]
    \centering
    \begin{tabularx}{1\textwidth}{|X|X|X|X|c|} \hline
    \diagbox[width=9.4em]
    {Cond.\\satisfied}{Equilibrium}
      & $(0,0,0)$ & $(0,0,\smfrac{\lambda_2-\mu}{C})$ & $(\bar n_{1a},\bar n_{1d},0)$ & $(n_{1a},n_{1d},n_2)$ \\ \hline
     \eqref{mutantfitter2ineq}& unstable & asp.~stable & asp.~stable & {unstable} \\ \hline
   Second inequality of \eqref{mutantfitter2ineq} and first inequality of \eqref{mutantlessfit2ineq}  & unstable & $\lambda_2>\mu \Rightarrow$ unstable, $\lambda_2\leq\mu \Rightarrow$ equals $(0,0,0)$ & asp.~stable  & $\nexists$ \\ \hline
     \eqref{mutantlessfit2ineq} & unstable &  $\lambda_2>\mu \Rightarrow$ unstable, $\lambda_2\leq\mu \Rightarrow$ equals $(0,0,0)$ & unstable & --- \\ \hline
      First inequality of \eqref{mutantfitter2ineq} and second inequality of \eqref{mutantlessfit2ineq}  & unstable & asp.~stable & $\lambda_1>\mu \Rightarrow$ unstable, $\lambda_1 \leq \mu \Rightarrow$ equals $(0,0,0)$  & $\nexists$ \\ \hline
     \end{tabularx} \smallskip
    \caption{Stability of coordinatewise nonnegative equilibria of the dynamical system \eqref{3dimHGT}.
    Here, `---' means `no assertion is provided' and `asp.' stands for `asymptotically'.}
    \label{table-stability}
\end{table}

\begin{remark}\label{remark-coexstability}
Proposition~\ref{prop-stability3d} provides no assertion about the stability of the coexistence equilibrium $(n_{1a},n_{1d},n_2)$ under the assumption~\eqref{mutantlessfit2ineq}, although we expect that this equilibrium is asymptotically stable. Our main convergence results about the dynamical system~\eqref{3dimHGT},  Proposition~\ref{prop-convergenceHGTguys} and Lemma~\ref{lemma-1invades2LV} in the Appendix, imply a certain global stability property of this equilibrium, namely that the domain of attraction of the equilibrium is not contained in any proper affine linear subspace of $\R^3$. Now, the methods used for our analysis of the same equilibrium under the assumption~\eqref{mutantfitter2ineq} imply that under \eqref{mutantlessfit2ineq}, the Jacobi matrix corresponding to~\eqref{3dimHGT} at this equilibrium has negative trace and negative determinant, see Remark~\ref{remark-whatisleftfromstability} for further details. Hence, the number of its eigenvalues with negative real parts must be equal to 1 or 3. If we knew that the equilibrium $(n_{1a},n_{1d},n_2)$ was hyperbolic (ie.\ that all eigenvalues have nonzero real parts), then the aforementioned global stability properties of the equilibrium would be sufficient to exclude that there is only one eigenvalue with negative real parts, and hence we could conclude that the equilibrium is asymptotically stable. However, we are not able to exclude the critical (non-hyperbolic) case that the Jacobi matrix has a conjugate pair of purely imaginary eigenvalues, in which case the stability of the equilibrium with respect to the system~\eqref{3dimHGT} need not be the same as for the linearized variant of the system around $(n_{1a},n_{1d},n_2)$. Fortunately, the missing assertion about the local stability of $(n_{1a},n_{1d},n_2)$ under condition~\eqref{mutantlessfit2ineq} is not needed for our analysis of invasion dynamics.
\end{remark}

\begin{remark}\label{remark-stability}
The four cases included in Table~\ref{table-stability} suggest the following behaviour for the stochastic individual-based model in the limit $K \to \infty$, which we will later make rigorous.
We tacitly disregard the cases where one of the inequalities in \eqref{mutantfitter2ineq} (or equivalently, \eqref{mutantlessfit2ineq}) is satisfied with an equality. 

\begin{enumerate}[(I)]
    \item If \eqref{mutantfitter2ineq} holds, we have unstable coexistence of traits 1 and 2. This corresponds to \emph{founder control}, that is, for large population sizes, the probability that a mutant of one trait can invade a population of the other trait living in equilibrium tends to zero. 
    Note that we borrowed the term `founder control' from spatial ecology, where it describes a situation such that whichever population first establishes at a given location, wins locally, i.e.\ can prevent the invasion of other populations at this location, see e.g.~\cite{V15}. 
    \item If the second inequality of \eqref{mutantfitter2ineq} and first inequality of \eqref{mutantlessfit2ineq}  hold, trait 2 is relatively unfit (or even absolutely unfit, in the sense that $\lambda_2 \leq \mu$) and the HGT rate is too small to compensate this disadvantage. Thus, trait 2 cannot invade trait 1. But the opposite invasion is possible, where trait 1 will even reach fixation (the population size will approach the equilibrium  $(\bar n_{1a},\bar n_{1d},0)$ after rescaling by $K$) and trait 2 will become extinct. 
    Note that one can only speak about trait 1 invading trait 2 if $\lambda_2 >\mu$ because otherwise trait 2 cannot survive on its own.
    \item If \eqref{mutantlessfit2ineq} holds, we have stable coexistence; even though we are not able to show that $(n_{1a},n_{1d},n_{2})$ is asymptotically stable, in Sections~\ref{sec-phase2HGTguys} and \ref{sec-phase2dormants} we will see that it exhibits strong non-local attracting properties. Both traits are able to invade the other one, but none of them is able to reach fixation and make the other one go extinct; instead, starting with a resident population of equilibrium size from one trait and a single mutant of the other one, with asymptotically positive probability the rescaled population size process will converge to the coexistence equilibrium $(n_{1a},n_{1d},n_2)$. For $C<\tau$, this regime contains cases when $\lambda_2 \leq \mu$, ie.\ where trait 2 can only survive thanks to HGT from trait 1, as we will discuss in Section~\ref{sec-regimes}.
    \item In the last case, when the first inequality of \eqref{mutantfitter2ineq} and the second inequality of \eqref{mutantlessfit2ineq} hold, HGT is strong or trait 2 is individually fitter than in the previous case. The result of this is that a mutant of trait 1 cannot invade a population of trait 2 of equilibrium size. In contrast, a mutant of trait 2 can invade a trait 1 population and even reach fixation (the population size will reach the equilibrium $(0,0,\bar n_2)$ after rescaling by $K$) with asymptotically positive probability, driving trait 1 into extinction. 
\end{enumerate}
\end{remark}

One may wonder whether in the second and in the third case, the one-type equilibria that are claimed to be asymptotically stable actually have a positive coordinate. This is indeed the case, which is shown by the following lemma.
\begin{lemma}\label{lemma-somebodyfit}
The following assertions hold.
\begin{enumerate}[(i)]
\item\label{if1fit2fit} If the first inequality of \eqref{mutantfitter2ineq} and the second inequality of \eqref{mutantlessfit2ineq} hold with $\lambda_1>\mu$, then $\lambda_2>\mu$.
\item If the first inequality of \eqref{mutantlessfit2ineq} and the second one of \eqref{mutantfitter2ineq} hold with $\lambda_2>\mu$, then $\lambda_1>\mu$.
\end{enumerate}
\end{lemma}
Lemma~\ref{lemma-somebodyfit} together with Corollary~\ref{cor-1fit} and Lemma~\ref{lemma-lessconditions} will guarantee that an individually unfit trait cannot drive another individually fit trait into extinction, hence there is no \emph{evolutionary suicide} (cf.~\cite[Section 7]{BCFMT18} or \cite[Section 1]{CMT19}) in our model. 
The proof of Lemma~\ref{lemma-somebodyfit} can also be found in Section~\ref{sec-stabilityproof}. 

To provide a better intuition about the four different parameter regimes described in Proposition~\ref{prop-stability3d}, 
in Section~\ref{sec-regimes} (and in particular in Figure~\ref{figure-regimes}) we depict the areas of fixation of one trait, stable coexistence of both traits, and founder control (with a corresponding unstable coexistence equilibrium in the dynamical system~\eqref{3dimHGT}) for different choices of the parameters. Note that the precise meaning of this figure (eg.\ the meaning of `stable coexistence') relies on our main results presented in Section~\ref{sec-results} below. While condition~\eqref{mutantfitter2ineq} implies $\lambda_1>\lambda_2$ and the \eqref{mutantlessfit2ineq} implies $\lambda_2>\lambda_1$, in the other two cases both $\lambda_1 >\lambda_2$ and $\lambda_2\geq\lambda_1$ are possible. 

We will discuss the special cases $\tau>p=0$ (HGT without dormancy) and $p>\tau=0$ (dormancy without HGT) in Sections~\ref{sec-coex2d} and \ref{sec-onlydormancybackwards}, since they are also covered by the proof techniques of the present paper. In a nutshell, the case $\tau>p=0$ turns out to be rather similar to the case $\tau,p>0$, with the main difference being that in the case (I) above (founder control), the underlying (two-dimensional) dynamical system does not exhibit an unstable coexistence equilibrium. In contrast, in the case $p>\tau=0$ we have `competitive exclusion' in the sense that `invasion implies fixation'. That is,  regimes (II) and (IV) cover all possibilities with $\max \{ \lambda_1,\lambda_2\}>\mu$, apart from boundary cases. Our previous paper~\cite{BT19} focused on the question of under what conditions trait 1 can invade trait 2 (with $\lambda_2,\lambda_1>\mu$) and reach fixation. The parameter regime where we found this to be possible is the analogue of  regime (II), and we pointed out that it also contains cases where $\lambda_1<\lambda_2$, ie.\ the trait benefiting from competition-induced dormancy may even win against a faster reproducing trait  lacking the capability of dormancy. 

\subsection{Heuristics of the mechanisms of invasion and fixation for the stochastic system}\label{sec-phase13} 

We now return to the analysis of the stochastic model that we introduced in Section~\ref{sec-modeldef}.  Our results are based on probabilistic individual-based invasion theory à  la  Champagnat (cf.~\cite{C06}), exhibiting the following now classical phases:
\begin{enumerate}[(1)]
    \item growth or extinction of the (initially small) mutant population to a `macroscopic size' of order $K$, while the (large)  resident population stays close to equilibrium,
   \item in case the probability of extinction during phase 1) does not tend to one as $K \to \infty$: a mean-field (a.k.a.~Lotka--Volterra) phase where both traits have macroscopic population sizes and the system can be approximated by a deterministic system of ODEs (in our case by~\eqref{3dimHGT}), converging either to a coexistence equilibrium (exhibiting a certain global stability property) or to the monomorphic equilibrium of the mutants,
   \item if there is no stable coexistence: previously macroscopic resident population becomes small and faces extinction, while the mutant population stays close to equilibrium.
\end{enumerate}

To be more concrete, assume that $K$ is large and that for some $i,j \in \{1,2\}$, $i \neq j$, there exists a monomorphic resident population of trait $i$ close to its equilibrium population size after rescaling by $K$, where we assume that trait $i$ is individually fit (i.e., $\lambda_i>\mu$), and a single mutant of trait $j$. The question is now whether an invasion of mutants is possible, i.e., whether the mutants are able to reach a population size of order $K$ with a probability that stays bounded away from zero as $K \to \infty$. We expect the following to hold. As long as the mutant population is small but not (yet) extinct, it essentially only feels competition from the active resident population. Thus, since during the time until its extinction or growth to size of order $K$ the resident population stays close to its equilibrium with high probability, the mutant population can be approximated by a linear branching process (which is two-type if $j=1$), similarly eg.~to the setting of \cite{C+16, C+19, BT19}. This phase ends either with the extinction of the mutant population (after $O(1)$ time units) or with the mutant population reaching a size of order of $K$, which takes of order $\log K$ time units.
Given that our three-dimensional population size process reaches a state where all sub-populations are of order $K$, rescaled by $K$ it can be approximated by the limiting dynamical system \eqref{3dimHGT}. Now, our aim is to show that the population process hits a state that converges to an initial condition of the dynamical system such that starting from this point, the solution of the system converges to the corresponding single stable equilibrium. This second, `mean-field' or `Lotka--Volterra' phase takes $O(1)$ time. Finally, in case this stable equilibrium is not the coexistence one but the monomorphic one of trait $j$, fixation occurs: after an additional time of order $\log K$ the trait $i$ population becomes extinct with high probability. Here, the mutant population stays close to equilibrium, and the resident one can be approximated by a branching process, which is now subcritical. In fact, this branching process is the same as the one corresponding to the mutants in the first phase of invasion in case the roles of $i$ and $j$ are interchanged. This third phase does not take place if we have stable coexistence.  

Depending on the stability of equilibria, we expect the following invasion dynamics:
\begin{enumerate}
    \item In case $(n_{1a},n_{1d},n_2)$ exists but it is unstable, invasion of trait $j$ is impossible. The mutant population becomes extinct in $O(1)$ time with high probability and the resident population stays close to equilibrium. This corresponds to founder control.
    \item In case $(n_{1a},n_{1d},n_2)$ exists and is asymptotically stable, invasion of trait $j$ is possible, and the population size process reaches a state with high probability from which the solution of the dynamical system converges to this stable equilibrium. In particular, with high probability, fixation does not occur.
    \item In case there is no coexistence equilibrium and the one-trait equilibrium of the mutant trait $j$ is asymptotically stable, invasion and fixation of trait $j$ is possible.
    \item In case there is no coexistence equilibrium and the one-trait equilibrium of the resident trait $i$ is asymptotically stable, invasion of trait $j$ is impossible. The mutant population goes extinct in $O(1)$ time with high probability, while the resident population stays close to equilibrium. 
\end{enumerate}

In particular, we expect
that the three (respectively two) phases of invasion are consistent in the sense that the possibility of invasion and fixation is in one-to-one correspondence with the local stability landscape of the dynamical system \eqref{3dimHGT}.
In other words, we will see that, apart from the critical case, a trait is able to invade if its invasion fitness is positive and it is able to reach fixation if and only if the invasion fitness of the other trait is negative (see Remark~\ref{remark-invfitness} for the notion and properties of invasion fitness). 


We can show that the behaviour of the approximating branching processes is indeed in correspondence with the one of the local stability portrait of \eqref{3dimHGT}. If trait $1$ is resident with $\lambda_1>\mu$ and trait $2$ is mutant, this branching process is defined as a continuous time Markov chain $(\widehat N_{2}(t))_{t \geq 0}$ with state space $\N_0$, assuming that $0$ is an absorbing state, and with the following rates for $n \in \N$:
\begin{itemize}
    \item $n \to n+1$ at rate $(\lambda_2+\tau \bar n_{1a})n$ (reproduction of mutants from birth or HGT to the residents),
    \item $n \to n-1$ at rate $(\mu+C \bar n_{1a})n$ (death of mutants by age or competition with the residents).
\end{itemize}
This branching process is supercritical, i.e., its birth $(n \to n+1)$ rate is higher than its death $(n \to n-1)$ rate,  if and only if
\[ \widehat \lambda:=\lambda_2-\mu-(C-\tau)\frac{\lambda_1-\mu}{C} \frac{\kappa\mu+\sigma}{\kappa\mu+(1-p)\sigma}>0, \numberthis\label{lambdahatdef} \]
where we used the definition of $\bar n_{1a}$ from \eqref{n1abardef}.
Using elementary operations, we conclude that this condition is equivalent to
\[ \lambda_1-\mu > \frac{Cp\sigma(\lambda_2-\mu)}{\tau(\kappa\mu+\sigma)} + \frac{C}{\tau} (\lambda_1-\lambda_2), \numberthis\label{secondinmutantlessfit2ineq} \]
which is the second inequality in \eqref{mutantlessfit2ineq}. Further, in case trait $2$ is initially resident and $1$ is mutant, and the only stable equilibrium of the system \eqref{3dimHGT} is $(\bar n_{1a}, \bar n_{1d}, 0)$, then the same branching process approximates the trait $2$ population in the last, third phase of invasion. Hence, this branching process is subcritical in case the strict reverse inequality of \eqref{secondinmutantlessfit2ineq} holds, i.e.,
\[ \lambda_1-\mu < \frac{Cp\sigma(\lambda_2-\mu)}{\tau(\kappa\mu+\sigma)} + \frac{C}{\tau} (\lambda_1-\lambda_2). \numberthis\label{reversesecondin} \]
The extinction probability of the branching process $(\widehat N_2(t))_{t \geq 0}$ is defined as
\[ q_2 = \P\big( \exists t < \infty \colon \widehat N_{2}(t)=0 \big|\widehat N_{2}(0)=1 \big). \numberthis\label{q2def} \]
Then we have $q_2=1$ if the process is not supercritical. Else, $q_2$ equals the unique fixed point of the probability generating function of the branching process in $(0,1)$, that is, $q_2$ is the unique solution of the quadratic equation
\[  (\lambda_2+\tau \bar n_{1a}) s^2 + (-\lambda_2-(\tau+C) \bar n_{1a}) s + C \bar n_{1a}=0  \]
being different from 1. That is,
\[ q_2 = \frac{(\lambda_2+\tau \bar n_{1a}) - \sqrt{(\lambda_2+(\tau+C)\bar n_{1a})^2-4 C \bar n_{1a}(\lambda_2+\tau \bar n_{1a})}}{2(\lambda_2+\tau \bar n_{1a})}. \numberthis\label{q2} \]

On the other hand, if trait $2$ is resident and trait $1$ is mutant, the branching process approximating the mutants in the first phase is two-type; it is given as a continuous time Markov chain $((\widehat N_{1a}(t),\widehat N_{1d}(t)))_{t \geq 0}$ with state space $\N_0 \times \N_0$, assuming that $(0,0)$ is an absorbing state, and with the following rates for $(n,m) \in \N_0 \times \N_0\setminus \{ (0,0) \}$:
\begin{itemize}
    \item $(n,m) \to (n+1,m)$ at rate $\lambda_1 n$ (reproduction of active mutants from birth),
    \item $(n,m) \to (n-1,m)$ at rate $(\mu +C(1-p) \bar n_2 + \tau \bar n_2)n$ (death of active mutants by age, competition with the residents or HGT from the residents),
    \item $(n,m) \to (n-1,m+1)$ at rate $(Cp\bar n_2)n$ (competition-induced switching of active mutants to dormancy),
    \item $(n,m) \to (n,m-1)$ at rate $\kappa\mu m$ (natural death of dormant mutants),
    \item $(n,m) \to (n+1,m-1)$ at rate $\sigma m$ (resuscitation of dormant mutants).
\end{itemize}
Using standard results of the theory of multitype branching processes \cite[Section 7.2]{AN72}, this process is supercritical, i.e., its probability of going extinct is less than one, if and only if the following \emph{mean matrix} has a positive eigenvalue:
\[ J = \begin{pmatrix} \lambda_1-\mu-(\lambda_2-\mu)-\frac{\tau}{C}(\lambda_2-\mu) & p(\lambda_2-\mu) \\ \sigma & -\kappa\mu-\sigma
\end{pmatrix}. \numberthis\label{JdefHGT} \]
The determinant of this matrix is
\[ \det J = -\big(\lambda_1-\lambda_2-\frac{\tau}{C} (\lambda_2-\mu) \big)(\kappa\mu+\sigma)-p\sigma(\lambda_2-\mu). \]
This is negative if and only if
\[ \frac{C}{\tau}(\lambda_1-\lambda_2) + \frac{Cp\sigma(\lambda_2-\mu)}{\tau(\kappa\mu+\sigma)} > \lambda_2-\mu, \numberthis\label{firstinmutantlessfit2ineq} \]
which is the first inequality in \eqref{mutantlessfit2ineq}. Under this condition, the matrix clearly has a positive eigenvalue. On the other hand, if the determinant is not negative, then there is no positive eigenvalue. Indeed, if the determinant was nonnegative but there was a positive eigenvalue, then the trace of the matrix would also be nonnegative, in other words, we would have
\[ \lambda_1-\lambda_2-\frac{\tau}{c}(\lambda_2-\mu) \geq \kappa\mu+\sigma. \]
But this would imply that $\det J \leq -(\kappa\mu+\sigma)^2 - p\sigma(\lambda_2-\mu)<0$, which would contradict the assumption that $\det J \geq 0$. We conclude that the branching process is supercritical if and only if \eqref{firstinmutantlessfit2ineq} is satisfied. Moreover, in case trait $1$ is initially resident and $2$ is mutant, and the only stable equilibrium of the system \eqref{3dimHGT} is $(0,0,\bar n_2)$, then the same two-type branching process approximates the trait $1$ population in the last, third phase of invasion. Hence, this branching process is subcritical in case the strict reverse inequality of \eqref{firstinmutantlessfit2ineq} holds, that is,
\[ \frac{C}{\tau}(\lambda_1-\lambda_2) + \frac{Cp\sigma(\lambda_2-\mu)}{\tau(\kappa\mu+\sigma)} < \lambda_2-\mu. \numberthis\label{reversefirstin} \]
Note that the largest eigenvalue of $J$ is given as
\[ 
\begin{aligned}
\widetilde \lambda& = \frac{1}{2C} \Big( -\widetilde \alpha  + \sqrt{\widetilde \alpha^2-4(-C^2(\kappa\mu+\sigma)(\lambda_1-\lambda_2)-C^2p\sigma(\lambda_2-\mu)+C(\kappa\mu+\sigma)\tau(\lambda_2-\mu)) } \Big)
\end{aligned} \numberthis\label{lambdaHGT} \]
where
\[ \widetilde \alpha= C(\kappa\mu+\sigma)-C(\lambda_1-\lambda_2)+(\lambda_2-\mu)\tau. \]
Further, we identify the extinction probabilities of the branching process $((\widehat N_{1a}(t),\widehat N_{1d}(t)))_{t \geq 0}$ as follows. We define
\[ q_1 = \P\big( \exists t < \infty \colon \widehat N_{1a}(t)+\widehat N_{1d}(t)=0 \big|  (\widehat N_{1a}(0),\widehat N_{1d}(0))=(1,0) \big). \numberthis\label{q1def} \]
Thanks to \cite[Section 7]{AN72}, if the process is not supercritical, then $q_1=1$. Else, $q_1$ is the first coordinate of the unique solution of the system of equations
\[ \numberthis\label{extinctionequation}
\begin{aligned}
 \lambda_1 (s_a^2-s_a) + p(\lambda_2-\mu) (s_d-s_a) + (\mu+C\big((1-p)+\frac{\tau}{C}\big)(\lambda_2-\mu))(1-s_a) &=0, \\
\sigma (s_a-s_d) + \kappa \mu(1-s_d) &=0,  \\
\end{aligned}
\]
in $[0,1]^2 \setminus \{ (1,1) \}$,
while the second coordinate of the same solution is the extinction probability given that the branching process is started from $(0,1)$.

\begin{remark}\label{remark-simplifyconditions}
It follows from the observations of this section that the eigenvalues $\widehat \lambda$ and $\widetilde \lambda$ have opposite signs, in particular, $\widehat \lambda = 0$ if and only if $\widetilde \lambda = 0$, i.e., if one of the branching processes $(\widehat N_{2}(t))_{t \geq 0}$ and $((\widehat N_{1a}(t),\widehat N_{1d}(t)))_{t \geq 0}$ is critical, then so is the other one. Hence, $\widetilde\lambda \neq 0$ implies that either \eqref{firstinmutantlessfit2ineq} or \eqref{reversefirstin} holds, further, either \eqref{secondinmutantlessfit2ineq} or \eqref{reversesecondin} is satisfied. We will use this throughout the paper in order to simplify the notation.
\end{remark}

\begin{remark}\label{remark-invfitness}
The \emph{invasion fitness} of a mutant individual in a one-trait resident population living in equilibrium is commonly defined as the initial growth rate of the mutant population. 
In case of a single-coordinate mutant population, the growth rate is to be understood as the difference between the initial birth rate and the initial death rate. From this we easily conclude that the invasion fitness of a mutant of trait 2 against a two-coordinate resident population of trait 1 close to the equilibrium population size $(\bar n_{1a},\bar n_{1d})$ is
\[ S(2;1)=\lambda_2-\mu-C \bar n_{1a} + \tau \bar n_{1a} = \lambda_2-\mu-(C-\tau)(\lambda_1-\mu)\frac{\kappa\mu+\sigma}{\kappa\mu+(1-p)\sigma}=\widehat\lambda,
\numberthis\label{invfitnessdormancy} \]
which is the birth rate plus the initial HGT rate minus the death rate minus the rate of death by competition for the approximating branching process $(\widehat N_{2}(t))_{t \geq 0}$.
In case the mutant population is given by a multi-type linear branching process, it follows from \cite[Section 7]{AN72} that the mean growth rate of the population is the Lyapunov exponent of the corresponding mean matrix. This equals
\[ S(1;2)=\widetilde \lambda \]
for the branching process $((\widehat N_{1a}(t),\widehat N_{1d}(t)))_{t \geq 0}$ approximating the trait 1 mutant population shortly after time 0 in case of a trait 2 resident population close to its equilibrium. Hence, we can rightfully call the expression $S(1;2)$ the corresponding invasion fitness. Note that in this case, the total mutant population size process $(\widehat N_{1a}(t)+\widehat N_{1d}(t))_{t \geq 0}$ is typically not Markovian.

We see that apart from the critical case $\widehat \lambda=\widetilde \lambda=0$, for $i,j \in \{1,2\}$, $i \neq j$, a mutant of trait $i$ can invade a resident population of trait $j$ living in equilibrium if and only if $S(i;j)>0$, and in this case, fixation of trait $i$ is possible if and only if $S(j;i)<0$, else, there is stable coexistence between the two traits. This aligns with the observations of~\cite[Section 3]{B19} about invasion fitnesses of one-coordinate populations.
\end{remark}

Summarizing, we see a consistency between the stability of equilibria of \eqref{3dimHGT} and the properties of the approximating branching processes as explained above. How precisely these branching processes approximate our population process will be explained during the proofs of our main results, which will be carried out in Appendix~\ref{sec-HGTproofs} and Appendix~\ref{sec-dormantproofs}.

\subsection{Statement of our main results}\label{sec-results}
In this section we provide mathematically rigorous answers to our main questions.
Recall that trait 1 has a dormant state, whereas 2 has none, but it can transfer trait 1 individuals into trait 2 ones via HGT. 
Making the heuristics of Section~\ref{sec-phase13} precise, we present the following main results, which describe the fate of a mutant of trait 2 in a population of trait 1 living in equilibrium, respectively (in case trait 2 is fit) the fate of an active mutant of trait 1 in a population of trait 2 in equilibrium. (Given this, the case when initially there is a dormant mutant instead of an active one can be handled analogously to \cite[Section 3.3]{BT19}.)  We assume throughout that $\lambda_1>\mu$.

For $\beta>0$ define the `fixation sets'
\[ S_\beta^2 = \{ 0 \} \times \{ 0 \} \times [\bar n_2-\beta,\bar n_2+\beta ], \numberthis\label{Sbeta2def} \]
\[ S_\beta^1 =  [ \bar n_{1a}-\beta, \bar n_{1_a} + \beta ] \times [\bar n_{1d}-\beta, \bar n_{1d} + \beta] \times \{ 0 \}. \numberthis\label{Sbeta1def} \]
Further, we define the `coexistence set'
\[ S_\beta^{\mathrm{co}} = [n_{1a}-\beta, n_{1a}+\beta] \times [n_{1d}-\beta,n_{1d}+\beta]\times [n_2-\beta,n_2+\beta] \numberthis\label{Sbetacodef} \]
in case $(n_{1a},n_{1d},n_2)$ exists as a coordinatewise positive equilibrium.
These are the closures of open $\ell^\infty$-neighbourhoods of the equilibria $(0,0,\bar n_2)$, $(\bar n_{1a},\bar n_{1d},0)$ respectively $(n_{1a},n_{1d},n_2)$. Next, we define a stopping time at which $\mathbf N_t^K$ reaches these sets, respectively:
\[ T_{S_\beta^i}  = \inf \{ t > 0 \colon \mathbf N_t^K \in S_\beta^i \}, \qquad i \in \{ 1,2,\mathrm{co} \}, \numberthis\label{TSbetaidef} \]
where we put $T_{S_{\beta}^{\rm co}}=\infty$ if $(n_{1a},n_{1d},n_2)$ does not exist as a coordinatewise positive equilibrium. Moreover, we define
the first time when the rescaled mutant population size reaches a threshold $x \geq 0$ (from below or above):
\[ T_{x}^2 := \inf \{ t > 0 \colon KN_{2,t}^K = \lfloor x K \rfloor \} \numberthis\label{Tlevel1} \]
(in case trait 2 is the mutant) respectively
\[ T_{x}^1 := \inf \{ t > 0 \colon N_{1a,t}^K+N_{1d,t}^K = \lfloor x K \rfloor \}. \numberthis\label{Tlevel2} \]

Let us now describe the invasion and fixation dynamics of a single mutant of trait 2 in a population of trait 1 living close to equilibrium. We disregard the critical cases when in one of the inequalities in \eqref{mutantlessfit2ineq} the corresponding equality holds. First, we derive the limiting probability of a successful invasion. In case this is positive, we also provide the asymptotic time until fixation respectively reaching the coexistence equilibrium.

\begin{theorem}\label{thm-invasionof2}
Assume that $\lambda_1>\mu$,
$(N^K_{1a}(0),N^K_{1d}(0)) \underset{K \to \infty}{\to} (\bar n_{1a},\bar n_{1d})$
and $N^K_{2}(0)=\smfrac{1}{K}$.
\begin{enumerate}
    \item Assume that \eqref{reversesecondin} holds. Then, for all $x>0$ we have 
    \[ \lim_{K \to \infty} \mathbb P \Big( T_x^2 < T_0^2 \Big) = 0. \]
    \item Assume that \eqref{secondinmutantlessfit2ineq} holds. Then, \smallskip
    \begin{itemize}
        \item if \eqref{firstinmutantlessfit2ineq} holds, then for all sufficiently small $\beta>0$ we have
\[ \lim_{K \to \infty} \mathbb P \Big( T_{S_\beta^{\mathrm{co}}} < T_0^2 \wedge T_{S_\beta^2} \Big) = 1- q_2, \numberthis\label{coexprob2} \]
and on the event $\{ T_{S_\beta^{\mathrm{co}}} < T_0^2 \wedge T_{S_\beta^{2}} 
\}$,
\[ \lim_{K \to \infty} \frac{T_{S_\beta^{\mathrm{co}}}}{\log K} = \frac{1}{\widehat \lambda}\numberthis\label{coexistenceof2} \]
in probability,
\item whereas if \eqref{reversefirstin} holds, then for all sufficiently small $\beta>0$, we have
\[ \lim_{K \to \infty} \mathbb P \Big( T_{S_\beta^{2}} < T_0^2 \Big) = 1- q_2, \numberthis\label{invasionprob2}\]
and on the event $\{ T_{S_\beta^{2}} < T_0^2 \}$,
\[ \lim_{K \to \infty} \frac{T_{S_\beta^{2}}}{\log K} = \frac{1}{\widehat \lambda}-\frac{1}{\widetilde\lambda} \numberthis\label{invasionof2} \]
in probability.
    \end{itemize}  
\end{enumerate}
\end{theorem}
Recall that if \eqref{firstinmutantlessfit2ineq} holds, then $\widehat\lambda>0$, whereas if \eqref{reversesecondin} holds, then $\widetilde\lambda<0$. These guarantee that the right-hand sides of \eqref{coexistenceof2} and \eqref{invasionof2} are positive under the corresponding conditions of Theorem~\ref{thm-invasionof2}. The difference of the form of~\eqref{coexistenceof2} and \eqref{invasionof2} is that the second (Lotka--Volterra) phase of invasion takes $O(1)$ time units but afterwards, if there is no stable coexistence, the extinction of the resident population (and its approximating subcritical branching process) will take an order of $\log K$ time.

Next, we show that in case of an unsuccessful invasion, with high probability, the extinction takes a sub-logarithmic time (in particular, the extinction happens during the first phase of the invasion), and at the time of extinction the resident population is close to its equilibrium population size. In the following theorem, $\Vert \cdot \Vert$ denotes the Euclidean norm on $\R^3$, which could certainly also be replaced by any other fixed norm. 
\begin{theorem}\label{thm-failureof2}
Under the assumptions of Theorem~\ref{thm-invasionof2}, on the event $\{ T_0^2 < T_{S_\beta^{\mathrm{co}}} \wedge T_{S_\beta^2} \}$,
\[ \lim_{K \to \infty} \frac{T_0^2}{\log K} =0 \numberthis\label{extinctionof2} \]
and
\[ \mathds 1_{\{ T_0^2 < T_{S_\beta^{\mathrm{co}}} \wedge T_{S_\beta^2}  \}} \big\Vert \mathbf N^K_{T_0^2}-(\bar n_{1a},\bar n_{1d},0) \big\Vert \underset{K \to \infty}{\longrightarrow} 0, \numberthis\label{lastoftheoremof2} \]
both in probability.
\end{theorem}
Next, we derive similar results for the case when trait 2 is individually fit and trait 1 tries to invade it. The corresponding analogue of Theorem~\ref{thm-invasionof2} is given as follows.
\begin{theorem}\label{thm-invasionof1}
Assume that $\lambda_2>\mu$,
$(N^K_{1a}(0),N^K_{1d}(0)) =(\frac{1}{K},0)$,
and 
$N^K_{2}(0) \underset{K\to\infty}{\longrightarrow} \bar n_{2}$.
\begin{enumerate}
    \item Assume that \eqref{reversefirstin} holds. Then, for all $x>0$ we have 
    \[ \lim_{K \to \infty} \mathbb P \Big( T_x^1 < T_0^1 \Big) = 0. \]
    \item Assume that \eqref{firstinmutantlessfit2ineq} holds. Then, \smallskip
    \begin{itemize}
        \item in case \eqref{secondinmutantlessfit2ineq} holds, then for all sufficiently small $\beta>0$, we have
\[ \lim_{K \to \infty} \mathbb P \Big( T_{S_\beta^{\mathrm{co}}} < T_0^1 \wedge T_{S_\beta^1} \Big) = 1- q_1, \numberthis\label{coexprob1} \]
and on the event $\{ T_{S_\beta^{\mathrm{co}}} < T_0^1 \wedge T_{S_\beta^1} \}$
\[ \lim_{K \to \infty} \frac{T_{S_\beta^{\mathrm{co}}}}{\log K} = \frac{1}{\widetilde \lambda}\numberthis\label{coexistenceof1} \]
in probability,
\item whereas in case \eqref{reversesecondin} holds, then for all small enough $\beta>0$, we have
\[ \lim_{K \to \infty} \mathbb P \Big( T_{S_\beta^{1}} < T_0^1 \wedge T_{S_\beta^{\mathrm{co}}}\Big) = 1- q_1, \numberthis\label{invasionprob1} \]
and on the event $\{ T_{S_\beta^{1}} < T_0^1 \wedge T_{S_\beta^{\mathrm{co}}}  \}$,
\[ \lim_{K \to \infty} \frac{ T_{S_\beta^{\mathrm{1}}}}{\log K} =\frac{1}{\widetilde\lambda}- \frac{1}{\widehat\lambda} \numberthis\label{invasionof1} \]
in probability.
    \end{itemize}  
\end{enumerate}
\end{theorem}
Finally, the analogue of Theorem~\ref{thm-failureof2} is the following.
\begin{theorem}\label{thm-failureof1}
Under the assumptions of Theorem~\ref{thm-invasionof1}, we have that on the event $\{ T_0^1 < T_{S_\beta^{\rm co}} \wedge T_{S_\beta^1} \}$, 
\[ \lim_{K \to \infty} \frac{T_0^1}{\log K} =0 \numberthis\label{extinctionof1} \]
and
\[ \mathds 1 \{ T_0^1 < T_{S_\beta^{\rm co}} \wedge T_{S_\beta^1}\} \big\Vert \mathbf N^K_{T_0^1}-(0,0,\bar n_{2}) \big\Vert \underset{K \to \infty}{\longrightarrow} 0, \numberthis\label{lastoftheoremof1} \]
both in probability.
\end{theorem}
Recall that under the condition \eqref{secondinmutantlessfit2ineq}, $\widetilde\lambda>0$, whereas if \eqref{reversefirstin} holds, then $\widehat\lambda<0$. Thanks to these facts, the right-hand sides of \eqref{coexistenceof1} and \eqref{invasionof1} are positive if the corresponding conditions of Theorem~\ref{thm-invasionof1} are satisfied.

The proof of Theorems~\ref{thm-invasionof2} and \ref{thm-failureof2} will be carried out in Appendix~\ref{sec-HGTproofs}. The proof of Theorems~\ref{thm-invasionof1} and~\ref{thm-failureof1} will be sketched in Appendix~\ref{sec-dormantproofs}. 
In Section~\ref{sec-methods} we provide a brief summary of the steps of our proofs that required some novel ideas, emphasizing that given the results of~\cite{C+19,BT19}, the main difficulty is to verify convergence of the dynamical system in the case when trait 2 is invading trait 1.

Regarding the interpretation of the different regimes of fixation, stable coexistence, and founder control, we gave an informal overview in Section~\ref{sec-stabilityHGT} in the context of the dynamical system~\eqref{3dimHGT}; see also Section~\ref{sec-regimes} for a visualization of these regimes for different choices of the parameters.
A key ingredient of the proof of the convergence of the dynamical system to the respective equilibrium is the following lemma, which will be essential in order to treat the second phase of invasion. Its
proof can be found in Appendix~\ref{sec-phase2HGTguys}.
\begin{lemma}\label{lemma-effectivecomp} The following assertions hold.
\begin{enumerate}[(i)]
\item\label{n2largeduetoHGT} Let $\lambda_1>\mu$. Then,~\eqref{secondinmutantlessfit2ineq} is equivalent to the condition $C\widetilde n_2 > (C-\tau) \bar n_{1a}$. 
\item\label{compeq} If $(n_{1a},n_{1d},n_2)$ exists as a coordinatewise positive equilibrium, then $C\widetilde n_2 = C(n_{1a}+n_2)-\tau n_{1a}$.
\item\label{coexcompineq} If~\eqref{firstinmutantlessfit2ineq} and~\eqref{secondinmutantlessfit2ineq} hold, then $C(n_{1a}+n_2)+\tau n_2>C(\lambda_1-\mu)$.
\item\label{trait2compineq} If~\eqref{reversefirstin} and \eqref{secondinmutantlessfit2ineq} hold with $\lambda_1>\mu$, then $(C+\tau)\bar n_2 > C(\lambda_1-\mu)$. 
\item\label{trait2coexcompineq} If~\eqref{firstinmutantlessfit2ineq} and~\eqref{secondinmutantlessfit2ineq} hold with $\lambda_1>\mu$, then $(C+\tau)\widetilde n_2 < C(\lambda_1-\mu)$. 
\end{enumerate}
\end{lemma}
Let us provide an interpretation of this lemma. Taking also the effect of HGT into account, at $t>0$, $C(n_{1a}(t)+n_2(t))+\tau n_{2}(t)$ is the `effective competitive pressure' felt by the active trait 1 population: the population is not only exposed to death or switching to dormancy due to competitive pressure, but trait 2 also removes individuals from the trait 1 population via HGT. In the same time, trait 2 feels an effective competitive pressure of $C(n_{1a}(t)+n_2(t))-\tau n_{1}(t)$: death by competition is at least partially compensated by turning trait 1 individuals into trait 2 ones via HGT. Assertion~\eqref{coexcompineq} tells that in case of stable coexistence, this effective competitive pressure is larger in the coexistence equilibrium $(n_{1a},n_{1d},n_2)$ than close to the one-type equilibrium $(\bar n_{1a},\bar n_{1d},0)$ of trait 1. Further, in view of \eqref{n2largeduetoHGT} and \eqref{compeq}, the (active) trait 2 population feels less competitive pressure close to the coexistence equilibrium $(n_{1a},n_{1d},n_2)$ than near $(\bar n_{1a},\bar n_{1d},0)$. These (together with additional arguments) will make it possible for trait 2 to decrease the trait 1 population until this coexistence equilibrium is not reached. On the other hand, under the assumptions of assertion~\eqref{trait2compineq}, there is no coexistence equilibrium, but in the one-trait equilibrium $(0,0,\bar n_2)$ of trait 2, the active trait 1  population feels more effective competitive pressure than in the vicinity of $(\bar n_{1a},\bar n_{1d},0)$. Further, thanks to \eqref{n2largeduetoHGT}, the trait 2 population feels less competitive pressure near $(0,0,\bar n_2)$ than close to $(\bar n_{1a},\bar n_{1d},0)$. This is the heuristic reason for a complete fixation of trait 2 and extinction of trait 1. Finally, the assertion~\eqref{trait2coexcompineq} can be interpreted as follows: in case~\eqref{secondinmutantlessfit2ineq} holds but there is no coexistence equilibrium, then trait 2 cannot make trait 1 go extinct, i.e., there is no fixation of trait 2 for $\lambda_2>\mu$ and no evolutionary suicide for $\lambda_2 \leq \mu$.

In the next remark, we briefly explain how our theorems answer our main questions posed in Section~\ref{sec-introductionHGT}.

\begin{remark}[Answers to the main questions]
\color{white} a \color{black} \\
{\em Invasion of trait 2 against trait 1.}
Assume that $\lambda_2>\mu$. Theorems \ref{thm-invasionof2}--\ref{thm-failureof2} explain that invasion of a trait 2 mutant against a trait 1 resident population (with $\lambda_1>\mu$) is successful with asymptotically positive probability as $K\to\infty$ under the condition~\eqref{secondinmutantlessfit2ineq} and with probability tending to zero under the condition~\eqref{reversesecondin}. If~\eqref{secondinmutantlessfit2ineq} holds, then with high probability, a successful invasion will be followed by fixation of trait 2 if~\eqref{reversefirstin} holds and by convergence to the coexistence equilibrium $(n_{1a},n_{1d},n_2)$ if~\eqref{firstinmutantlessfit2ineq} holds. The time until coexistence divided by $\log K$ converges to $1/\widehat\lambda$ in probability, whereas the time until fixation of trait 1 divided by $\log K$ tends to $1/\widehat\lambda-1/\widetilde \lambda$, where the additional term corresponds to the length of the additional third phase of the invasion where the formerly resident population goes extinct. In case of a failed invasion, the time until the extinction of a mutant is sub-logarithmic in $K$, and the resident population stays near equilibrium until this extinction with high probability. \\
{\em Invasion of trait 1 against trait 2.}
Theorems~\ref{thm-invasionof2}--\ref{thm-failureof2} tell that for $\lambda_1>\mu$, analogous result holds for the reverse invasion direction, where trait 1 tries to invade a trait 2 resident population (which only makes sense if $\lambda_2>\mu$). Here, the roles of traits 1 and 2, the ones of \eqref{reversefirstin} and \eqref{reversesecondin}, the ones of \eqref{firstinmutantlessfit2ineq} and \eqref{secondinmutantlessfit2ineq}, as well as the ones of $\widehat\lambda$ and $\widetilde\lambda$ have to be interchanged with each other.
In particular, if both \eqref{reversefirstin} and \eqref{reversesecondin} hold, then the resident reaches fixation and the mutant goes extinct in any case, ie.\ we have founder control.

\end{remark}

\section{Discussion of methods, parameter regimes, and special cases}
\label{sec-discussion12}
In this section we discuss different aspects of our model and the employed methods. Since we present the proofs of our main results in the Appendix, in Section~\ref{sec-methods} we provide a short overview of the techniques used in the proofs. In Section~\ref{sec-regimes} we visualize and interpret the different regimes of invasion and fixation. Finally, in the remainder of the section, we investigate special cases of our model. In Section~\ref{sec-coex2d} we treat the case when there is only HGT but no dormancy in the system, whereas in Section~\ref{sec-onlydormancybackwards} we analyse the case when there is only competition-induced dormancy but no HGT. In the latter two sections, we omit the proofs because they can be obtained as special cases of the ones in the Appendix.

\subsection{Methods}\label{sec-methods}
As already mentioned in Section~\ref{sec-phase13}, our approach is based on the individual-based invasion theory introduced in~\cite{C06}. 
On the one hand, during the analysis of the invasion (and, if applicable, fixation) of trait 1 in a trait 2 population, we are able to use multiple arguments of our recent paper \cite{BT19} where we investigated competition-induced dormancy without HGT.
The analysis of phases (1) and (if applicable) (3) is a relatively straightforward adaptation of the methods of the aforementioned paper, which rely on a number of proof techniques from the papers \cite{C+16,C+19}. Here, the main task is to take into account the additional effect of HGT and the arising coexistence equilibrium. 

On the other hand, handling the deterministic phase (2) of invasion requires novel ideas, and this is the most tedious part of the proof. Even determining the local stability of the equilibria of the limiting dynamical system is a nontrivial task, and in our proofs we also need a certain global stability, namely a convergence of the solution of the system to the one-trait equilibrium of mutants (in case of no coexistence) respectively to the coexistence equilibrium. As for the invasion of trait 1 (exhibiting dormancy) against trait 2 (benefiting from HGT), there are still some proof techniques from \cite{C+19} that can be extended to this case, based on some additional stochastic results using the Kesten--Stigum theorem (see Section~\ref{sec-phase2dormants}). For the invasion of trait 2 against trait 1, which has no HGT-free analogue in \cite{BT19}, we provide an entirely deterministic proof based on observations regarding the competitive pressure felt by the active trait 1 and trait 2 population, involving an approximation of the trait 1 population by a two-dimensional dynamical system.

\subsection{Fixation vs.\ coexistence landscapes in different parameter regimes}\label{sec-regimes}
In Figure~\ref{figure-regimes}, we present numerical examples for the parameter regimes corresponding to the different scenarios of invasion, fixation, and coexistence described in these theorems. We choose the same values of $C,\mu,\kappa$ for all simulations for simplicity, and we fix different values of $\tau,p>0$ in each of the examples. Given these, we determine the fate of the population depending on the birth rates $\lambda_1,\lambda_2$. We only consider the case $\lambda_1>\mu$
, whereas we allow for $\lambda_2\leq \mu$. The notation for the coloured regions is the following.
\begin{itemize}
    \item Yellow: founder control (with unstable coexistence).
%
    \begin{itemize}
    \item I.\ Founder control (with $\lambda_2>\lambda_1$).
    \end{itemize}
    \item Blue: fixation of trait 1.
    \begin{itemize}
    \item II.\ Fixation of trait 1 despite $\lambda_2>\lambda_1$.
    \item II'.\ Fixation of trait 1 with $\lambda_1>\lambda_2>\mu$.
    \item II''.\ Fixation of trait 1 with $\lambda_1>\mu>\lambda_2$.
    \end{itemize}
    \item Red: stable coexistence.
    \begin{itemize}
    \item III.\ Stable coexistence with $\lambda_1>\lambda_2>\mu$.
    \item III'.\ Stable coexistence with $\lambda_1>\mu>\lambda_2$. 
    \end{itemize}
    \item Green: fixation of trait 2.
    \begin{itemize}
     \item IV.\ Fixation of trait 2 with $\lambda_2>\lambda_1$.
    \item IV'.\ Fixation of trait 2 despite $\lambda_1>\lambda_2$.
    \end{itemize}
\end{itemize}
In each figure, 
\[ \lambda_2-\mu=\frac{Cp\sigma}{\tau(\kappa\mu+\sigma)}(\lambda_2-\mu)+\frac{C}{\tau}(\lambda_1-\lambda_2) \]
is the equation of the orange critical line, whereas the blue critical line is given as follows:
\[ \lambda_1-\mu=\frac{Cp\sigma}{\tau(\kappa\mu+\sigma)}(\lambda_2-\mu)+\frac{C}{\tau}(\lambda_1-\lambda_2). \]
On these lines themselves, the behaviour of our population process is unknown. 


\begin{figure}
\begin{subfigure}{0.45\textwidth}
\begin{tikzpicture}[scale=0.35,domain=1:7.5,xscale=2]
\draw (-0.2,0) -- (1,0) node[right] {};
\draw (1,-0.3) node[below] {1};
\draw (-0.3,1) node[left] {1};
\draw[name path=X] (1,0) -- (7.5,0) node[right] {};
\draw[->] (7.5,0) -- (8,0) node[right] {$\lambda_1$};
\draw[->] (0,-0.2) -- (0,16.8) node[above] {$\lambda_2$};
\draw[color=red, dashed] (-0.2,1) -- (1,1) node {}; 
\draw[color=red, dashed] (1,-0.2) -- (1,1) node {};
\draw[color=green!60!black, thick,name path=L1L2]    plot (\x,\x)             node[right] {$\lambda_2=\lambda_1$}; 
\draw[color=blue, thick, name path=BLUE]   plot (\x,{(-5+10*\x)/5})    node[right] {};
\draw[color=orange, thick, name path=ORANGE] plot (\x,{(9*\x-5)/4}) node[right] {}; 
\draw[color=red, dashed,name path=L2MU] plot(\x,{1}) node[right] {$\lambda_2=\mu$}; 
\draw[color=red, dashed] (1,1) -- (1,15.7) node {}; 
\draw (6.8,0.5) node {II''}; 
\draw (6.8,4) node {II'}; 
\draw (6.8,10) node {II}; 
\draw (7.2,14.2) node {I}; 
\draw (4,14) node {IV}; 
\draw[color=red] (1.3,16.3) node {$\lambda_1=\mu$}; 
\draw[color=red, name path=L1MU] (1,15.7) node {}; 
\tikzfillbetween[of=X and L1L2, on layer=ft]{blue, opacity=0.2};
\tikzfillbetween[of=ORANGE and BLUE, on layer=ft]{yellow, opacity=0.2};
\tikzfillbetween[of=BLUE and L1L2, on layer=ft]{blue, opacity=0.2};
\tikzfillbetween[of=ORANGE and L1MU, on layer=ft]{green, opacity=0.2};
\end{tikzpicture}
\caption{Case $\tau=0.1$, $\sigma=0.9$, $p=0.1$. HGT is weak compared to dormancy. Coexistence can only be unstable (and in that case, we have founder control (I), yellow region).}
\label{fig-weakHGT}
\end{subfigure}
\hspace{.5cm}
\begin{subfigure}{0.45\textwidth}
\vspace{8pt}
\begin{tikzpicture}[scale=0.65,domain=1:8]
\draw[name path=X] (1,0) -- (8,0) node[right] {};
\draw[->] (8,0) -- (8.5,0) node[right] {$\lambda_1$};
\draw (-0.2,0) -- (1,0) node {};
\draw[->] (0,-0.2) -- (0,8.5) node[above]{$\lambda_2$};
\draw[color=red, dashed] (-0.2,1) -- (1,1) node {}; 
\draw[color=red, dashed] (1,-0.2) -- (1,1) node {};
\draw[color=green!60!black, thick, name path=L1L2]    plot (\x,\x)         node[right] {$\lambda_2=\lambda_1$}; 
\draw[color=orange, thick,name path=ORANGE]   plot (\x,{(1-0.05/0.8+(1/0.8-1)*\x)/(-0.05/0.8+1/0.8)})    node[right] {};
\draw[color=blue, thick,name path=BLUE] plot (\x,{(1+0.05/0.8+(1/0.8)*\x)/(1+0.05/0.8+1/0.8)}) node[right] {}; 
\draw[color=red, dashed,name path=L2MU] plot(\x,{1}) node[right] {$\lambda_2=\mu$}; 
\draw[color=red, dashed,name path=L1MU] plot(1,\x) node[right] {}; 
\draw[color=red] (1,8.3) node {$\lambda_1=\mu$};
\draw (7,0.5) node {II''}; \draw (7,1.5) node {II'}; 
\draw (7.5,3.5) node {III}; 
\draw (7,6) node {IV'}; 
\draw (6,7) node {IV}; 
\draw (1,-0.3) node[below] {1};
\draw (-0.3,1) node[left] {1};
\tikzfillbetween[of=X and ORANGE, on layer=ft]{blue, opacity=0.2};
\tikzfillbetween[of=ORANGE and BLUE, on layer=ft]{red, opacity=0.2};
\tikzfillbetween[of=BLUE and L1L2, on layer=ft]{green, opacity=0.2};
\tikzfillbetween[of=L1MU and L1L2, on layer=ft]{green, opacity=0.2};
\end{tikzpicture}
\caption{Case $\sigma=1$, $\tau=0.8$, $p=0.05$. HGT is strong compared to dormancy but still weaker than competition ($\tau<C$). Coexistence can only be stable, and it occurs only for $\lambda_2>\mu$.}
\label{fig-strongHGT}
\end{subfigure}
\begin{subfigure}{0.65\textwidth}
\begin{center}
\begin{tikzpicture}[scale=0.8,domain=1:8]
\draw (-0.2,0) -- (1,0) node[right] {};
\draw (1,-0.3) node[below] {1};
\draw (-0.3,1) node[left] {1};
\draw[name path=X] (1,0) -- (5.7,0) node[right] {};
\draw[name path=X2] (5.7,0) -- (8,0) node[right] {};
\draw[->] (8,0) -- (8.5,0) node[right] {$\lambda_1$};
\draw[->] (0,-0.5) -- (0,8.5) node[above] {$\lambda_2$};
\draw[color=red, dashed] (-0.2,1) -- (1,1) node {}; 
\draw[color=red, dashed] (1,-0.2) -- (1,1) node {};
\draw[color=green!60!black, thick, name path = L1L2]    plot (\x,\x)             node[right] {$\lambda_2=\lambda_1$}; 
\draw[color=orange, thick, name path=ORANGE]   plot (\x,{(1-0.05/1.2+(1/1.2-1)*\x)/(-0.05/1.2+1/1.2)})    node[right] {};
\draw[color=blue, thick, name path=BLUE] plot (\x,{(1-0.05/1.2+(1/1.2)*\x)/(1-0.05/1.2+1/1.2)}) node[right] {}; 
\draw[color=red, dashed,name path=L2MU] plot(\x,{1}) node[right] {$\lambda_2=\mu$}; 
\draw[color=red, dashed, name path=L1MU] plot(1,\x) node[right] {}; 
\draw[color=red] (1,8.3) node {$\lambda_1=\mu$};
\draw (7.5,2.5) node {III}; \draw (7.5,0.5) node {III'}; \draw (7,5) node {IV'}; 
\draw (6,7) node {IV}; 
\draw (2,0.5) node {II''}; 
\tikzfillbetween[of=L2MU and BLUE, on layer=ft]{red, opacity=0.2};
\tikzfillbetween[of=L2MU and X2, on layer=ft]{red, opacity=0.2};
\tikzfillbetween[of=ORANGE and X, on layer=ft]{blue, opacity=0.2};
\tikzfillbetween[of=BLUE and L1L2, on layer=ft]{green, opacity=0.2};
\tikzfillbetween[of=L1MU and L1L2, on layer=ft]{green, opacity=0.2};
\end{tikzpicture}
\caption{Case $\sigma=1$, $\tau=1.2$, $p=0.05$: HGT is strong compared to dormancy and even stronger than competition ($\tau>C$). Coexistence can only be stable, and it occurs even for $\lambda_2\leq\mu$.}
\label{fig-verystrongHGT}
\end{center}
\end{subfigure}
\caption{Fixation vs.\ coexistence landscapes for $\lambda_1$ and $\lambda_2$, and all other parameters fixed.
We always put $C=\mu=1$ and $\kappa=0$.
}\label{figure-regimes}
\end{figure}
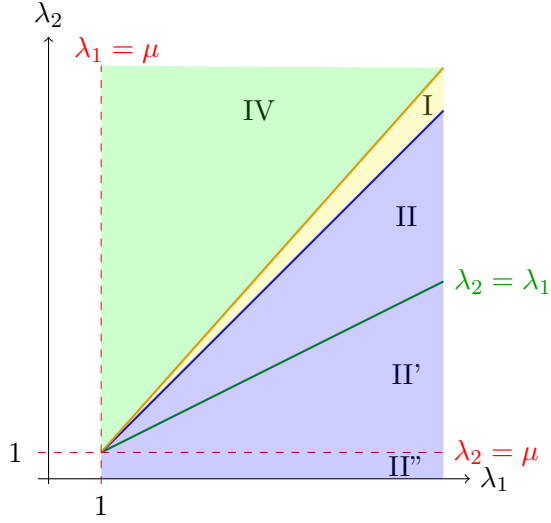
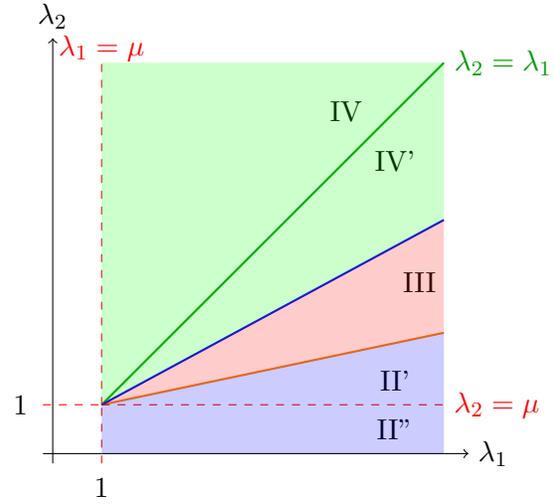
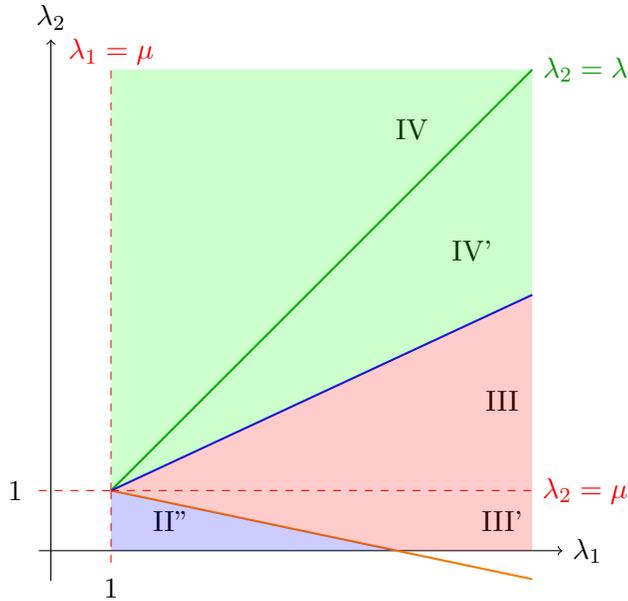

\begin{enumerate}
\item Figure~\ref{fig-weakHGT}: HGT is weak compared to dormancy, i.e., $\tau < \frac{Cp\sigma}{\kappa\mu+\sigma}$. In this case, only an unstable coexistence equilibrium (regime I) may exist (which then corresponds to founder control). The existence of regime II\ shows that trait 1 has an advantage caused by dormancy: It can even reach fixation in some cases when $\lambda_2>\lambda_1$, and further increase of $\lambda_2$ first leads to an unstable coexistence and only afterwards to the fixation of trait 2.
\item Figures~\ref{fig-strongHGT} and \ref{fig-verystrongHGT}: HGT is strong compared to dormancy, i.e., $\tau > \frac{Cp\sigma}{\kappa\mu+\sigma}$. 
In this case, trait 2 has an advantage due to HGT, which gives rise to regime IV'\ where trait 2 reaches fixation despite $\lambda_2<\lambda_1$, whereas regime II\ is now absent. Moreover, there always exists a stable coexistence regime (regime III) with $\mu<\lambda_2<\lambda_1$. Depending on whether $C<\tau$ or $C>\tau$ (i.e., whether the interaction between traits 1 and 2 is harmful or actually beneficial for trait 2), we distinguish two sub-cases.
\begin{itemize}
    \item Figure~\ref{fig-strongHGT}: The slope of the orange line is still positive, so that regime II'\ still exists. Hence, stable coexistence can only occur for $\lambda_2>\mu$ (in other words, regime III'\ is absent and regime II''\ covers the whole area where $\lambda_2<\mu$), but an arbitrarily small positive $\lambda_2-\mu$ can give rise to stable coexistence given that $\lambda_1>\mu$ is also close to zero from above. We observe that increasing $\lambda_1$ is always beneficial for trait 1, since the length of the vertical section of regime II'\ increases in $\lambda_1$. 
    \item Figure~\ref{fig-verystrongHGT}: The slope of the orange line is negative. Hence, stable coexistence can also occur for $\lambda_2 \leq \mu$ (see regime III'). Observe that $\lambda_2=\mu$ leads to coexistence whenever $\lambda_1>\mu$. Further for any positive $\lambda_2 \neq \mu$, stable coexistence is possible given that $\lambda_1$ is sufficiently large. In fact, if $\lambda_1$ is larger than the intersection of the orange line with the $x$-axis, then any $\lambda_2 \in (0, \mu)$ is in the stable coexistence regime. Consequently, regime II'\ is absent, and increasing $\lambda_1$ is actually deleterious for trait 1 because the length of the vertical section of regime II''\ decreases in $\lambda_1$ and reaches 0 at a finite value of $\lambda_1$.
\end{itemize}
\end{enumerate}
Note that in general, as long as $p \in (0,1)$, $\kappa\geq 0$, and $\sigma,\mu,C,\tau>0$, the orange line has a negative slope if and only if $\tau>C$, i.e., HGT is stronger than competition. If this is satisfied and also \eqref{mutantlessfit2ineq} holds, then stable coexistence occurs for some choice of the parameters with $\lambda_2 \leq \mu < \lambda_1$. For $C=\tau$, the orange critical line coincides with the one given by $\lambda_2=\mu$. Hence, if $\lambda_2>\mu$, then trait 2 can coexist with trait 1 once $\lambda_1$ is sufficiently large, but if $\lambda_2<\mu$, then trait 2 cannot invade trait 1.

\subsection{The dormancy-free case with HGT}\label{sec-coex2d}
In the special case $p=0$ when trait 1 exhibits no dormancy, the dormant coordinate of trait 1 can be ignored. This leads to the limiting dynamical system
\begin{equation}\label{2dimHGTgood}
\begin{aligned}
\frac{\d n_{1,t}}{\d t} & = n_{1,t}\big( \lambda_1-\mu-Cn_{1,t}-Cn_{2,t}-\tau n_{2,t} \big), \\
\frac{\d n_{2,t}}{\d t} & = n_{2,t} \big( \lambda_2-\mu-Cn_{1,t}-Cn_{2,t} +\tau n_{1,t} \big),
\end{aligned}
\end{equation}
with some competition parameter $C>0$. This is a special case of the dynamical system studied in \cite[Proposition 4.1]{BCFMT18} corresponding to density-dependent HGT. However, their system comes with rather different notation, therefore we briefly describe the main properties of our model for the reader's convenience.

Using the terminology of \cite[Section 1.1]{B19}, \eqref{2dimHGTgood} is a Lotka--Volterra system with constant intraspecific competition, but with an asymmetric relation between the two traits: the interaction between the two traits is less disadvantageous for trait 2 than for trait 1. As long as $\lambda_1>\mu$ and $\tau<C$, this interaction is disadvantageous also for trait 2, hence the Lotka--Volterra system is competitive. Else if $\lambda_1>\mu$ but $\tau \geq C$, the interaction is still disadvantageous for trait 1 but beneficial (for $\tau > C$) or neutral (in the boundary case $\tau=C$) for trait 2. This yields a host-parasite or prey-predator type interaction between trait 1 and trait 2 if $\tau>C$. (For $p>0$, the system~\eqref{3dimHGT} is not of Lotka--Volterra type in terms of~\cite[Section 1.1]{B19}, but it is clear that the interaction between traits 1 and 2 is still competitive for $C>\tau$ and predator--prey or host--parasite type if $\tau > C$.)

Let us recall the equilibrium $\bar n_2=\smfrac{(\lambda_2-\mu)\vee 0}{C}$, and let us denote the analogous one-type equilibrium $\smfrac{(\lambda_1-\mu) \vee 0}{C}$ of trait 1 by $\bar n_1$. It is straightforward to derive that the system \eqref{2dimHGTgood} exhibits a coexistence equilibrium $(n_1,n_2)$ with two positive coordinates if and only if
\[ (\lambda_2-\mu) < \frac{C}{\tau} (\lambda_1-\lambda_2) <\lambda_1-\mu. \numberthis\label{coexistencecondnodormancy} \]
This chain of inequalities clearly implies $\lambda_1>\lambda_2$, and from this it follows that it also implies $\lambda_1>\mu$. Again, $\lambda_2 \leq \mu$ (i.e., trait 2 being individually unfit) is possible under conditon~\eqref{coexistencecondnodormancy}.
In case~\eqref{coexistencecondnodormancy} holds, the coexistence equilibrium is unique and given as
\begin{equation}\label{nodormancyeq}
\begin{aligned}
n_2& =\frac{\lambda_1-\mu - \frac{C}{\tau} (\lambda_1-\lambda_2)}{\tau}, \\
n_1& =\frac{\mu-\lambda_2+\frac{C}{\tau}(\lambda_1-\lambda_2)}{\tau}.
\end{aligned}
\end{equation}
If \eqref{coexistencecondnodormancy} holds, then elementary computations imply that $(n_1, n_2)$ is asymptotically stable, whereas all other equilibria, namely $(0,0)$, $(\bar n_1,0)$, and $(0,\bar n_2)$, are unstable. (In particular, verifying local asymptotic stability of the coexistence equilibrium is straightforward here, unlike for the three-dimensional system~\eqref{3dimHGT}). Else if \[ (\lambda_2-\mu) < \frac{C}{\tau} (\lambda_1-\lambda_2) > \lambda_1-\mu \numberthis\label{2strongernodormancy} \] 
holds with $\lambda_1>\mu$, $(0,\bar n_2)$ is asymptotically stable, whereas $(0,0)$ and $(0,\bar n_2)$ are unstable. Similarly, if
\[ (\lambda_2-\mu) > \frac{C}{\tau} (\lambda_1-\lambda_2) < \lambda_1-\mu \numberthis\label{1strongernodormancy} \]
holds with $\lambda_2>\mu$, $(0,\bar n_1)$ is asymptotically stable, whereas $(0,0)$ and $(\bar n_1,0)$ are unstable. Finally, if
\[ (\lambda_2-\mu) > \frac{C}{\tau} (\lambda_1-\lambda_2) > \lambda_1-\mu, \numberthis\label{foundercontrolnodormancy} \]
then $(0,0)$ is unstable and both $(\bar n_1, 0)$ and $(0,\bar n_2)$ are asymptotically stable. Based on this and using the methods of Sections~\ref{sec-HGTproofs} and \ref{sec-dormantproofs}, one can derive that if at least one of the conditions \eqref{coexistencecondnodormancy}, \eqref{2strongernodormancy}, \eqref{1strongernodormancy}, and~\eqref{foundercontrolnodormancy} holds (i.e., if all equilibria of \eqref{2dimHGTgood} are hyperbolic), then none of the approximating branching processes defined analogously to Section~\ref{sec-phase13} is critical, and the following assertions hold:
\begin{enumerate}[(I)]
    \item If~\eqref{coexistencecondnodormancy} holds, then both traits can invade the other one but none of them can reach fixation. Starting from a resident population of one trait and a single mutant of the other, the system will converge with asymptotically positive probability as $K \to \infty$ to the stable coexistence equilibrium $(n_1,n_2)$. This is the analogue of the case when~\eqref{mutantlessfit2ineq} holds if trait 1 exhibits dormancy. It is easy to see that just as in the case $p>0$, this regime contains choices of parameters with $\lambda_2 < \mu$ if and only if $C>\tau$, that is, if trait 2 is individually unfit, one can find $\lambda_1>\mu$ corresponding to stable coexistence of the two traits if and only if HGT is stronger than competition.
    \item If~\eqref{1strongernodormancy} holds, then trait 1 can invade trait 2 and even reach fixation with asymptotically positive probability, whereas with high probability, trait 2 cannot invade trait 1. This corresponds to the case when~\eqref{reversefirstin} and \eqref{secondinmutantlessfit2ineq} hold.
    \item If~\eqref{foundercontrolnodormancy} holds, then with probability tending to 1 as $K \to \infty$, none of the two traits can invade the other. This is the case of founder control, which corresponds to the case~\eqref{mutantfitter2ineq} in case trait 1 exhibits competition-induced dormancy. However, there is a major difference: under the assumption \eqref{mutantfitter2ineq}, the coexistence equilibrium $(n_{1a},n_{1d},n_2)$ exists as a coordinatewise positive (but unstable) equilibrium, whereas under condition \eqref{foundercontrolnodormancy}, there exists no coexistence equilibrium.
    \item If~\eqref{2strongernodormancy} holds, then the assertions provided about the case (II) are true with traits 1 and 2 interchanged. This is analogous to the case when \eqref{firstinmutantlessfit2ineq} and~\eqref{reversesecondin} hold.
\end{enumerate}
What remains to analyse are the critical cases when one of the inequalities of~\eqref{coexistencecondnodormancy} is true with the inequality replaced by an equality, which we defer to later work.

In Figure~\ref{figure-phasediagrams} we depict three choices of parameters corresponding to stable coexistence in~\eqref{2dimHGTgood}.


\begin{figure}
    \centering
    \vspace{1pt}
    \begin{subfigure}{0.4\textwidth}
    \includegraphics[scale=2]{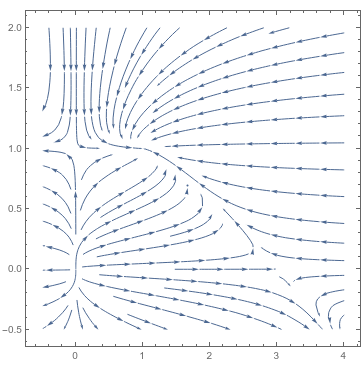}
    \caption{$\lambda_1=5,\lambda_2=3,\mu=2,C=\tau=1$: trait 2 is individually fit and the coexistence condition \eqref{coexistencecondnodormancy} holds. Then, $(0,0)$ is a source, the one-trait equilibria of traits 1 and 2 are saddle points and $(n_1,n_2)=(1,1)$ is a sink.\\}
    \end{subfigure}
    \hspace{1cm}
    \begin{subfigure}{0.4\textwidth}
    \vspace{-36pt}
    \includegraphics[scale=2]{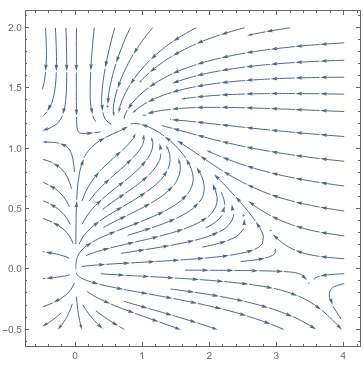} 
    \caption{$C=0.9$ and otherwise the same parameters: the only qualitative change is that $(n_1,n_2)$ is now a stable focus.}
    \end{subfigure}
    \hspace{1pt}
    \begin{subfigure}{1\textwidth}
    \centering
    \includegraphics[scale=2]{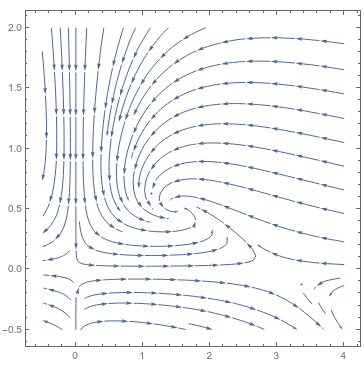}
     \caption{$\lambda_1=5,\lambda_2=1,\mu=2,C=1,\tau=2$: trait 2 is individually unfit but the coexistence condition \eqref{coexistencecondnodormancy} holds. $(0,0)$ and $(\bar n_{1},0)$ are now saddle points, whereas the single-trait equilibrium of trait 2 is not included in the positive orthant. The coexistence equilibrium $(n_1,n_2)$ is again a stable focus.\\}
   \end{subfigure}
 \caption{Three cases of stable coexistence for~\eqref{2dimHGTgood}. The $x$ axis shows the trait 1 and the $y$ axis the trait 2 population size.}\label{figure-phasediagrams}
\end{figure}
\subsection{The HGT-free case with dormancy.}\label{sec-onlydormancybackwards}
Using the notation of Section~\ref{sec-main}, the special case $\tau=0$ of no HGT leads to the dynamical system
\begin{equation}\label{3dimnoHGT}
\begin{aligned}
\frac{\d n_{1a}(t)}{\d t} & = n_{1a}(t)\big( \lambda_1-\mu-C (n_{1a}(t)+n_{2}(t)) \big)  + \sigma n_{1d}(t), \\
\frac{\d n_{1d}(t)}{\d t} & = p C n_{1a}(t)(n_{1a}(t)+n_{2}(t)) - (\kappa\mu+\sigma) n_{1d}(t), \\
\frac{\d n_{2}(t)}{\d t} & = n_{2}(t) \big( \lambda_2-\mu-C (n_{1a}(t)+n_{2}(t)) \big). 
\end{aligned}
\end{equation}
Here, the competition between the two traits is equally disadvantageous for the two traits, whereas trait 1 benefits from competition-induced dormancy.
The main result of our previous paper~\cite{BT19} is that under the condition
\[ \lambda_2-\lambda_1 < p(\lambda_2-\mu)\frac{\sigma}{\kappa\mu+\sigma}, \qquad \lambda_1>\mu,\lambda_2>\mu,  \numberthis\label{fitterw/oHGT}\]
the one-trait equilibrium $(\bar n_{1a},\bar n_{1d},0)$ of trait 1 is asymptotically stable, whereas the one-trait equilibrium $(0,0,\bar n_2)$ of trait 2 and the equilibrium $(0,0,0)$ are unstable, and trait 1 can invade trait 2 with asymptotically positive probability as $K \to \infty$, while invasion implies fixation. Our main motivation in~\cite{BT19} was to show that invasion of trait 1 is possible in some cases when trait 2 has a higher reproduction rate (which is possible under condition~\eqref{fitterw/oHGT}), which is excluded in the case $p=0$ of no dormancy and constant competition. Hence, we did not treat any case where $\lambda_1 \geq \lambda_2$, however, since we showed that invasion and fixation of trait 1 is possible in some cases when $\lambda_2>\lambda_1$, the same assertion follows in the case $\lambda_1 \geq \lambda_2$ by a simple coupling argument. We did not consider the case when  \eqref{fitterw/oHGT} does not hold and we did not tell about the invasion of trait 2 against trait 1. Since these cases can also be handled with the methods employed in the Appendix of the present paper, let us now describe the landscape of invasion and fixation in all non-critical cases that have not been addressed before.

If~\eqref{fitterw/oHGT} holds, then, not surprisingly, a mutant of trait 2 trying to invade a resident population of trait 1 will go extinct (already during the first phase of invasion) with high probability as $K \to \infty$. In contrast, in case~\eqref{fitterw/oHGT} does not hold but instead we have
\[ \lambda_2-\lambda_1 > p(\lambda_2-\mu)\frac{\sigma}{\kappa\mu+\sigma}, \qquad \lambda_1>\mu,\lambda_2>\mu,  \numberthis\label{lessfitw/oHGT}\]
then $(0,0,\bar n_2)$ is asymptotically stable, whereas $(\bar n_{1a},\bar n_{1d},0)$ and $(0,0,0)$ are unstable. In this case, trait 2 can invade trait, with asymptotically positive probability, whereas with probability tending to one, trait 1 cannot invade trait 2. Also here, a successful invasion implies fixation with high probability as $K\to\infty$. Note that all these assertions can easily be derived as special cases of the main results of the present paper. They are also not surprising given the results of~\cite{BT19},  but they cannot be proven with the methods of that paper, especially handling the second phase of invasion of trait 2 against trait 1 requires additional arguments, which can be chosen analogously to the Appendix, Section~\ref{sec-phase2HGTguys}.

In absence of HGT, $\lambda_i \leq \mu$ will lead to an extinction of the trait $i$ population in $O(1)$ time with high probability as $K \to \infty$. Further, also in the case $\lambda_1,\lambda_2>\mu$ there is no coexistence; founder control is also excluded apart from the boundary case when neither~\eqref{fitterw/oHGT} nor \eqref{lessfitw/oHGT} holds, which is the only case when $(\bar n_{1a},\bar n_{1d},0)$ and $(0,0,\bar n_2)$ are non-hyperbolic given that both traits are fit. This critical case requires further analysis, which we defer to later work.

Note that since $p>0$, $\sigma,\mu>0$ and $\kappa\geq 0$, \eqref{lessfitw/oHGT} with $\lambda_1,\lambda_2>\mu$ always implies that $\lambda_2>\lambda_1$. That is, trait 2, which lacks dormancy, can only invade trait 1 if it has a strictly higher birth rate than trait 1, and the difference must even be bounded away from zero (where the bound depends on the parameters). This is certainly not the case for $\tau>0$ where invasion of trait 2 can occur even in the case $\lambda_1>\mu \geq \lambda_2$.

Finally, let us also note that in the model of the present paper with HGT and dormancy, the condition~\eqref{mutantfitter2ineq} of founder control and unstable coexistence implies condition~\eqref{fitterw/oHGT}. This means that a positive but not very high amount of HGT can give rise to founder control in some cases where trait 1 would dominate trait 2 for $\tau=0$.

\newpage

\appendix
\section{Preliminaries: existence and stability of equilibria}\label{sec-preliminaries}
\subsection{Existence of a coexistence equilibrium}\label{sec-coexistenceproof}
This section is devoted to the proof of the results of Section~\ref{sec-coexistence}. We first verify Lemma~\ref{lemma-coexistence}.
\begin{proof}[Proof of Lemma~\ref{lemma-coexistence}.]
Assume that $(n_{1a},n_{1d},n_2)$ is a coexistence equilibrium. Making all the three equations of \eqref{3dimHGT} equal to zero and using that none of the coordinates is zero, we arrive at the following characterization of the coexistence equilibrium (cf.~\eqref{coexeq}).
\begin{equation}\numberthis\label{coexeqoncemore}
    \begin{aligned}
    n_{1a} & = \frac{C(\kappa\mu+\sigma)(\lambda_2-\lambda_1)+Cp\sigma(\mu-\lambda_2)+(\kappa\mu+\sigma)\tau(\lambda_2-\mu)}{\tau(Cp\sigma-(\kappa\mu+\sigma)\tau)}, \\
    n_{1d} & = \frac{pC(\lambda_2-\lambda_1) \big(C(\kappa\mu+\sigma)(\lambda_2-\lambda_1)+Cp\sigma(\mu-\lambda_2)+(\kappa\mu+\sigma)\tau(\lambda_2-\mu)\big)}{\tau(Cp\sigma-(\kappa\mu+\sigma)\tau)^2}, \\
    n_{2} & = \frac{C(\kappa\mu+\sigma)(\lambda_2-\lambda_1)+Cp\sigma(\mu-\lambda_2)+(\kappa\mu+\sigma)\tau(\lambda_1-\mu)}{-\tau(Cp\sigma-(\kappa\mu+\sigma)\tau)},
    \end{aligned}
\end{equation}
which already implies that there is at most one coexistence equilibrium. In particular,
\[ n_{1d} = \frac{pC(\lambda_2-\lambda_1)}{Cp\sigma-(\kappa\mu+\sigma)\tau} n_{1a}. \]
Therefore, in order to have positivity of both $n_{1a}$ and $n_{1d}$, one of the following conditions has to be satisfied:
\begin{enumerate}
    \item $\lambda_2>\lambda_1$ and $\tau<\frac{Cp\sigma}{\kappa\mu+\sigma}$.
    \item $\lambda_1>\lambda_2$ and $\tau>\frac{Cp\sigma}{\kappa\mu+\sigma}$. 
\end{enumerate}
In case the first condition is satisfied, both the numerator and the denominator of the expression for $n_{1a}$ in \eqref{coexeqoncemore} are positive. Hence, the denominator of the expression for $n_{2}$ is negative, and hence the numerator of the expression has also to be negative, which holds if and only if condition \eqref{mutantfitter2ineq} holds (which in particular implies that $\lambda_2>\lambda_1$). 

Similarly, if the second condition is satisfied, then both the numerator and the denominator of $n_{1a}$ in \eqref{coexeqoncemore} are negative. Therefore, the denominator of the expression for $n_{2}$ is positive, which implies that the numerator must also be positive, which is equivalent to condition \eqref{mutantlessfit2ineq} (which in particular implies that $\lambda_1>\lambda_2$). 
We conclude that a coexistence equilibrium exists if and only if either \eqref{mutantfitter2ineq} holds and $\tau < \frac{Cp\sigma}{\kappa\mu+\sigma}$ or \eqref{mutantlessfit2ineq} holds and $\tau > \frac{Cp\sigma}{\kappa\mu+\sigma}$. Hence, the proof of Lemma~\ref{lemma-coexistence} is finished as soon as we have verified the following lemma.
\end{proof}
\begin{lemma}\label{lemma-lessconditions}
Assume $\lambda_1>\mu$. Then, \eqref{mutantfitter2ineq} implies  $\tau<\frac{Cp\sigma}{\kappa\mu+\sigma}$ and \eqref{mutantlessfit2ineq} implies $\tau>\frac{Cp\sigma}{\kappa\mu+\sigma}$.
\end{lemma}
\begin{proof}
We first show that if $\lambda_1>\mu$, then the expression \[ \frac{Cp\sigma}{\tau(\kappa\mu+\sigma)} (\lambda_2-\mu)+\frac{C}{\tau}(\lambda_1-\lambda_2) \numberthis\label{middleterm} \]
appearing both in \eqref{mutantfitter2ineq} and \eqref{mutantlessfit2ineq} is positive under either of these equations. Indeed, if \eqref{mutantfitter2ineq} holds, then this positivity is clear from the second inequality of \eqref{mutantfitter2ineq}. On the other hand, it is true in general that \eqref{middleterm} equals
\[ \frac{C}{\tau} \frac{\kappa\mu+(1-p)\sigma}{\kappa\mu+\sigma} \Big[ (\lambda_1-\mu)\frac{\kappa\mu+\sigma}{\kappa\mu+(1-p)\sigma} -(\lambda_2-\mu) \Big], \]
which is clearly positive if $\lambda_1>\mu$ and \eqref{mutantlessfit2ineq} holds, since in this case $\lambda_1>\lambda_2$ also holds.

Let us first assume for a contradiction that \eqref{mutantfitter2ineq} holds with $\lambda_1>\mu$, in particular $\lambda_2>\lambda_1$, but $\tau \geq \frac{Cp\sigma}{\kappa\mu+\sigma}$. Then, using the positivity of \eqref{middleterm}, we can estimate  
\[
\begin{aligned}
\frac{Cp\sigma}{\tau(\kappa\mu+\sigma)}(\lambda_2-\mu)+\frac{C}{\tau}  &(\lambda_1-\lambda_2) =\frac{Cp\sigma}{\tau(\kappa\mu+\sigma)} \Big( \lambda_2-\mu+ \frac{\kappa\mu+\sigma}{p\sigma}(\lambda_1-\lambda_2) \Big) 
\\ &\leq  \lambda_2-\mu + \frac{\kappa\mu+\sigma}{p\sigma}(\lambda_1-\lambda_2) \\ & = - \frac{\kappa\mu-(1-p)\sigma}{p\sigma}(\lambda_2-\mu) + \frac{\kappa\mu+\sigma}{p\sigma}(\lambda_1-\mu) \\ & = \frac{\kappa\mu+(1-p)\sigma}{p\sigma} \Big( -(\lambda_2-\mu)+(\lambda_1-\mu) \frac{\kappa\mu+\sigma}{\kappa\mu+(1-p)\sigma} \Big)
\\ & \leq \frac{\kappa\mu+(1-p)\sigma}{p\sigma} \Big( -(\lambda_1-\mu) +(\lambda_1-\mu) \frac{\kappa\mu+\sigma}{\kappa\mu+(1-p)\sigma}\Big)=(\lambda_1-\mu),
\end{aligned}
\]
which contradicts the second inequality of \eqref{mutantfitter2ineq}. 

Second, let us assume for a contradiction that \eqref{mutantlessfit2ineq} holds with $\lambda_1>\mu$, in particular $\lambda_1>\lambda_2$, but $\tau \leq \frac{Cp\sigma}{\kappa\mu+\sigma}$. Then, since \eqref{middleterm} is again positive, we obtain
\[
\begin{aligned}
\frac{Cp\sigma}{\tau(\kappa\mu+\sigma)}(\lambda_2-\mu)+\frac{C}{\tau}   &(\lambda_1-\lambda_2)=\frac{Cp\sigma}{\tau(\kappa\mu+\sigma)} \Big( \lambda_2-\mu+ \frac{\kappa\mu+\sigma}{p\sigma}(\lambda_1-\lambda_2) \Big) 
\\ &\geq  \lambda_2-\mu + \frac{\kappa\mu+\sigma}{p\sigma}(\lambda_1-\lambda_2) \\ & = - \frac{\kappa\mu-(1-p)\sigma}{p\sigma}(\lambda_2-\mu) +  \frac{\kappa\mu+\sigma}{p\sigma}(\lambda_1-\mu) \\ & = \frac{\kappa\mu+(1-p)\sigma}{p\sigma} \Big( -(\lambda_2-\mu)+(\lambda_1-\mu) \frac{\kappa\mu+\sigma}{\kappa\mu+(1-p)\sigma} \Big)
\\ & \geq \frac{\kappa\mu+(1-p)\sigma}{p\sigma} \Big( -(\lambda_1-\mu) +(\lambda_1-\mu) \frac{\kappa\mu+\sigma}{\kappa\mu+(1-p)\sigma}\Big)=(\lambda_1-\mu),
\end{aligned}
\]
which contradicts the second inequality of \eqref{mutantlessfit2ineq}. 
Hence, the lemma follows.
\end{proof}
Finally, we prove Corollary~\ref{cor-1fit}. 
\begin{proof}[Proof of Corollary~\ref{cor-1fit}.]
Let us first assume that condition \eqref{mutantfitter2ineq} holds. Then, by Lemma~\ref{lemma-lessconditions}, it follows that $\tau<\frac{Cp\sigma}{\kappa\mu+\sigma}$. Assume now that $\lambda_1 \leq \mu$. Then, by \eqref{mutantfitter2ineq}, the numerator of the right-hand side of $n_{2}$ in \eqref{coexeqoncemore} must be negative. This together with the conditions $\tau<\frac{Cp\sigma}{\kappa\mu+\sigma}$, $\lambda_2>\lambda_1$, and $p \in (0,1)$ implies that
\[
\begin{aligned}
& C(\kappa\mu+\sigma)(\lambda_2-\lambda_1)+Cp\sigma(\mu-\lambda_2)+(\kappa\mu+\sigma)\tau(\lambda_1-\mu)  \\ & \geq C(\kappa\mu+\sigma)(\lambda_2-\lambda_1)+Cp\sigma(\mu-\lambda_2)+Cp\sigma(\lambda_1-\mu)  \\ & = C(\kappa\mu+\sigma-p\sigma) (\lambda_2-\lambda_1) > 0,
\end{aligned}
\]
which contradicts the assumption that the numerator of the expression for $n_2$ in \eqref{coexeqoncemore} is negative.

Second, let us assume that condition \eqref{mutantlessfit2ineq} is satisfied. Thanks to Lemma~\ref{lemma-lessconditions}, this implies that $\tau>\frac{Cp\sigma}{\kappa\mu+\sigma}$. Assume that $\lambda_1 \leq \mu$. Then, by \eqref{mutantlessfit2ineq}, the numerator of the right-hand side of the expression for $n_2$ in \eqref{coexeq} must be positive. This together with the conditions $\tau>\frac{Cp\sigma}{\kappa\mu+\sigma}$, $\lambda_2<\lambda_1$, and $p \in (0,1)$ yields
\[
\begin{aligned}
& C(\kappa\mu+\sigma)(\lambda_2-\lambda_1)+Cp\sigma(\mu-\lambda_2)+(\kappa\mu+\sigma)\tau(\lambda_1-\mu)  \\ & \leq C(\kappa\mu+\sigma)(\lambda_2-\lambda_1)+Cp\sigma(\mu-\lambda_2)+Cp\sigma(\lambda_1-\mu)  \\ & = C(\kappa\mu+\sigma-p\sigma) (\lambda_2-\lambda_1) < 0,
\end{aligned}
\]
which is again a contradiction.
\end{proof}

\subsection{Stability of equilibria}\label{sec-stabilityproof}
In this section, we verify Proposition~\ref{prop-stability3d} and Lemma~\ref{lemma-somebodyfit}, using Lemma~\ref{lemma-lessconditions}.
\begin{proof}[Proof of Proposition~\ref{prop-stability3d}.]
At any equilibrium $(\widehat n_{1a}, \widehat n_{1d}, \widehat n_2)$ we have the Jacobi matrix
\begin{equation}\label{3dJacobi}
A(\widehat n_{1a}, \widehat n_{1d}, \widehat n_2)= \begin{pmatrix}
\lambda_1-\mu-2 C\widehat n_{1a}-(C+\tau) \widehat n_2 & \sigma & (-C-\tau)\widehat n_{1a} \\
2 p C \widehat n_{1a} + p C \widehat n_2 & - \kappa\mu-\sigma & p C \widehat n_{1a} \\
(-C+\tau)\widehat n_{2} & 0 & \lambda_2-\mu-2 C \widehat n_{2}-(C-\tau) \widehat n_{1a}
\end{pmatrix}.
\end{equation}

At $(0,0,0)$, we have
\[ A(0,0,0)= \begin{pmatrix}
\lambda_1-\mu & \sigma & 0 \\
0 & - \kappa\mu-\sigma & 0 \\
0 & 0 & \lambda_2-\mu
\end{pmatrix}. \]
The eigenvalues of this matrix are its diagonal entries, and thus $\lambda_1-\mu>0$ implies that the matrix is indefinite. If $\lambda_2-\mu \geq 0$, it has only one negative eigenvalue, else (i.e., if trait 2 is individually strictly unfit) it has two ones. Either way, $(0,0,0)$ is unstable. 

At $(0,0,\smfrac{\lambda_2-\mu}{C})$, which equilibrium is contained in the closed positive orthant if and only if $\lambda_2>\mu$ (in which case it equals $(0,0,\bar n_2)$), the Jacobi matrix is given as
\[ A\Big(0,0,\frac{\lambda_2-\mu}{C}\Big) = \begin{pmatrix}
\lambda_1-\mu-\frac{(C+\tau)(\lambda_2-\mu)}{C} & \sigma & 0 \\
p(\lambda_2-\mu)& - \kappa\mu-\sigma & 0 \\
\frac{(-C+\tau)(\lambda_2-\mu)}{C} & 0 & -(\lambda_2-\mu)
\end{pmatrix}. \]
We immediately see that $-(\lambda_2-\mu)$ is an eigenvalue of this matrix, which is negative under the assumption that $\lambda_2>\mu$. Further, the determinant of $A(0,0,\smfrac{\lambda_2-\mu}{C})$ is the product of this eigenvalue and the determinant of
\[A_1\Big(0,0,\frac{\lambda_2-\mu}{C}\Big):= \begin{pmatrix}
\lambda_1-\mu-\frac{(C+\tau)(\lambda_2-\mu)}{C} & \sigma  \\
p(\lambda_2-\mu)& - \kappa\mu-\sigma
\end{pmatrix}. \]
Now, if the first inequality in \eqref{mutantfitter2ineq} is satisfied, then we have
\[ \det A_1\Big(0,0,\frac{\lambda_2-\mu}{C}\Big)=\Big(\lambda_1-\lambda_2-\frac{\tau}{C}(\lambda_2-\mu)\Big)(-\kappa\mu-\sigma)-p(\lambda_2-\mu)\sigma  >p(\lambda_2-\mu)\sigma -p(\lambda_2-\mu)\sigma =0. \]
On the other hand, the same arguments imply that
\[ \Tr A_1\Big(0,0,\frac{\lambda_2-\mu}{C}\Big) < -\frac{p\sigma}{\kappa\mu+\sigma} (\lambda_2-\mu)-\kappa\mu-\sigma<0. \]
Hence, both eigenvalues of $A_1(0,0,\smfrac{\lambda_2-\mu}{C})$ must have a strictly negative real part, which implies that $(0,0,\smfrac{\lambda_2-\mu}{C})$ is asymptotically stable. Else, $\det A_1(0,0,\frac{\lambda_2-\mu}{C})$ is nonpositive, in particular if the first inequality of \eqref{mutantlessfit2ineq} holds, then it is strictly negative, which implies that $A_1(0,0,\smfrac{\lambda_2-\mu}{C})$ is indefinite (with two real eigenvalues, one of them being positive and one of them negative), and hence $A(0,0,\smfrac{\lambda_2-\mu}{C})$ is also indefinite and $(0,0,\smfrac{\lambda_2-\mu}{C})$ is unstable. 
In the latter case, at least one of the eigenvalues of $A_1\Big(0,0,\frac{\lambda_2-\mu}{C}\Big)$ has to have a strictly positive real part, which implies that $(0,0,\smfrac{\lambda_2-\mu}{C})$ is unstable. 

Next, let us investigate the stability of the equilibrium $(\bar n_{1a}, \bar n_{1d}, 0)$. Thanks to Corollary~\ref{cor-1fit}, this equilibrium is unequal to $(0,0,0)$ unless the first inequality of \eqref{mutantfitter2ineq} and the second one of \eqref{mutantlessfit2ineq} holds, which is the only case when $\lambda_1\leq\mu$ is possible. Since we have already analysed the stability of $(0,0,0)$, in the following, without loss of generality we can assume that $\lambda_1>\mu$. Then, $(\bar n_{1a},\bar n_{1d},0)$ has two positive coordinates (cf.~\cite{BT19}), and at this equilibrium, we have the Jacobi matrix
\begin{equation}
A(\bar n_{1a}, \bar n_{1d}, 0)= \begin{pmatrix}
\lambda_1-\mu-2 C\bar n_{1a} & \sigma & (-C-\tau)\bar n_{1a} \\
2 p C \bar n_{1a} & - \kappa\mu-\sigma & p C \bar n_{1a} \\
0 & 0 & \lambda_2-\mu-(C-\tau)\bar n_{1a}
\end{pmatrix}.
\end{equation}
We immediately see that the last diagonal entry of the Jacobi matrix is an eigenvalue of the matrix, whereas the other two eigenvalues are equal to the eigenvalues of 
\[ A_1(\bar n_{1a}, \bar n_{1d}, 0)= \begin{pmatrix}
\lambda_1-\mu-2 C\bar n_{1a} & \sigma  \\
2 p C \bar n_{1a} & - \kappa\mu-\sigma
\end{pmatrix}. \]
Under the assumption that $\lambda_1>\mu$, both eigenvalues of this matrix are negative, see \cite[Section 2.2.1]{BT19}. Further, using the definition of $\bar n_{1a}$, the last diagonal entry equals
\[ \lambda_2-\mu-(C-\tau) \frac{\kappa\mu+\sigma}{\kappa\mu+(1-p)\sigma}\frac{\lambda_1-\mu}{C}.\]
This being negative is equivalent to the condition
\[ (\lambda_2-\mu)\frac{p\sigma-\kappa\mu-\sigma}{\kappa\mu+\sigma}>(\lambda_1-\mu)(1-C/\tau), \]
which is equivalent to the second inequality in \eqref{mutantfitter2ineq}; in this case the equilibrium $(\bar n_{1a}, \bar n_{1d}, 0)$ is asymptotically stable. Similarly, this being positive is equivalent to the second inequality in \eqref{mutantlessfit2ineq}; in this case the equilibrium is unstable.

Lastly, we analyse the stability of the coexistence equilibrium $(n_{1a},n_{1d},n_2)$. Thanks to Lemmas~\ref{lemma-coexistence} and \ref{lemma-lessconditions}, this equilibrium is contained in the open positive orthant if and only if either \eqref{mutantfitter2ineq} or \eqref{mutantlessfit2ineq} holds. Our goal is now to show that under condition \eqref{mutantfitter2ineq} the coexistence equilibrium is unstable. 
Hence, we want to determine the signs of the real parts of the eigenvalues of the Jacobi matrix corresponding to this equilibrium, which depend essentially on all the model parameters.

First of all, since $(n_{1a},n_{1d},n_2)$ is an equilibrium, all right-hand sides in \eqref{3dimHGT} are equal to zero. Using that each of the three coordinates are unequal to zero, this implies
\begin{align}
    \lambda_1-\mu-C(n_{1a}+n_2)-\tau n_2 + \sigma \frac{n_{1d}}{n_{1a}}&=0, \label{first-coexeq} \\
    pC(n_{1a}+n_2)-(\kappa\mu+\sigma) \frac{n_{1d}}{n_{1a}}&=0, \label{second-coexeq} \\
    \lambda_2-\mu-C(n_{1a}+n_2)+\tau n_{1a} & = 0. \label{third-coexeq}
\end{align}
Let us substitute \eqref{first-coexeq} and \eqref{third-coexeq} into \eqref{3dJacobi} for $(\widehat n_{1a},\widehat n_{1d},\widehat n_2)=(n_{1a},n_{1d},n_2)$. This gives that
\[ A(n_{1a},n_{1d},n_2)=\begin{pmatrix} - Cn_{1a} -\sigma \frac{n_{1d}}{n_{1a}} & \sigma & (-C-\tau)n_{1a} \\ 2pC n_{1a}+p C n_2 & -\kappa\mu-\sigma & pCn_{1a} \\ (-C+\tau) n_2 & 0 & - C n_{2} \end{pmatrix}. \numberthis\label{Arewritten} \]

During our analysis, it will be convenient to permute the order of the three equations of \eqref{3dimHGT}. To be more precise, for a permutation matrix $P \in \R^{3\times 3}$, $P A(n_{1a},n_{1d},n_2) P^T$ has the same eigenvalues as $A(n_{1a},n_{1d},n_2)$, and $P A(n_{1a},n_{1d},n_2) P^T$ also equals the Jacobi matrix of the permuted system of ODEs. We choose \[ P=\begin{pmatrix} 1 & 0 & 0 \\ 0 & 0 & 1 \\ 0 & 1 & 0 \end{pmatrix}, \] which yields $P A(n_{1a},n_{1d},n_2) P^T=:\widetilde A (n_{1a},n_{1d},n_{2})$ where 
\[ \widetilde A (n_{1a},n_{1d},n_{2})= \begin{pmatrix} -C n_{1a}-\sigma \frac{n_{1d}}{n_{1a}} & (-C-\tau)n_{1a} &  \sigma \\ (-C+\tau) n_2 & -C n_2 & 0 \\ 2pCn_{1a}+pC n_2 & p C n_{1a} & -\kappa\mu-\sigma \end{pmatrix}. \]
The second leading principal minor of this matrix equals
\[ C^2 n_{1a}n_2 + C \sigma \frac{n_{1d}}{n_{1a}} n_2-(C^2-\tau^2)n_{1a}n_2=C \sigma \frac{n_{1d}}{n_{1a}} n_2+\tau^2 n_{1a}n_2. \]
Hence, the determinant of the entire Jacobi matrix is given as
\[
\begin{aligned}
\det \widetilde A (n_{1a},n_{1d},n_{2}) &= -\big(C\sigma \frac{n_{1d}}{n_{1a}}n_2 + \tau^2 n_{1a} n_2\big)(\kappa\mu+\sigma) + C n_2 \sigma (2p C n_{1a}+p Cn_{2} ) - pC(C-\tau) n_{1a} n_2 \sigma 
\\ & = -\big(C\sigma \frac{n_{1d}}{n_{1a}} n_2+ \tau^2 n_{1a} n_2\big)(\kappa\mu+\sigma) + Cp\sigma n_2\big(C (n_{1a}+n_2)+\tau n_{1a}\big)
\\ & = \tau n_{1a} n_2 ((-\kappa\mu-\sigma)\tau + Cp\sigma) +C\sigma n_2 \big(Cp(n_{1a}+n_2) - \frac{n_{1d}}{n_{1a}}(\kappa\mu+\sigma)\big),
\\ & =  \tau n_{1a} n_2 ((-\kappa\mu-\sigma)\tau + Cp\sigma). 
\end{aligned}\numberthis\label{finaldet3d} \]
where in the last line we used \eqref{second-coexeq}. Now, under the assumption \eqref{mutantfitter2ineq}, the right-hand side of \eqref{finaldet3d} is positive. Hence, if there exists a pair of conjugate eigenvalues, their product is positive, and hence the third eigenvalue (where we count eigenvalues with multiplicity, in particular, the third eigenvalue is necessarily real) must be positive in order that the determinant is positive. Else, all the three eigenvalues must be real, and since their product is positive, there must be a positive one among them. We conclude that $(n_{1a},n_{1d},n_2)$ is unstable under condition~\eqref{mutantfitter2ineq}. We also note that since the trace of $A(n_{1a},n_{1d},n_2)$ is negative, at least one of the eigenvalues must have a negative real part. This concludes the proof. 
\end{proof} 
The next remark summarizes the consequences of the proof of Proposition~\ref{prop-stability3d} regarding the stability of $(n_{1a},n_{1d},n_2)$ under the assumption~\eqref{mutantlessfit2ineq} of `stable coexistence'.
\begin{remark}\label{remark-whatisleftfromstability}
If~\eqref{mutantlessfit2ineq} holds, then the determinant of the Jacobi matrix $A(n_{1a},n_{1d},n_2)$ is still given by the chain of equalities~\eqref{finaldet3d}, but the right-hand side is negative according to Lemma~\ref{lemma-lessconditions}. On the other hand, $A(n_{1a},n_{1d},n_2)$ still has a negative trace, which can be seen from~\eqref{Arewritten} and the fact that $(n_{1a},n_{1d},n_2)$ has three positive coordinates. It follows that there is either one eigenvalue with negative real part or three such eigenvalues. See Remark~\ref{remark-coexstability} for further details of the stability of this equilibrium. 
\end{remark}

Finally, we prove Lemma~\ref{lemma-somebodyfit}.
\begin{proof}[Proof of Lemma~\ref{lemma-somebodyfit}.]
Let us first verify the statement (i). Under its assumptions, we have
\[ \lambda_2-\mu < \frac{C p\sigma}{\tau(\kappa\mu+\sigma)}(\lambda_2-\mu) + \frac{C}{\tau}(\lambda_1-\lambda_2) > \lambda_1-\mu. \]
Since also by assumption, we have $\lambda_2>\mu$, $C,\tau,p,\sigma>0$, and $\kappa\geq 0$, it follows that
\[ \lambda_2-\mu < \frac{C}{\tau} \Big( \frac{p\sigma}{\kappa\mu+\sigma}(\lambda_2-\mu)+(\lambda_1-\lambda_2) \Big) < \frac{C}{\tau} (\lambda_2-\mu+\lambda_1-\lambda_2) = \frac{C}{\tau} (\lambda_1-\mu). \]
Since $\lambda_2>\mu$, this implies that $\lambda_1>\mu$. 

Let us now prove the assertion (ii). Under its assumptions, the following holds
\[ \lambda_2-\mu > \frac{C p\sigma}{\tau(\kappa\mu+\sigma)}(\lambda_2-\mu) + \frac{C}{\tau}(\lambda_1-\lambda_2) < \lambda_1-\mu, \numberthis\label{firstfirstsecondsecond} \]
further, $\lambda_1>\mu$. Let us now assume for contradiction that $\lambda_2 \leq \mu$. Then, we can estimate
\[ \lambda_1-\mu > \frac{C}{\tau} \Big( \frac{p\sigma}{\kappa\mu+\sigma}(\lambda_2-\mu)+(\lambda_1-\lambda_2) \Big) \geq \frac{C}{\tau} (\lambda_1-\mu) > \lambda_2-\mu. \]
But this contradicts \eqref{firstfirstsecondsecond}, hence the assertion (ii).
\end{proof}

\section{The infinitesimal generator of the population process}\label{sec-generator}
According to Section~\ref{sec-modeldef}, the infinitesimal generator $\widetilde{\mathcal L}$ of the continuous time Markov chain $(\mathbf N_t)_{t \geq 0} = ((N_{1a,t},N_{1d,t},N_{2,t}))_{t \geq 0}$ describing our population process maps bounded continuous functions $f \colon \N_0^3 \to \R$ to $\widetilde{\mathcal L} f \colon \N_0^3 \to \R$ and is given as follows. Let $(x_{1a},x_{1d},x_2) \in \N_0^3$, then we have
\[ 
\begin{aligned}
 \widetilde{\mathcal L} f(x_{1a},x_{1d},x_2) & = (f(x_{1a}+1,x_{1d},x_2)-f(x_{1a},x_{1d},x_2)) \lambda_1 x_{1a} \\ & \qquad + (f(x_{1a}-1,x_{1d},x_2)-f(x_{1a},x_{1d},x_2))(\mu+(1-p)\smfrac{C}{K}(x_{1a}+x_2)) x_{1a} \\
 & \qquad + (f(x_{1a}-1,x_{1d}+1,x_2)-f(x_{1a},x_{1d},x_2)) p\smfrac{C}{K}(x_{1a}+x_2) x_{1a} \\
& \qquad + (f(x_{1a}-1,x_{1d},x_2+1)-f(x_{1a},x_{1d},x_2)) \smfrac{\tau}{K} x_{1a} x_2 \\
& \qquad + (f(x_{1a},x_{1d}-1,x_2)-f(x_{1a},x_{1d},x_2))\kappa\mu x_{1d} \\ & \qquad + (f(x_{1a}+1,x_{1d}-1,x_2)-f(x_{1a},x_{1d},x_2)) \sigma x_{1d} \\
& \qquad + (f(x_{1a},x_{1d},x_2+1)-f(x_{1a},x_{1d},x_2))\lambda_2 x_2 \\
& \qquad + (f(x_{1a},x_{1d},x_2-1)-f(x_{1a},x_{1d},x_2))(\mu+\smfrac{C}{K}(x_{1a}+x_2))x_2. 
\end{aligned}
\]
Note that $\widetilde{\mathcal L}(\cdot)=\mathcal L(K\cdot)$ holds for the infinitesimal generator $\mathcal L$ of $(\mathbf N_t^K)_{t \geq 0}$ used in the proof of Lemmas~\ref{lemma-residentsstayHGTguys} and \ref{lemma-residentsstaydormants}.

\section{Invasion of trait 2 against trait 1: proof of Theorems~\ref{thm-invasionof2} and \ref{thm-failureof2}}\label{sec-HGTproofs}
The case when a trait without a dormant state tries to invade another trait being able to go dormant under competitive pressure has no direct analogue in \cite{BT19}, although various techniques from that paper can be adapted to handle it. We investigate the first, second, and third phase of invasion in Sections~\ref{sec-phase1HGTguys}, \ref{sec-phase2HGTguys}, and \ref{sec-phase2HGTguys}, respectively. As for the first and last phase of invasion, many ideas and techniques of the proof are borrowed from \cite[Section 3.1]{C+19}. 
\subsection{The first phase of invasion: mutant growth or extinction}\label{sec-phase1HGTguys}
In the case when trait 1 is resident and trait 2 is mutant, for $\eps>0$ we define the stopping time
\[ R_\eps^1: = \inf \Big\{ t \geq 0 \colon \max \{ \big| N_{1a,t}^K - \bar n_{1a} \big|, \big| N_{1d,t}^K - \bar n_{1d} \big| \} > \eps \Big\}, \numberthis\label{R1epsdef} \]
which is the first time that the resident population leaves the closed $\ell^\infty$-ball of radius $\eps$ around the equilibrium $(\bar n_{1a},\bar n_{1d})$. Certainly, this stopping time depends on $K$, but we omit this from the notation for simplicity. Then our main result about the first phase of invasion is the following.
\begin{prop}\label{prop-firstphaseHGTguys}
Assume that $\lambda_1>\mu$ and $\widetilde \lambda \neq 0$. Let $K \mapsto m_1^K=(m_{1a}^K,m_{1d}^K)$ be a function from $(0,\infty)$ to $[0,\infty) \times [0,\infty)$ such that $m_1^K \in( \smfrac{1}{K} \N_0) \times (\smfrac{1}{K} \N_0)  $ and $\lim_{K \to \infty} m_1^K=(\bar n_{1a},\bar n_{1d})$. Then there exists a constant $b>0$ and a function $\bar f \colon (0,\infty) \to (0,\infty)$ tending to zero as $\eps \downarrow 0$ such that
\[ \limsup_{K \to \infty} \Big| \P \Big( T_{\sqrt\eps}^2 < T_0^2 \wedge R_{b\sqrt\eps}^1, \Big| \frac{T_{\sqrt\eps}^2}{\log K} - \frac{1}{\widehat \lambda} \Big| \leq \bar f(\eps) \Big| \mathbf N_0^K=\big(m_{1a}^K,m_{1d}^K,\frac{1}{K}\big)  \Big)-(1-q_1) \Big| = o_\eps(1) \numberthis\label{invasiontimeHGTguys} \]
and
\[ \limsup_{K \to \infty} \Big| \P\Big(T_0^2 < T_{\sqrt\eps}^2 \wedge R_{b\sqrt\eps}^1~ \Big|  \mathbf N_0^K=\big(m_{1a}^K,m_{1d}^K,\frac{1}{K}\big)  \Big) - q_1 \Big|=o_{\eps}(1), \numberthis\label{secondofpropHGTguys} \]
where $o_{\eps}(1)$ tends to zero as $\eps \downarrow 0$.
\end{prop}
The assertion of this proposition is analogous to the one of \cite[Proposition 4.1]{BT19}, apart from the fact that it also includes the case when the approximating branching process $(\widehat N_{2}(t))_{t \geq 0}$ is subcritical. Two substantial differences from the setting of \cite{BT19} are that now the resident trait has a dormant state and the mutant trait does not, and that HGT is beneficial for the mutants and harmful for the residents (unlike in the invasion in the opposite direction analysed in Section~\ref{sec-phase1dormants}). We start the proof with the following lemma.
\begin{lemma}\label{lemma-residentsstayHGTguys}
Under the assumptions of Proposition~\ref{prop-firstphaseHGTguys}, there exist two positive constants $b$ and $\eps_0$ such that for any $\eps \in (0,\eps_0)$,
\[ \limsup_{K \to \infty} \P \big(R^1_{b\sqrt\eps} \leq T^2_{\sqrt\eps} \wedge T_0^2 \big)=0. \]
\end{lemma} 
\begin{proof}
We verify this lemma via coupling the rescaled population size $\mathbf N_{1,t}^K=(N_{1a,t}^K,N_{1d,t}^K)$ with two two-type birth-and-death processes, $\mathbf N_{1,t}^1=(N_{1a,t}^1,N_{1d,t}^1)$ and $\mathbf N_{1,t}^2=(N_{1a,t}^2,N_{1d,t}^2)$, on time scales where the mutant population is still small compared to $K$. (The latter processes will also depend on $K$, but we omit the notation $K$ from their nomenclature for simplicity.) To be more precise, similarly to \cite[Section 3.1.2]{C+19}, our goal is to choose $(\mathbf N_{1,t}^1)_{t \geq 0}$ and $(\mathbf N_{1,t}^2)_{t \geq 0}$ so that
\[ N_{1\upsilon,t}^1 \leq N_{1\upsilon,t}^K \leq N_{1\upsilon,t}^2, \qquad \text{a.s.} \qquad \forall t \leq T_0^2 \wedge T_{\sqrt\eps}^2, \qquad \forall \upsilon \in \{ a, d \}. \numberthis\label{firstresidentcouplingHGTguys} \]

In order to satisfy \eqref{firstresidentcouplingHGTguys}, for all sufficiently small $\eps>0$, the processes $(\mathbf N_{1,t}^1)_{t \geq 0}$ and $(\mathbf N_{1,t}^2)_{t \geq 0}$ can be chosen with the following birth and death rates
\begin{align*}
\mathbf N_{1,t}^1 \colon  & \big( \frac{i}{K}, \frac{j}{K} \big)  \to \big( \frac{i+1}{K}, \frac{j}{K} \big) & \text{at rate } & i\lambda_1, \\
 & \big( \frac{i}{K}, \frac{j}{K} \big)  \to  \big( \frac{i-1}{K}, \frac{j}{K} \big) & \text{at rate } &i \big( \mu+C(1-p) \frac{i}{K} + (C+\tau) \sqrt\eps \big), \\
& \big( \frac{i}{K}, \frac{j}{K} \big) \to  \big( \frac{i-1}{K},  \frac{j+1}{K} \big) & \text{at rate }  &i Cp \frac{i}{K}, \\
& \big( \frac{i}{K}, \frac{j}{K} \big) \to  \big( \frac{i}{K},  \frac{j-1}{K} \big) & \text{at rate }  &j\kappa\mu, \\
& \big( \frac{i}{K}, \frac{j}{K} \big) \to  \big( \frac{i+1}{K},  \frac{j-1}{K} \big) & \text{at rate }  &j\sigma, \\
\end{align*}
and
\begin{align*}
\mathbf N_{1,t}^2 \colon  & \big( \frac{i}{K}, \frac{j}{K} \big)  \to \big( \frac{i+1}{K}, \frac{j}{K} \big) & \text{at rate } & i\lambda_1, \\
 & \big( \frac{i}{K}, \frac{j}{K} \big)  \to  \big( \frac{i-1}{K}, \frac{j}{K} \big) & \text{at rate } &i \big( \mu+C(1-p) \big( \frac{i}{K}-\sqrt\eps\big) \big), \\
& \big( \frac{i}{K}, \frac{j}{K} \big) \to  \big( \frac{i-1}{K},  \frac{j+1}{K} \big) & \text{at rate }  &i Cp \big( \frac{i}{K} + \sqrt\eps \big), \\
& \big( \frac{i}{K}, \frac{j}{K} \big) \to  \big( \frac{i}{K},  \frac{j-1}{K} \big) & \text{at rate }  &j\kappa\mu, \\
& \big( \frac{i}{K}, \frac{j}{K} \big) \to  \big( \frac{i+1}{K},  \frac{j-1}{K} \big) & \text{at rate }  &j\sigma. \\
\end{align*}
Informally speaking, the coupling \eqref{firstresidentcouplingHGTguys} holds thanks to the fact that for branching processes having the same kind of transitions as the trait 1 population in the frequency process defined in \ref{sec-modeldef}, competition-induced switching to dormancy is more favourable for an active individual than immediate death by competition, but not better, and for $\kappa>0$ strictly worse, than not being hit by a competitive event at all, further, HGT is unfavourable; see also \cite[Section 4.1]{BT19} for further details.

Let us estimate the time until which the processes $\mathbf N_1^1$ and $\mathbf N_1^2$ stay close to the value $(\bar n_{1a},\bar n_{1d})$ (in $\ell^\infty$-norm for simplicity). We define the stopping times
\[ R^{1,i}_{\eps}:=\inf \big\{ t \geq 0 \colon N^i_{1a,t} \notin [\bar n_{1a}-\eps,\bar n_{1a}+\eps] \text{ or } N^i_{1d,t} \notin [\bar n_{1d}-\eps,\bar n_{1d}+\eps] \big\}, \qquad i \in \{1,2\},~\eps>0. \]
As $K \to \infty$, according to \cite[Theorem 2.1, p.~456]{EK}, uniformly on any fixed time interval of the form $[0,T]$, $T>0$, $\mathbf N_{1,t}^{1}$ converges in probability to the unique solution to
\[
\begin{aligned}
\dot n_{1a,1}(t) & = n_{1a,1}(t) (\lambda_1-\mu-C n_{1a,1}(t)-(C+\tau) \sqrt\eps)+\sigma n_{1d,1}(t), \\
\dot n_{1d,1}(t) & = pCn_{1a,1}(t)^2-(\kappa\mu+\sigma)n_{1d,1}(t),
\end{aligned} \]
given that the initial conditions converge in probability to the initial condition of the limiting dynamical system.
Similarly, for large $K$, the dynamics of $\mathbf N_{1,t}^2$ is close to the one of the unique solution to
\[
\begin{aligned}
\dot n_{1a,2}(t) & = n_{1a,2}(t) (\lambda_1-\mu-C( n_{1a,2}(t)-\sqrt\eps))+\sigma n_{1d,2}(t), \\
\dot n_{1d,2}(t) & = pCn_{1a,2}(t)(n_{1a,2}(t)+\sqrt\eps)-(\kappa\mu+\sigma)n_{1d,2}(t).
\end{aligned} \]
The equilibria of the system of ODEs $(n_{1a,i}(t),n_{1d,i}(t))$, $i \in \{ 1, 2 \}$, are $(0,0)$ and respectively an equilibrium of the form $(\bar n_{1a}+o_\eps(1), \bar n_{1d}+o_\eps(1))$, which we denote by $(\bar n_{1a}^{i,\eps},\bar n_{1d}^{i,\eps})$. 
For $\eps>0$ small enough, the equilibrium $(0,0)$ is unstable and the one $(\bar n_{1a}^{i,\eps},\bar n_{1d}^{i,\eps})$ is asymptotically stable for all $i \in \{1,2\}$. Further, using a straightforward adaptation of the proof of \cite[Lemma 4.6]{BT19}, one can verify that for sufficiently small $\eps>0$, for any coordinatewise strictly positive initial condition, we have 
\[ \lim_{t \to \infty} \big( n_{1a,i}(t), n_{1d,i}(t) \big) = (\bar n_{1a}^{i,\eps},\bar n_{1d}^{i,\eps}), \qquad \forall i \in \{1,2\}. \]
This implies that there exists $\eps_0>0$ and $b \geq 2$ such that for all $0 <\eps \leq \eps_0$ and for all $i \in \{1,2\}$ and $j \in \{ 1a,1d\}$,
\[ \big|\bar n_{j}-\bar n_{j}^{i,\eps}\big| \leq (b-1) \sqrt\eps,  \qquad \text{ and } \qquad 0 \notin [\bar n_{j}-b\sqrt\eps, \bar n_{j}+b\sqrt\eps]. \numberthis\label{Bcond} \] 
Now, thanks to a result about exit of jump processes from a domain by Freidlin and Wentzell \cite[Chapter 5]{FW84} (see~\cite[Section 4.2]{C06} for more details in a very similar situation), 
there exists a family (over $K$) of Markov jump processes $\widetilde {\mathbf N}^1_1=(\widetilde {\mathbf N}^1_{1,t})_{t \geq 0}=(\widetilde N^1_{1a,t},\widetilde N^1_{1d,t})_{t \geq 0}$ with positive, bounded, Lipschitz continuous transition rates that are uniformly bounded away from 0 such that for
\[ \widetilde R^{1,i}_{\eps}:=\inf \big\{ t \geq 0 \colon \widetilde N^i_{1a,t} \notin [\bar n_{1a}-\eps,\bar n_{1a}+\eps] \text{ or } \widetilde N^i_{1d,t} \notin [\bar n_{1d}-\eps,\bar n_{1d}+\eps] \big\}, \qquad i \in \{1,2\},~\eps>0, \]
there exists $V>0$ such that
\[ \P(R^{1,1}_{b\sqrt\eps}>\e^{KV}) = \P(\widetilde R^{1,1}_{b\sqrt\eps}>\e^{KV}) \underset{K \to \infty}{\longrightarrow} 0. \numberthis\label{FW1HGTguys} \] 
Using similar arguments for $\mathbf N_{1}^2$, we derive that for $\eps>0$, $V>0$ small enough, we have that
\[ \P(R^1_{b\sqrt\eps}>\e^{KV})\underset{K \to \infty}{\longrightarrow} 0.  \numberthis\label{FW2HGTguys} \]
Now, on the event $\{ R_{b\sqrt\eps}^1 \leq T_0^2 \wedge T_{\sqrt\eps}^2 \}$ we have $R_{b\sqrt\eps}^1\geq R_{b\sqrt\eps}^{1,1} \wedge R_{b\sqrt\eps}^{1,2}$. Using \eqref{FW1HGTguys} and \eqref{FW2HGTguys}, we derive that
\[ \limsup_{K \to \infty} \P \big( R_{b\sqrt\eps}^1 \leq \e^{KV}, R_{b\sqrt\eps}^1 \leq T_0^2 \wedge T_{\sqrt\eps}^2 \big) = 0. \]
Moreover, using Markov's inequality,
\begin{align*}
    \P(R_{b\sqrt\eps}^1 \leq T_0^2 \wedge T_{\sqrt\eps}^2) & \leq \P \big( R_{b\sqrt\eps}^1 \leq \e^{KV}, R_{b\sqrt\eps}^1 \leq T_0^2 \wedge T_{\sqrt\eps}^2 \big) + \P(R_{b\sqrt\eps}^1 \wedge T_0^2 \wedge T_{\sqrt\eps}^2 \geq \e^{KV} \big) \\ &  \leq \P \big( R_{b\sqrt\eps}^1 \leq \e^{KV}, R_{b\sqrt\eps}^1 \leq T_0^2 \wedge T_{\sqrt\eps}^2 \big) + \e^{-KV} \E(R_{b\sqrt\eps}^1 \wedge T_0^2 \wedge T_{\sqrt\eps}^2 ). 
\end{align*}
Since we have
\[ \E \big[R_{b\sqrt\eps}^1\wedge T_0^2 \wedge T_{\sqrt\eps}^2 \big] \leq \E \Big[\int_0^{R_{b\sqrt\eps}^1\wedge T_0^2 \wedge T_{\sqrt\eps}^2} K N_{2,t}^K \d t \Big], \]
it suffices to show that there exists $\widetilde C>0$ such that 
\[ \E \Big[\int_0^{R_{b\sqrt\eps}^1\wedge T_0^2 \wedge T_{\sqrt\eps}^2} K N_{2,t}^K \d t \Big] \leq \widetilde C \sqrt\eps K. \numberthis\label{expectedstoppingtimeHGTguys} \]
This can be done similarly to \cite[Section 3.1.2]{C+19} (we note that a direct analogue of this argument cannot be found in \cite{BT19}), whereas the case that we have to handle is simpler since we consider a one-type mutant population instead of a two-type one. We claim that it is enough to show that there exists a function $g \colon (\smfrac{1}{K} \N_0)^3 \to \R$ defined as
\[ g(n_{1a},n_{1d},n_{2})=\gamma n_2 \numberthis\label{gammaHGTguys} \]
for a suitably chosen $\gamma \in \R$, such that
\[ \Lcal g(\mathbf N_{t}^K) \geq N_{2,t}^K \numberthis\label{largegeneratorHGTguys}\]
where $\Lcal$ is the infinitesimal generator of $(\mathbf N_t^K)_{t \geq 0}$ (cf.~Appendix~\ref{sec-generator}).
Indeed, if \eqref{largegeneratorHGTguys} holds, then thanks to Dynkin's formula we have
\[ 
\begin{aligned}
\E \Big[\int_0^{R_{b\sqrt\eps}^1 \wedge T_0^2 \wedge T_{\sqrt\eps}^2} K N_{2,t}^K \d t \Big] & \leq  \E \Big[\int_0^{R_{b\sqrt\eps}^1 \wedge T_0^2 \wedge T_{\sqrt\eps}^2} K\Lcal g(\mathbf N_{t}^K)\d t \Big] =\E\big[ K g(\mathbf N_{R_{b\sqrt\eps}^1 \wedge T_0^2 \wedge T_{\sqrt\eps}^2}^K)-Kg(\mathbf N_{0}^K) \big] \\
& \leq |\gamma| (\sqrt\eps K-1),
\end{aligned}
\]
which implies the existence of $b,\eps>0$ such that \eqref{expectedstoppingtimeHGTguys} holds for all $\eps>0$ small enough. Here, Dynkin's formula can indeed be applied because $ \E [R_{b\sqrt\eps}^1\wedge T_0^2 \wedge T_{\sqrt\eps}^2]$ is finite. That holds because given our initial conditions, with positive probability the single initial trait 2 individual dies due to natural death within a unit length of time before any event of the process $\mathbf N_t^K$ occurs, and hence already $T_0^2$ is stochastically dominated by a geometric random variable, which has all moments.
Now, we have
\[ \mathcal L g(\mathbf N_t^K) = N_{2,t}^K \gamma \big( \lambda_2-\mu-C N_{2,t}^K - (C-\tau)  N_{1a,t}^K \big). \]
As long as $t \leq R_{b\sqrt\eps}^1 \wedge T_0^2 \wedge T_{\sqrt\eps}^2$, we have that
\[  \gamma\big| \big( \lambda_2-\mu - C N_{2,t}^K - (C-\tau)  N_{1a,t}^K \big)- \big(\lambda_2-\mu- (C-\tau) \bar n_{1a} \big) \big| \leq b (C+\tau) \sqrt\eps+ C \sqrt\eps=o_\eps(1). \]
As we have seen in Section~\ref{sec-phase13}, the sign of $\lambda_2-\mu-(C-\tau) \bar n_{1a} $ is positive (respectively 0 respectively negative) if and only if the approximating branching process $(\widehat N_2(t))_{t \geq 0}$ is supercritical (respectively critical respectively subcritical). Further, under the assumption that $\widetilde \lambda \neq 0$, the branching process is not critical (cf.~Remark~\ref{remark-simplifyconditions}). Hence, if it is supercritical (respectively subcritical), one can choose $\gamma>0$ (respectively $\gamma<0$) such that  $\gamma (\lambda_2-\mu-(C-\tau) \bar n_{1a})>1$. For all sufficiently small $\eps>0$, this implies the existence of $\gamma$ such that the corresponding $g$ satisfies \eqref{largegeneratorHGTguys}.
\end{proof}
\begin{proof}[Proof of Proposition~\ref{prop-firstphaseHGTguys}]
Now, we consider our population process on the event
\[ A_\eps := \{ T_0^2 \wedge T_{\sqrt\eps}^2 < R_{b\sqrt\eps}^1 \} \]
for sufficiently small $\eps>0$. In this event, the invasion or extinction of the mutant population will happen before the resident population leaves a small neighbourhood of its two-coordinate equilibrium $(\bar n_{1a},\bar n_{1d})$. On $A_\eps$ we couple the process $KN_{2,t}^K$ with two branching processes $N_{2,t}^{\eps,1}$ and $N_{2,t}^{\eps,2}$ on $\N_0$ (which also depend on $K$, but we omit this from the notation for readability) such that almost surely, for any $t <  t_\eps:=T_0^2 \wedge T_{\sqrt\eps}^2 \wedge R_{b\sqrt\eps}^1$,
\[ \begin{aligned}
N_{2,t}^{\eps,1} & \leq \widehat N_{2}(t) \leq N_{2,t}^{\eps,2}, \\
N_{2,t}^{\eps,1} & \leq K N_{2,t}^K \leq N_{2,t}^{\eps,2},
\end{aligned} \numberthis\label{mutantcouplingHGTguys}
\]
where we again recall the approximating branching process $\widehat N_{2}(t)$ defined in Section~\ref{sec-phase13}.
We claim that in order to satisfy \eqref{mutantcouplingHGTguys}, these processes can be defined as follows:
\[ \begin{aligned}
N_{2,t}^{\eps,1} \colon \quad & i \to i+1 & \text{at rate} & \quad i (\lambda_2+\tau (\bar n_{1a}-b\sqrt\eps )), \\
& i \to i-1 & \text{at rate}  &\quad i(\mu+C(\sqrt\eps+\bar n_{1a}+b\sqrt\eps)),
\end{aligned}
\]
and
\[ \begin{aligned}
N_{2,t}^{\eps,2} \colon \quad & i \to i+1 & \text{at rate} & \quad i (\lambda_2+\tau (\bar n_{1a}+b\sqrt\eps )), \\
& i \to i-1 & \text{at rate}  &\quad i(\mu+C(\bar n_{1a}-b\sqrt\eps)).
\end{aligned}
\]
Informally speaking, the coupling \eqref{mutantcouplingHGTguys} holds because in order to increase (decrease) a branching process of a population that can only play the role of transmitter in HGT, one needs to decrease (increase) its competitive death rate and increase (decrease) its HGT rate. 

For $j \in \{ 1,2\}$, let $q_2^{(\eps,j)}$ denote the extinction probability of the process $N_{2,t}^{\eps,j}$ started from $N_{2,0}^{\eps,j}=1$. The extinction probability of a branching process of a HGT-transmitting population is continuous with respect to the competitive event and HGT rates of the process, apart from the point where the branching process is critical. Further, the extinction probability increases with the rate of death by competition and decreases with the HGT rate. These assertions are proven in \cite[Sections A.3]{C+19}, where we note that HGT for the trait 2 population is just a special case of birth.

Hence, it follows from the first line of \eqref{mutantcouplingHGTguys} that for fixed $\eps>0$,
\[ q_2^{(\eps,2)} \leq q_2 \leq q_2^{(\eps,1)} \]
and for $j \in \{ 1,2\}$,
\[ 0 \leq \liminf_{\eps \downarrow 0} \big| q_2^{(\eps,j)}-q_2 \big| \leq \limsup _{\eps \downarrow 0} \big| q_2^{(\eps,j)}-q_2 \big| \leq \limsup_{\eps \downarrow 0} \big| q_2^{(\eps,1)}-q_2^{(\eps,2)} \big| = 0, \numberthis\label{qineqHGTguys} \]
where we recall the extinction probability $q_2$ of the approximating branching process $(\widehat N_{2}(t))_{t \geq 0}$ defined in \eqref{q2def}.  

Next, we show that the probabilities of extinction and invasion of the actual process $N_{2,t}^K$ also converge to $q_2$ and $1-q_2$, respectively, with high probability as $K \to \infty$. We define the stopping times, for $j \in \{ 1,2 \}$,
\[ T_x^{(\eps,j),2} := \inf \{ t >0 \colon N^{\eps,j}_{2,t} = \lfloor K x \rfloor \}, \qquad x \in \R. \]
Using the coupling in the second line of \eqref{mutantcouplingHGTguys}, which is valid on $A_\eps$, we have
\[ \P\big(T_{\sqrt\eps}^{(\eps,1),2} \leq T_0^{(\eps,1),2}, A_\eps\big) \leq \P\big(T_{\sqrt\eps}^{2} \leq T_0^{2},A_\eps \big) \leq  \P\big(T_{\sqrt\eps}^{(\eps,2),2} \leq T_0^{(\eps,2),2}, A_\eps\big). \numberthis\label{couplednonextinctionHGTguys} \] 
Indeed, if a process reaches the size $K\sqrt\eps$ before dying out, then the same holds for a larger process as well. However, $A_\eps$ is independent of $(N_{2,t}^{\eps,j})_{t \geq 0}$ for both $j = 1$ and $j = 2$, and thus
\[ \liminf_{K \to \infty} \P\big( T_{\sqrt\eps}^{(\eps,1),2}\leq T_0^{(\eps,1),2}, A_\eps \big)=\liminf_{K \to \infty}   \P(A_\eps)\P\big( T_{\sqrt\eps}^{(\eps,1),2}  \leq T_{0}^{(\eps,1),2} \big) \geq (1-q_2^{(\eps,1)})(1-o_\eps(1)) \numberthis\label{eps-LB-HGTguys} \]
and
\[ \limsup_{K \to \infty}  \P\big( T_{\sqrt\eps}^{(\eps,2),2}\leq T_0^{(\eps,2),2}, A_\eps \big) =\limsup_{K \to \infty}   \P(A_\eps) \P\big( T_{\sqrt\eps}^{(\eps,2),2} \leq T_{0}^{(\eps,2),2} \big) \leq 1-q_2^{(\eps,2)}(1+o_\eps(1)). \numberthis\label{eps+UB-HGTguys} \]
Letting $K \to \infty$ in \eqref{couplednonextinctionHGTguys} and applying \eqref{eps-LB-HGTguys} and \eqref{eps+UB-HGTguys} yields that 
\[
\begin{aligned}
(1-q_2^{(\eps,1)})(1-o_\eps(1)) & \leq \liminf_{K \to \infty} \P\big( T_{\sqrt\eps}^{(\eps,1),2} \leq T_0^{(\eps,1),2}, A_\eps \big) \\ & \leq \limsup_{K \to \infty}  \P\big( T_{\sqrt\eps}^{(\eps,2),2}\leq T_0^{(\eps,2),2}, A_\eps \big) \\ &\leq  (1-q_2^{(\eps,2)})(1+o_\eps(1)).
\end{aligned} \]
Hence,
\[ \limsup_{K \to \infty} \big| \P(T_{\sqrt\eps}^{2} \leq T_0^{2},A_\eps )-(1-q_2) \big|=o_\eps(1), \]
as required. The equation \eqref{secondofpropHGTguys} can be derived similarly. 

Finally, we show that in the case of invasion (which happens with probability tending to $1-q_2$) the time before the mutant population reaches size $K \sqrt\eps$ is of order $\log K/\widehat \lambda$, where we recall that $\widehat \lambda$ was defined in \eqref{lambdahatdef} as the birth rate minus the death rate of the approximating branching process $\widehat N_{2}(t)$. Having \eqref{secondofpropHGTguys}, we can without loss of generality assume that $q_2<1$, which is equivalent to the assumption that \eqref{secondinmutantlessfit2ineq} holds.

For $j \in \{ 1,2 \}$, let $\widehat\lambda^{(\eps,j)}$ denote the difference of the birth rate and the death rate of the process $N_{2,t}^{\eps,j}$. This difference is positive for all sufficiently small $\eps>0$ and converges to $\widehat \lambda$ as $\eps \downarrow 0$. In other words, there exists a nonnegative function $\bar f \colon (0,\infty) \to (0,\infty)$ with $\lim_{\eps \downarrow 0} \bar f(\eps)=0$ such that for any $\eps>0$ small enough,
\[ \Big| \frac{\widehat\lambda^{(\eps,j)}}{\widehat\lambda}-1\Big| \leq \frac{\bar f(\eps)}{2}. \numberthis\label{eigenvaluesclose}\]
Let us fix $\eps$ small enough such that \eqref{eigenvaluesclose} holds. Then from the second line of \eqref{mutantcouplingHGTguys} we deduce that
\[ \P \Big( T^{(\eps,1),2}_{\sqrt\eps} \leq T^{(\eps,1),2}_{0} \wedge \frac{\log K}{\widehat \lambda} (1+\bar f(\eps)),A_\eps \Big) \leq \P \Big( T^{2}_{\sqrt\eps} \leq T^{2}_{0} \wedge \frac{\log K}{\widehat \lambda} (1+\bar f(\eps)),A_\eps \Big). \]
Using this together with the independence between $A_\eps$ and $(N_{2,t}^{\eps,j})_{t \geq 0}$ and employing \cite[Section 7.5]{AN72}, we obtain for $\eps>0$ small enough (in particular such that $\bar f(\eps)<1$)
\[
\begin{aligned}
& \liminf_{K \to \infty} \P \Big( T^{(\eps,1),2}_{\sqrt\eps} \leq T^{(\eps,1),2}_{0} \wedge \frac{\log K}{\widehat \lambda} (1+\bar f(\eps)),A_\eps \Big) \\
\geq & \liminf_{K \to \infty} \P \Big( T^{(\eps,1),2}_{\sqrt\eps} \leq \frac{\log K}{\widehat \lambda} (1+\bar f(\eps)) \Big) \P(A_\eps) \\
\geq & \liminf_{K \to \infty} \P \Big( T^{(\eps,1),2}_{\sqrt\eps} \leq \frac{\log K}{\widehat \lambda^{(\eps,1)}} \big( 1-\frac{\bar f(\eps)}{2} \big) (1+\bar f(\eps))  \Big) \P(A_\eps) \\
\geq &  \liminf_{K \to \infty} \P \Big( T^{(\eps,1),2}_{\sqrt\eps} \leq \frac{\log K}{\widehat \lambda^{(\eps,1)}} \Big) \P(A_\eps) \\
\geq & (1-q_2^{(\eps,1)})(1-o_\eps(1)). 
\end{aligned} \numberthis\label{festimatesHGTguys}
\]
Similarly, using the second line of \eqref{mutantcouplingHGTguys}, we derive that for all sufficiently small $\eps>0$ 
\begin{align*}
    \P \Big( T^{(\eps,2),2}_{\sqrt\eps} \leq T^{(\eps,2),2}_{0} \wedge \frac{\log K}{\widehat \lambda} (1-\bar f(\eps)),A_\eps \Big) \geq \P \Big( T^{2}_{\sqrt\eps} \leq T^{2}_{0} \wedge \frac{\log K}{\widehat \lambda} (1-\bar f(\eps)),A_\eps \Big),
\end{align*}
and arguments analogous to the ones used in \eqref{festimatesHGTguys} imply that
\[ \liminf_{K \to \infty} \P \Big( T^{(\eps,2),2}_{\sqrt\eps} \leq T^{(\eps,2),2}_{0}, T^{(\eps,2),2}_{\sqrt\eps} \geq \frac{\log K}{\widehat \lambda} (1-\bar f(\eps)),A_\eps \Big) \geq (1-q_2^{(\eps,2)})(1+o_\eps(1)). \]
These together imply \eqref{invasiontimeHGTguys}, which concludes the proof of the proposition.
\end{proof}

\subsection{The second phase of invasion: Lotka--Volterra phase}\label{sec-phase2HGTguys}
Assume that $q_2<1$, in other words, \eqref{secondinmutantlessfit2ineq} holds. Then, for large $K$, with probability close to $1-q_2$, after the end of the first phase, the rescaled population process $\mathbf N_t^K=(N_{1a,t}^K,N_{1d,t}^K,N_{2,t}^K)$ is contained in the set 
\[ \mathfrak B^3_\eps=[\bar n_{1a}-b\sqrt\eps,\bar n_{1a}+b\sqrt\eps] \times [\bar n_{1d}-b\sqrt\eps, \bar n_{1d}+b\sqrt\eps] \times [\smfrac{\sqrt \eps}{2}, \sqrt \eps], \numberthis\label{BHGTguys} \] given that $K$ is so large that $ \smfrac{\lfloor \sqrt \eps K \rfloor}{K} \geq \smfrac{\sqrt \eps}{2} $, where we recall the constant $b$ from Proposition~\ref{prop-firstphaseHGTguys}. Starting from this point in time, as long as all the three sub-populations are comparable to $K$, the process $\mathbf N_t^K$ can well be approximated by the solution of the limiting dynamical system \eqref{3dimHGT}, according to \cite[Theorem 2.1, p.~456]{EK}. More precisely, the process converges in probability uniformly over compact time intervals given convergence of initial conditions in probability to the limiting initial condition, analogously to the proof of Lemma~\ref{lemma-residentsstayHGTguys}. Now, we have two different cases to treat. Namely, we want to show that starting from such an initial condition, if \eqref{firstinmutantlessfit2ineq} holds, then the dynamical system converges (as $t \to \infty$) to the coexistence equilibrium $(n_{1a},n_{1d},n_2)$, whereas if \eqref{reversefirstin} holds, then it converges to the one-type equilibrium $(0,0,\bar n_2)$. In the latter case, we will additionally have a third phase of invasion where trait 2 reaches fixation, which will be analysed in Section~\ref{sec-phase3HGTguys}. Note that such an assertion extends the statement of Proposition~\ref{prop-stability3d} about local stability of equilibria into a statement about global stability.

We proceed with the second phase of invasion in the present section as follows. We first verify Lemma~\ref{lemma-effectivecomp}, which investigates the competitive pressure felt by the type 1 and type 2 active population near the equilibria with at least one positive coordinate. Based on this, we prove Proposition~\ref{prop-convergenceHGTguys}, which asserts the convergence of the dynamical system towards the corresponding stable equilibrium as $t \to \infty$. Let us introduce $\widetilde n_2=\frac{\lambda_2-\mu}{C}$, which is equal to $\bar n_2$ if $\lambda_2 \geq \mu$. Note that $(0,0,\widetilde n_2)$ is always an equilibrium of \eqref{3dimHGT}, although not necessarily a coordinatewise nonnegative one. Recall further that $(n_{1a},n_{1d},n_2)$ denotes the coexistence equilibrium defined in \eqref{coexeqoncemore}.

\begin{proof}[Proof of Lemma~\ref{lemma-effectivecomp}.]
Assertion \eqref{n2largeduetoHGT} can be verified via turning the inequality $C\widetilde n_2>(C-\tau)\bar n_{1a}$ into \eqref{secondinmutantlessfit2ineq} via elementary transformations, using the definitions of $\bar n_{1a}$ (cf.~\eqref{n1abardef}) and $\widetilde n_2$, the assumption that $\lambda_1>\mu$, and the facts that $\tau,C,\sigma>0$, $p \in (0,1)$, and $\kappa \geq 0$. We leave the details for the reader.  

As for assertion~\eqref{compeq}, note that thanks to Lemma~\ref{lemma-lessconditions}, if $(n_{1a},n_{1d},n_2)$ exists as a coordinatewise positive equilibrium, then $Cp\sigma \neq (\kappa\mu+\sigma)\tau$. Now, using \eqref{coexeq}, we compute
\[
\begin{aligned}
C&(n_{1a}+n_2)-\tau n_{1a} \\& = \frac{C(\kappa\mu+\sigma)\tau(\lambda_2-\lambda_1)-\tau C(\kappa\mu+\sigma)(\lambda_2-\lambda_1)-\tau Cp\sigma(\mu-\lambda_2)-(\kappa\mu+\sigma)\tau^2 (\lambda_2-\mu)}{\tau (Cp\sigma-(\kappa\mu+\sigma)\tau)} \\ & = \frac{Cp\sigma(\lambda_2-\mu)+(\kappa\mu+\sigma)\tau(\lambda_2-\mu)}{Cp\sigma-(\kappa\mu+\sigma)\tau} = \lambda_2-\mu = C \widetilde n_2,
\end{aligned} \]
as required.

Assertion~\eqref{coexcompineq} can be verified similarly, using again that in view of Lemma~\ref{lemma-lessconditions}, \eqref{firstinmutantlessfit2ineq} together with \eqref{secondinmutantlessfit2ineq} implies $Cp\sigma < (\kappa\mu+\sigma)\tau$. Using \eqref{coexeq} again, we obtain
\[ 
\begin{aligned}
C&(n_{1a}+n_2) + \tau n_2 \\& = \frac{C(\kappa\mu+\sigma)\tau(\lambda_2-\lambda_1)-C(\kappa\mu+\sigma)(\lambda_2-\lambda_1)\tau+Cp\sigma(\lambda_2-\mu)\tau-(\kappa\mu+\sigma)\tau(\lambda_1-\mu)}{\tau(Cp\sigma-(\kappa\mu+\sigma)\tau)} \\ & = \frac{Cp\sigma(\lambda_2-\mu)-(\kappa\mu+\sigma)(\lambda_1-\mu)}{Cp\sigma-(\kappa\mu+\sigma)} \geq \frac{Cp\sigma(\lambda_1-\mu)-(\kappa\mu+\sigma)(\lambda_1-\mu)}{Cp\sigma-(\kappa\mu+\sigma)} = \lambda_1-\mu,
\end{aligned}
\]
as wanted.

Next, we prove assertion~\eqref{trait2compineq}. Under the assumptions of the assertion, using that~\eqref{reversefirstin} holds, we obtain
\[ \lambda_2-\mu > \frac{Cp\sigma}{\tau(\kappa\mu+\sigma)}(\lambda_2-\mu) + \frac{C}{\tau} (\lambda_1-\mu) - \frac{C}{\tau} (\lambda_2-\mu). \]
In other words,
\[ (\lambda_2 - \mu)\Big( 1 + \frac{C}{\tau}\frac{\kappa\mu+(1-p)\sigma}{\kappa\mu+\sigma} \Big) > \frac{C}{\tau} (\lambda_1-\mu). \numberthis\label{abitlessisstillabitmore}\]
Thanks to part~\eqref{if1fit2fit} of Lemma~\ref{lemma-somebodyfit}, we have $\lambda_2>\mu$ under the assumptions of the assertion, and hence \eqref{abitlessisstillabitmore} implies
\[ (\lambda_2-\mu) (\tau+C) > C(\lambda_1-\mu), \]
and hence assertion~\eqref{trait2compineq} follows. 

Finally, let us consider assertion~\eqref{trait2coexcompineq}. If $\lambda_2 \leq \mu$, then there is nothing to show. Else, the proof can be carried out analogously to the one of assertion~\eqref{trait2compineq}. 
\end{proof}
Next, we state and prove Proposition~\ref{prop-convergenceHGTguys}.
\begin{prop}\label{prop-convergenceHGTguys}
Assume that the initial condition $(n_{1a}(0),n_{1d}(0),n_{2}(0))$ of the dynamical system \eqref{3dimHGT} is contained in $\mathfrak B^3_\eps$. Then, given that $\eps>0$ is sufficiently small, the following assertions hold about the solution $t \mapsto (n_{1a}(t),n_{1d}(t),n_2(t))$ of \eqref{3dimHGT}.
\begin{enumerate}
    \item If \eqref{reversefirstin} and \eqref{secondinmutantlessfit2ineq} hold, then
    \[ \lim_{t \to \infty} (n_{1a}(t),n_{1d}(t),n_2(t)) =(0,0,\bar n_2), \numberthis\label{fixationof2} \]
    \item whereas if \eqref{firstinmutantlessfit2ineq} and \eqref{secondinmutantlessfit2ineq} hold, then 
    \[ \lim_{t \to \infty} (n_{1a}(t),n_{1d}(t),n_2(t)) = (n_{1a},n_{1d},n_2). \numberthis\label{convergencetocoexistence} \]
\end{enumerate}
\end{prop}
\begin{proof}
Assume throughout the proof that \eqref{secondinmutantlessfit2ineq} holds. Let $\eps>0$ be such that \eqref{Bcond} holds and $\lambda_2-\mu-C \bar n_{1a}(t)-(2+2(C+\tau)b)\sqrt\eps>0$; such $\eps$ exists thanks to part~\eqref{n2largeduetoHGT} of Lemma~\ref{lemma-effectivecomp}. At $t=0$, we have
\[ n_2(t) (\lambda_2-\mu-C(n_{1a}(t)+n_2(t)) + \tau n_{1a}(t)) \geq \frac{\sqrt\eps}{2} (C\widetilde n_2-(C-\tau)\bar n_{1a}-(1+(C+\tau) b)\sqrt\eps)>0. \]
Hence, by continuity of $t \mapsto \bar n_{2}(t)$, for a positive amount of time after $t=0$, $n_2(t)$ will be increasing. More precisely, $n_2(t)$ increases as long as $C(n_{1a}(t)+n_2(t))-\tau n_{1a}(t)<\lambda_2-\mu$ (cf.~also part~\eqref{compeq} of Lemma~\ref{lemma-effectivecomp}). Our goal is now to prove that in case \eqref{firstinmutantlessfit2ineq}, this implies \eqref{convergencetocoexistence}, whereas in case \eqref{reversefirstin} holds, it implies \eqref{fixationof2}. We verify these assertions via approximating the partial dynamics of $t \mapsto (n_{1a}(t),n_{1d}(t))$ by the system where  
\[ C(n_{1a}(t)+n_2(t))-\tau n_{1a}(t)=\lambda_2-\mu \numberthis\label{artificialcomp} \]
is substituted into the first two coordinates of \eqref{3dimHGT} and then performing a perturbation argument. If \eqref{artificialcomp} holds, then we have
\[  
\begin{aligned}
C&(n_{1a}(t)+n_2(t))+\tau n_{2}(t) =C(n_{1a}(t)+n_2(t))-\tau n_{1a}(t)+\tau(n_{1a}(t)+n_2(t))  \\ &=\lambda_2-\mu+\tau(n_{1a}(t)+n_2(t))= \lambda_2-\mu + \frac{\tau}{C} \Big( C(n_{1a}(t)+n_2(t))-\tau n_{1a}(t)\Big)+\frac{\tau^2}{C} n_{1a}(t) \\ & =\big( 1+\frac{\tau}{C} \Big) (\lambda_2-\mu)+\frac{\tau^2}{C} n_{1a}(t),
\end{aligned} \]
where the right-hand side does not depend on $n_2(t)$; further,
\[
p C n_{1a}(t)(n_{1a}(t)+n_2(t)) = p n_{1a}(t)\big( C(n_{1a}(t)+n_2(t))-\tau n_{1a}(t) + \tau n_{1a}(t) \big) = p n_{1a}(t) (\lambda_2-\mu+\tau n_{1a}(t)),
\]
where the right-hand side is also independent of $n_2(t)$. 
Thus, we shall first investigate the two-dimensional system of ODEs
\[
\begin{aligned}
 \dot n_{1a}(t) & = n_{1a}(t)\big(\lambda_1-\lambda_2-\frac{\tau}{C}(\lambda_2-\mu)-\frac{\tau^2}{C}n_{1a}(t)\big)+\sigma n_{1d}(t), \\
 \dot n_{1d}(t) & = n_{1a}(t) p(\lambda_2-\mu+\tau n_{1a}(t)) - (\kappa\mu+\sigma)n_{1d}(t).
\end{aligned}
\numberthis\label{2dimHGTapprox}
\]
Then, $(0,0)$ is an equilibrium of this system, which is unstable (cf.~Proposition~\ref{prop-stability3d}). Further, if \eqref{firstinmutantlessfit2ineq} holds, then $(n_{1a},n_{1d})$ is also an equilibrium, which is asymptotically stable, whereas if \eqref{reversefirstin} holds, then there is no other coordinatewise nonnegative equilibrium. Clearly,~\eqref{2dimHGTapprox} is a two-dimensional system whose coordinatewise nonnegative solutions are all bounded (where the bound depends on the initial condition).

Let us first analyse the case when \eqref{reversefirstin} holds. In this case, $\lambda_1-\lambda_2-\smfrac{\tau}{C}(\lambda_2-\mu)<0$ by part~\eqref{trait2compineq} of Lemma~\ref{lemma-effectivecomp}. Hence, as long as $n_{1a}(t) \geq 0$ and $n_{1d}(t) \geq 0$ (which certainly holds for all $t > 0$ if it holds for $t = 0$), the divergence of the system, which equals $\lambda_1-\lambda_2-\smfrac{\tau}{C}(\lambda_2-\mu)-2\smfrac{\tau^2}{C}n_{1a}(t)-(\kappa\mu+\sigma)$, is strictly negative. Hence, the Bendixson criterion implies that the system \eqref{2dimHGTapprox} has no nontrivial periodic trajectory in the closed positive orthant. Hence, starting from a coordinatewise nonnegative trajectory at $t=0$, the system must converge to $(0,0)$ as $t \to \infty$.

Let us now consider the case when \eqref{firstinmutantlessfit2ineq} holds. Note that in this case we have $(C+\tau) \widetilde n_2 < C(\lambda_1-\mu)$ according to part~\eqref{trait2coexcompineq} of Lemma~\ref{lemma-effectivecomp}. Hence, starting from an initial condition in $[0,\infty)^2 \setminus \{(0,0)\}$, the solution of \eqref{3dimHGT} cannot converge to $(0,0)$ as $t \to \infty$; more precisely, we see immediately that
\[ \frac{\tau^2}{C} n_{1a}(t) \geq  \lambda_1-\lambda_2-\frac{\tau}{C}(\lambda_2-\mu) >0 \numberthis\label{can'tgetsmall} \]
holds for all sufficiently large $t >0$. Now, if $t_0>0$ is such that \eqref{can'tgetsmall} is satisfied for all $t>t_0$, then the divergence of the system, which equals $\lambda_1-\lambda_2-\frac{\tau}{C}(\lambda_2-\mu)-2\frac{\tau^2}{C}n_{1a}(t)-(\kappa\mu+\sigma)n_{1d}(t)$, is strictly negative. Hence, the system \eqref{2dimHGTapprox} cannot have a nontrivial periodic solution. Consequently, it must converge to an equilibrium as $t \to \infty$, which can only be $(n_{1a},n_{1d})$. 

Let us now turn back to the original system \eqref{3dimHGT} started from $\mathfrak B^3_\eps$. We have seen that $n_2(t)$ is increasing as long as \eqref{artificialcomp} holds with `$=$' replaced by `$<$' (which is certainly true for $t \in [0,\delta')$ for some $\delta'>0$, thanks to part~\eqref{n2largeduetoHGT} of Lemma~\ref{lemma-effectivecomp}).  In particular,  part~\eqref{n2largeduetoHGT} of Lemma~\ref{lemma-effectivecomp} also implies that $\liminf_{t \to\infty} n_2(t)>0$. Hence, there are two possibilities. 

The first one is that $t \mapsto n_2(t)$ eventually becomes monotone. Then, by boundedness, $\lim_{t \to \infty}\dot n_2(t)=0$. Thus, in view of the third line of \eqref{3dimHGT}, since $n_2(t)$ stays bounded away from zero for all times, it follows that $C(n_{1a}(t)+n_2(t))-\tau n_{1a}(t)$ tends to $\lambda_2-\mu$. Hence, 
\color{black} once $\eps>0$ is sufficiently small, for all sufficiently small $\delta>0$ we can find $t_1=t_1(\delta)>0$ such that uniformly for all initial conditions $(n_{1a}(0),n_{1d}(0),n_2(0)) \in \mathfrak B^3_\eps$, for all $t > t_1$ we have 
\[ \lambda_2-\mu-\delta < C(n_{1a}(t)+n_2(t))-\tau n_{1a}(t) < \lambda_2-\mu+\delta. \]
Now, for any $\varrho \in (-\delta,\delta)$, the two-dimensional system
\[
\begin{aligned}
 \dot n_{1a}(t) & = n_{1a}(t)\big(\lambda_1-\lambda_2-\frac{\tau}{C}(\lambda_2-\mu)-\frac{\tau^2}{C}n_{1a}(t)-\varrho\big)+\sigma n_{1d}(t), \\
 \dot n_{1d}(t) & = n_{1a}(t) p(\lambda_2-\mu+\tau n_{1a}(t)+\varrho) - (\kappa\mu+\sigma)n_{1d}(t),
\end{aligned}
\numberthis\label{2dimHGTapproxperturbed}
\]
which is defined analogously to \eqref{2dimHGTapprox} but with $C(n_{1a}(t)+n_2(t))-\tau n_{1a}(t)=\lambda_2-\mu+\varrho$ everywhere instead of $C(n_{1a}(t)+n_2(t))-\tau n_{1a}(t)=\lambda_2-\mu$, exhibits the following behaviour. If $\delta$ is small enough, then for any $\varrho \in (-\delta,\delta)$, the solution $(n_{1a}(t),n_{1d}(t))$ of \eqref{2dimHGTapproxperturbed} converges to $(0,0)$ under the assumption \eqref{reversefirstin} and to a coordinatewise positive equilibrium $(n_{1a}^{\varrho},n_{1d}^\varrho)$ under the assumption \eqref{firstinmutantlessfit2ineq}, which tends to $(n_{1a},n_{1d})$ uniformly over $\varrho \in (-\delta,\delta)$ as $\delta \downarrow 0$. This implies that under the assumption \eqref{reversesecondin}, the pair of coordinates $((n_{1a}(t),n_{1d}(t)))_{t \geq 0}$ of the solution of \eqref{3dimHGT} tends to $(0,0)$. Since we have already seen that that $C(n_{1a}(t)+n_2(t))-\tau n_{1a}(t) \to \lambda_2-\mu$ as $t \to \infty$, \color{black} we conclude that $\lim_{t \to \infty} n_2(t)=\lambda_2-\mu=C\bar n_2$, which implies \eqref{fixationof2}. On the other hand, under condition~\eqref{firstinmutantlessfit2ineq}, given that $(n_{1a}(t),n_{1d}(t)) \to (n_{1a},n_{1d})$ and $C(n_{1a}(t)+n_2(t))-\tau n_{1a}(t) \to \lambda_2-\mu$ as $t \to \infty$, we have
\[ 
\begin{aligned}
\lim_{t \to \infty} n_2(t) & = \lim_{t \to \infty} \frac{\lambda_2-\mu+(\tau-C)n_{1a}(t)}{C} = \frac{\lambda_2-\mu+(\tau-C) n_{1a}}{C} \\ & = \frac{C (n_{1a}+n_2)-\tau n_{1a}+(\tau-C) n_{1a}}{C} = n_2,
\end{aligned}
\]
where in the penultimate step we used part~\eqref{compeq} of Lemma~\ref{lemma-effectivecomp}. 

The second possibility is that $t \mapsto n_2(t)$ is not monotone on any interval of the form $[s,\infty)$, $s \geq 0$. Then, since $t \mapsto n_2(t)$ is continuously differentiable according to the Picard--Lindelöf theorem, it follows that the set $E$ of loci of local maxima and local minima of $t \mapsto n_2(t)$ is unbounded (and at most countable). At any local maximum or minimum locus $t$ we have $\dot n_2(t)=0$, and since it is impossible that $n_2(t)=0$, this implies that $C(n_{1a}(t)+n_2(t))-\tau n_{1a}(t)=\lambda_2-\mu$. Now, thanks to our observations regarding the behaviour of the system \eqref{2dimHGTapprox}, the pair of coordinates $((n_{1a}(t),n_{1d}(t)))_{t \geq 0}$ of the solution \eqref{3dimHGT} satisfies $\lim_{t \to \infty, t \in E} (n_{1a}(t),n_{1d}(t))=(0,0)$ under the assumption \eqref{reversefirstin} and  $\lim_{t \to \infty, t \in E} (n_{1a}(t),n_{1d}(t))=(n_{1a},n_{1d})$ under the assumption \eqref{firstinmutantlessfit2ineq}. Similarly to the previous case, we conclude that $\lim_{t \to \infty, t \in E} n_2(t)=\bar n_2$ under the condition \eqref{reversefirstin} and $\lim_{t \to \infty, t\in E} n_2(t) = n_2$ under the one \eqref{firstinmutantlessfit2ineq}. But since $E$ is the set of loci of local extrema of $n_2(t)$, this implies that the limiting assertions of the previous sentence remain true if we remove the condition $t \in E$. Now, the fact that $n_2(t)$ has a limit and the continuous differentiability of $t \mapsto n_2(t)$ imply that $\dot n_2(t)$ tends to 0, and since $\lim_{t\to\infty} n_2(t) \neq 0$, it follows that $\lim_{t\to\infty} C(n_{1a}(t)+n_2(t))-\tau n_{1a}(t)=\lambda_2-\mu$. Therefore, under the condition \eqref{reversefirstin}, we have $\lim_{t \to\infty} n_{1a}(t)=0$. Hence, it is also true that $\lim_{t\to\infty} n_{1d}(t)=0$. On the other hand, under the assumption \eqref{firstinmutantlessfit2ineq}, we derive that $\lim_{t \to \infty} n_{1a}(t)=n_{1a}$ (cf.~part~\eqref{compeq} of Lemma~\ref{lemma-effectivecomp}). In particular, $\lim_{t \to \infty} \dot n_{1a}(t)=0$, and hence $\lim_{t\to\infty} \sigma n_{1d}(t)=n_{1a}(\lambda_1-\mu-C(n_{1a}+n_{2})-\tau n_2)=\sigma n_{1d}$. Thus, we conclude the proposition.
\end{proof}

\subsection{The third phase of invasion: extinction of the resident population}\label{sec-phase3HGTguys}
Let us assume that \eqref{secondinmutantlessfit2ineq} holds. After the second phase, the rescaled process $\mathbf N_t^K$ is near the state $(n_{1a},n_{1d},n_2)$ under condition \eqref{firstinmutantlessfit2ineq} and close to $(\bar n_{1a},\bar n_{1d},0)$ in case \eqref{reversefirstin}. In the first case, there is no third phase; the system stays close to the coexistence equilibrium $(n_{1a},n_{1d},n_2)$ for an amount of time that is at least exponential in $K$, which can be derived similarly as \eqref{FW1HGTguys} and \eqref{FW2HGTguys}. (On larger time scales, $(n_{1a},n_{1d},n_2)$ is only metastable: eventually the population will die out with high probability, cf.~eg.~\cite[Sections 2, 4.2]{C06}.). On the other hand, in the second case, there is one, which ends with the extinction of the formerly resident trait 1 with high probability after a time that is logarithmic in $K$, while trait 2 stays close to its equilibrium population size $\bar n_2$. To be more precise, the following proposition holds, where we recall the set $S_\beta^2$ from \eqref{Sbeta2def} and the stopping time $T_{S_\beta^2}$ from \eqref{TSbetaidef}. 
\begin{prop}\label{prop-thirdphaseHGTguys}
Assume that \eqref{secondinmutantlessfit2ineq} and \eqref{reversefirstin} hold with $\lambda_1>\mu$. There exist $\eps_0,C_0>0$ such that for all $\eps \in (0,\eps_0)$, if there exists $\eta \in (0,1/2)$ that satisfies
\[ \big| N^K_{2}(0)-\bar n_{2} \big| \leq \eps~\text{ and }~\eta \eps/2 \leq N^K_{1a}(0)+N^K_{1d}(0) \leq \eps/2, \]
then
\begin{align*}
    \forall E>-1/\widetilde \lambda,  \quad & \P(T_{S_{2\eps}^2} \leq E\log K) \underset{K \to \infty}{\longrightarrow} 1, \\
    \forall 0 \leq E <-1/\widetilde \lambda, \quad & \P(T_{S_{2\eps}^2} \leq E \log K) \underset{K \to \infty}{\longrightarrow} 0.
\end{align*}
\end{prop}
\begin{proof}
Analogously to the proof of Lemma~\ref{lemma-residentsstayHGTguys}, we first show that the rescaled mutant population size vector $N_{2,t}^K$ stays close to its equilibrium $\bar n_2$ for long times, given that the resident population is small. To this aim, we use methods from \cite[Proof of Proposition 4.1, Step 1]{C+16}, which were also used in \cite[Section 4.3]{BT19} in the case of a one-coordinate resident and a two-coordinate mutant population. For $\eps>0$, we define
\[ R_\eps^2: = \inf \Big\{ t \geq 0 \colon  \big| N_{2,t}^K - \bar n_{2} \big| > \eps \Big\}, \numberthis\label{R2epsdef} \]
the first time the mutant population leaves an $\eps$-ball around its one-trait equilibrium, and we recall the stopping times $T_0^1$ and $T_\eps^1$ from Section~\ref{sec-results}.
These stopping times also depend on $K$, but we have dropped the $K$-dependence from the notation for readability. 

We couple $N_{2,t}^K$ with two two-type birth-and-death processes, $N_{2,t}^{\eps,\preceq}$ and $N_{2,t}^{\eps,\succeq}$, such that
\[ N^{\eps,\preceq}_{2,t} \leq N_{2,t}^K \leq N^{\eps,\succeq}_{2,t}, \qquad \text{a.s.} \qquad \forall 0 \leq t \leq T_{\eps}^1. \numberthis\label{lastmutantcouplingHGTguys} \] 
In order to fulfill \eqref{lastmutantcouplingHGTguys}, these processes can be defined with the following rates:
\[ \begin{aligned}
N_{2}^{\eps,\preceq} \colon & i \to i+1 & \text{ at rate } & i\lambda_2, 
& i \to i-1 & \text{ at rate } i \Big(\mu + C \big(\eps+\frac{i}{K}\big)\Big),
\end{aligned}
\]
and
\[ \begin{aligned}
N_{2}^{\eps,\succeq} \colon & i \to i+1 & \text{ at rate } & i(\lambda_2+\tau\eps), 
& i \to i-1 & \text{ at rate } i \Big(\mu + C \frac{i}{K} \Big).
\end{aligned}
\]
The principle of this coupling is the same as for the coupling satisfying \eqref{mutantcouplingHGTguys}.

We now show that the processes $N_{2,t}^{\eps,\preceq}$ and $N_{2,t}^{\eps,\succeq}$ will stay close to $\bar n_{2}$ for at least an exponential (in $K$) time with a probability close to 1 for large $K$, similarly to the behaviour of the resident population size in the first phase of invasion (cf.~Section~\ref{sec-phase1HGTguys}). To this aim, we will investigate the stopping time
\[ R_{\varrho}^{j,2} = \inf \big\{ t \geq 0 \colon N_{2,t}^{\eps,j} \notin [\bar n_{2}-\varrho,\bar n_{2}+\varrho] \big\} 
\big\} 
\]
for $\varrho >0$ and $j \in \{ \preceq, \succeq \}$. Let us first study the process $N_{2,t}^{\eps,\preceq}$. By \cite[Theorem 2.1, p.~456]{EK}, the dynamics of this process is close to the dynamics of the unique solution to
\[
\begin{aligned}
\dot n_{2}(t)&=n_{2}(t)(\lambda_2-\mu-C \eps-C n_{2}(t)).
\end{aligned}
\]
Since \eqref{secondinmutantlessfit2ineq} and \eqref{reversefirstin} hold with $\lambda_1>\mu$, $\lambda_2>\mu$ follows thanks to part~\eqref{if1fit2fit} of Lemma~\ref{lemma-somebodyfit}. Hence, we find that for all sufficiently small $\eps>0$, this system has a unique coordinatewise positive equilibrium $\bar n_{2}^{\eps,\preceq}=\smfrac{\lambda_2-\mu-C\eps}{C}$, which is asymptotically stable and tends to $\bar n_2=\lambda_2-\mu$ as $\eps \downarrow 0$, whereas $0$ is unstable and there is no other coordinatewise nonnegative equilibrium. Further, starting from any positive initial condition, the solution converges to the stable positive equilibrium as $t \to \infty$ (which is a direct consequence of the boundedness of the solution, given the aforementioned stability properties of the two equilibria). Thus, we can find $\eps'_0>0$ such that for any $\eps \in (0,\eps'_0)$,
\[  \big| \bar n_{2}^{\eps,\preceq}-\bar n_{2} \big| \leq \eps \text{ and } 0 \notin [\bar n_{2}-2\eps, \bar n_{2} + 2\eps] . \] 
Now, similarly to the proof of Lemma~\ref{lemma-residentsstayHGTguys}, we can use results by Freidlin and Wentzell about exit of jump processes from a domain \cite[Section 5]{FW84} in order to construct a family (over $K$) of Markov processes $(\widetilde N_{2,t})_{t \geq 0}$ whose transition rates are positive, bounded, Lipschitz continuous, and uniformly bounded away from 0 such that for
\[ \widetilde R_{\eps}^{2,\preceq} = \inf \big\{ t \geq 0 \colon \widetilde N_{2,t} \notin [\bar n_{2}-\eps,\bar n_{2}+\eps] \big\}  \]
there exists $V>0$ such that for all $i \in \{ a, d \}$ we have
\[ \P\big( R^{2,\preceq}_{2\eps} > \e^{KV} \big) = \P\big( \widetilde R^{2,\preceq}_{2\eps} > \e^{KV} \big) \underset{K \to \infty}{\longrightarrow} 1. \numberthis\label{firstFWextinctionHGTguys} \] 
See~\cite[Section 4.2]{C06} for further details. 
Similarly, we obtain
\[ \P\big( R^{2,\succeq}_{2\eps} > \e^{KV} \big) \underset{K \to \infty}{\longrightarrow} 1, \numberthis\label{secondFWextinctionHGTguys} \]
where without loss of generality we can assume that the constants $V$ in \eqref{firstFWextinctionHGTguys} and in \eqref{secondFWextinctionHGTguys} are the same. 
Now note that $R_{2\eps}^2 \geq R_{2\eps}^{2,\preceq} \wedge R_{2\eps}^{2,\succeq}$ on the event $\{ R_{2\eps}^2 \leq T_\eps^1 \}$. This together with \eqref{firstFWextinctionHGTguys} and \eqref{secondFWextinctionHGTguys} implies that
\[ \lim_{K \to \infty} \P\big( R_{2\eps}^2 \leq \e^{KV} \wedge T^1_\eps)=0. \numberthis\label{mutantsstayinequilibriumHGTguys} \]
Now, we can find a pair of two-type branching processes $\mathbf N_1^{\eps,\preceq}=(N_{1a}^{\eps,\preceq},N_{1d}^{\eps,\preceq})=(N_{1a,t}^{\eps,\preceq},N_{1d,t}^{\eps,\preceq})_{t \geq 0}$ and $\mathbf N_1^{\eps,\succeq}=(N_{1a}^{\eps,\succeq},N_{1d}^{\eps,\succeq})=(N_{1a,t}^{\eps,\succeq},N_{1d,t}^{\eps,\succeq})_{t \geq 0}$ such that 
\[ N_{1\upsilon,t}^{\eps,\preceq} \leq K N_{1\upsilon,t}^K \leq N_{1\upsilon,t}^{\eps,\succeq}, \qquad \forall \upsilon \in \{ a, d \}, \numberthis\label{subcriticalcouplingHGTguys} \]
almost surely on the time interval $[ 0, R_{2\eps}^2  \wedge  T^1_\eps ]$.
Indeed, in order for \eqref{subcriticalcouplingHGTguys} to be satisfied, the processes $\mathbf N_{1}^{\eps,\preceq}$ and $\mathbf N_{1}^{\eps,\succeq}$ can be defined with the following rates and initial conditions: 
\[ \begin{aligned}
(N_{1a,t}^{\eps,\preceq},N_{1d,t}^{\eps,\preceq}) \colon \quad &  \Big(\frac{i}{K},\frac{j}{K}\Big) \to \Big(\frac{i+1}{K},\frac{j}{K}\Big)& \text{at rate} & \quad i \lambda_1, \\
& \Big(\frac{i}{K},\frac{j}{K}\Big) \to \Big(\frac{i-1}{K},\frac{j}{K}\Big) & \text{at rate}  &\quad i(\mu+(1-p)(C(\bar n_2+3\eps))+\tau(\bar n_2+2\eps)), \\
&  \Big(\frac{i}{K},\frac{j}{K}\Big) \to \Big(\frac{i-1}{K},\frac{j+1}{K}\Big) & \text{at rate}& \quad pC i(\bar n_2-2\eps), \\
&  \Big(\frac{i}{K},\frac{j}{K}\Big) \to \Big(\frac{i+1}{K},\frac{j-1}{K}\Big)  & \text{at rate} &\quad j\sigma, \\
&  \Big(\frac{i}{K},\frac{j}{K}\Big) \to \Big(\frac{i}{K},\frac{j-1}{K}\Big) & \text{at rate} &\quad j\kappa\mu,
\end{aligned}
\]
started from an initial condition $(N_{1a,0}^{\eps,\preceq},N_{1d,0}^{\eps,\preceq})$ satisfying $\lfloor K(N_{1a,0}^{\eps,\preceq}+N_{1d,0}^{\eps,\preceq}) \rfloor= \eps/2$, and
\[ \begin{aligned}
(N_{1a,t}^{\eps,\succeq},N_{1d,t}^{\eps,\succeq}) \colon \quad &  \Big(\frac{i}{K},\frac{j}{K}\Big) \to \Big(\frac{i+1}{K},\frac{j}{K}\Big)& \text{at rate} & \quad i \lambda_1, \\
& \Big(\frac{i}{K},\frac{j}{K}\Big) \to \Big(\frac{i-1}{K},\frac{j}{K}\Big) & \text{at rate}  &\quad i(\mu+(1-p)(C(\bar n_2-2\eps))+\tau(\bar n_2-2\eps)), \\
&  \Big(\frac{i}{K},\frac{j}{K}\Big) \to \Big(\frac{i-1}{K},\frac{j+1}{K}\Big) & \text{at rate}& \quad pCi(\bar n_2+3\eps), \\
&  \Big(\frac{i}{K},\frac{j}{K}\Big) \to \Big(\frac{i+1}{K},\frac{j-1}{K}\Big)  & \text{at rate} &\quad j\sigma, \\
&  \Big(\frac{i}{K},\frac{j}{K}\Big) \to \Big(\frac{i}{K},\frac{j-1}{K}\Big) & \text{at rate} &\quad j\kappa\mu. 
\end{aligned}
\]
started from an initial condition $(N_{1a,0}^{\eps,\succeq},N_{1d,0}^{\eps,\succeq})$ satisfying $\lceil K( N_{1a,0}^{\eps,\succeq}+N_{1d,0}^{\eps,\succeq}) \rceil =\eta\eps/2$. This coupling is based on the same principle as the one satisfying \eqref{firstresidentcouplingHGTguys}.

For all $\eps>0$ sufficiently small, both of these two-type branching processes are subcritical, cf.~Section~\ref{sec-phase13}. In other words, the largest eigenvalues $\widetilde \lambda^{\eps,\preceq}$ respectively $\widetilde\lambda^{\eps,\succeq}$ of their mean matrices $J^{\eps,\preceq}$ respectively $J^{\eps,\succeq}$ defined analogously to the matrix $J$ introduced in \eqref{JdefHGT}, are strictly negative. Further, these eigenvalues both converge to the largest eigenvalue $\widetilde \lambda$ of $J$, which was explicitly expressed in \eqref{lambdaHGT}, as $\eps \downarrow 0$.
From this, analogously to \cite[Section 3.3]{C+19}, we conclude that the extinction time of these processes started from $[\lfloor\smfrac{ \eta K \eps}{2} \rfloor, \lfloor \smfrac{\eps K}{2} \rfloor +1]$ is of order $(\smfrac{1}{\widetilde \lambda} \pm O(\eps))\log K$. This in turn follows from the general assertion~\cite[p.~202]{AN72}  that for a two-type branching process $\mathfrak N=(\mathfrak N(t))_{t \geq 0}=(\mathfrak N_1(t),\mathfrak N_2(t))_{t \geq 0}$ that is subcritical, with its mean matrix having largest eigenvalue $\bar \lambda<0$,
given that $\mathfrak N_1(0)+\mathfrak N_2(0) \in [\lfloor\smfrac{ \eta K \eps}{2} \rfloor, \lfloor \smfrac{\eps K}{2} \rfloor +1]$, defining
\[ \mathfrak S_\eps^{\mathfrak N} = \inf \{ t \geq 0 \colon \mathfrak N_1(t)+\mathfrak N_2(t) \geq \lfloor \eps K \rfloor \} , \qquad \eps > 0,  \numberthis\label{extinction2typelower} \]
and
\[ \mathfrak S_0^{\mathfrak N} = \inf \{ t \geq 0 \colon \mathfrak N_1(t)+\mathfrak N_2(t)  =0  \} , \numberthis\label{extinction2typeupper} \]
the following hold:
\[\forall C<\bar \lambda^{-1}, \qquad \lim_{K \to \infty} \P(\mathfrak S_0^{\mathfrak N} \leq C \log K)=0\]
and
\[ \forall C>\bar \lambda^{-1}, \qquad \lim_{K \to \infty} \P(\mathfrak S_0^{\mathfrak N} \leq C \log K)=1. \]
Further, if $\mathfrak N_1(0)+\mathfrak N_2(0)=\lfloor \smfrac{ \eta K \eps}{2} \rfloor$, then for all small enough $\eps>0$,
\[ \lim_{K \to \infty} \P\Big( \mathfrak S_0^{\Ncal} > K \wedge \mathfrak S_{\lfloor \eps K \rfloor }^{\mathfrak N} \Big) = 0. \numberthis\label{diebeforeyougrowupHGTguys} \]
Now for $E \geq 0$ we can estimate as follows
\[ 
\begin{aligned}
\P(T_0^1 < E \log K) - \P\big(\mathfrak S_0^{N_1^{\eps,\preceq}}  < E \log K \big)
\leq & \P\big( T_0^1 > T_\eps^1 \wedge K \big) + \P\big( T_\eps^1 \wedge K > R_{2\eps}^2  \big) \\
\leq & \P\big( \mathfrak S_0^{N_1^{\eps,\succeq}} > \mathfrak S_{\lfloor \eps K \rfloor}^{N_1^{\eps,\succeq}} \wedge K \big) +\P\big( T_\eps^1 \wedge K > R_{2\eps}^2 \big).
\end{aligned} \numberthis\label{strangeboundHGTguys} \]
Here, the first inequality can be verified in the following way:
\[
\begin{aligned}
\P(T_0^1 < E \log K)&- \P \big( \mathfrak S_0^{N_1^{\eps,\preceq}} < E \log K \big)  =
 \P(T_0^1 < E \log K \leq \mathfrak S_0^{N_1^{\eps,\preceq}})
\\ &\leq    
\P(R_{2\eps}^2 < T_0^1 < E \log K, R_{2\eps}^2 < T_\eps^1) + \P(
 T_\eps^1 <T_0^1 < E \log K, R_{2\eps}^2 > T_\eps^1)
\\ & \leq
\P(R_{2\eps}^2 < T_\eps^1 \wedge E \log K) + \P(T_0^1 > T_\eps^1)
\\ & \leq
 \P(R_{2\eps}^2 < T_\eps^1 \wedge K) + \P(T_0^1 > T^1_\eps \wedge K)
\end{aligned} \numberthis\label{proofofspeciation}
 \]
As soon as $\eps>0$ is sufficiently small, the second term in the last line of \eqref{strangeboundHGTguys} tends to zero as $K \to \infty$ according to \eqref{mutantsstayinequilibriumHGTguys} and so does the first one according to \eqref{diebeforeyougrowupHGTguys}. We conclude that
\[ \limsup_{K \to \infty} \P\big( T_0^1 < E\log K \big) \leq \lim_{K \to \infty} \P\big( \mathfrak S_0^{N_1^{\eps,\preceq}} \leq E \log K \big) \]
and
\[ \liminf_{K \to \infty} \P\big( T_0^1 < E\log K \big) \geq \lim_{K \to \infty} \P\big( \mathfrak S_0^{N_1^{\eps,\succeq}} \leq E\log K \big). \]
Hence, we conclude the proposition.
\end{proof}
\subsection{Proof of Theorems~\ref{thm-invasionof2} and \ref{thm-failureof2}.}
Putting together Propositions~\ref{prop-firstphaseHGTguys}, \ref{prop-convergenceHGTguys}, and \ref{prop-thirdphaseHGTguys}, we now prove our main results, making use of some arguments from \cite[Section 3.4]{C+19}, similarly to \cite[Section 3.4]{BT19}. However, the possibility of convergence to a stable coexistence equilibrium gives rise to a case distinction that was absent from these two papers. Throughout the proof we will assume that $\lambda_1>\mu$. Our proof strongly relies on the coupling \eqref{mutantcouplingHGTguys}. More precisely, we define a Bernoulli random variable $\widehat B$ as the indicator of nonextinction
\[ \widehat B:= \mathds 1 \{ \forall t>0 \colon \widehat N_{2}(t)>0 \} \]
of the approximating branching process $(\widehat N_{2}(t))_{t \geq 0}$ defined Section~\ref{sec-phase13}, which is initially coupled with the same branching processes as $(K N_{2,t}^K)_{t \geq 0}$ according to \eqref{mutantcouplingHGTguys}. 
Let $\bar f$ be the function defined in Proposition~\ref{prop-firstphaseHGTguys}. Throughout the remainder of the proof, we can assume that $\eps>0$ is so small that $\bar f(\eps) <1$.

We want to show that
\[ \liminf_{K \to \infty} \mathfrak E(K,\eps) \geq q_2-o(\eps) \numberthis\label{extinctionlowerHGTguys}\]
holds for 
\[ \mathfrak E(K,\eps):= \P \Big( \frac{T_0^2}{\log K} \leq \bar f(\eps), T_0^2 < T_{S_\beta^2} \wedge T_{S_\beta^{\rm co}}, \widehat B=0 \Big).\]
Moreover, we want to show that in case $q_2<1$, 
\[ \liminf_{K \to \infty} \mathfrak I^{j_0}(K,\eps) \geq 1-q_2-o(\eps), \numberthis\label{survivallowerHGTguys}\]
where we define
\[ \mathfrak I^{1}(K,\eps):=\P \Big( \Big| \frac{ T_{S_\beta^2}}{\log K}-\Big( \frac{1}{\widetilde \lambda} - \frac{1}{\widehat \lambda} \Big)\Big| \leq \bar f(\eps), T_{S_\beta^2} < T_0^2 , \widehat B=1 \Big) \]
and
\[ \mathfrak I^{\mathrm{co}}(K,\eps):=\P \Big( \Big| \frac{T_{S_\beta^{\mathrm{co}}}}{\log K}- \frac{1}{\widetilde \lambda}  \Big| \leq \bar f(\eps), T_{S_\beta^{\mathrm{co}}} < T_0^2 \wedge T_{S_\beta^2}, \widehat B=1 \Big), \]
further, we put $j_0=\mathrm{co}$ in case $(n_{1a},n_{1d},n_2)$ exists as a coordinatewise positive equilibrium (this corresponds to the case when both \eqref{firstinmutantlessfit2ineq} and \eqref{secondinmutantlessfit2ineq} hold; cf.~also Lemmas~\ref{lemma-coexistence} and~\ref{lemma-lessconditions}) and $j_0=2$ otherwise (which corresponds to the case when \eqref{firstinmutantlessfit2ineq} and \eqref{reversesecondin} hold). 
Throughout the proof, $\beta$ is to be thought as positive and sufficiently small; we will comment on how small it should be in the individual cases.

The assertions~\eqref{extinctionlowerHGTguys} and~\eqref{survivallowerHGTguys} together will imply Theorem~\ref{thm-invasionof2} and the equation \eqref{extinctionof2} in Theorem~\ref{thm-failureof2}. The other assertion of Theorem~\ref{thm-failureof2}, equation \eqref{lastoftheoremof2}, follows already from \eqref{secondofpropHGTguys}.

Let us first consider the case of mutant extinction in the first phase of invasion and verify \eqref{extinctionlowerHGTguys}. It is clear that we have 
\[ \mathfrak E(K,\eps)\geq \P \Big( \frac{T_0^2}{\log K} \leq \bar f(\eps), T_0^2 < T_{S_\beta^2} \wedge T_{S_\beta^{\rm co}}, \widehat B=0, T_0^2< T_{\sqrt\eps}^2 \wedge R_{b\sqrt\eps}^1 \Big). \]
Now, according to our initial conditions, one can choose $\beta>0$ sufficiently small such that for all sufficiently small $\eps>0$ we have
\[ T_{\sqrt\eps}^2 \wedge R_{b\sqrt\eps}^1 < T_{S_\beta^2} \wedge T_{S_\beta^{\rm co}}, \]
almost surely. We assume further on during the proof that $\beta$ satisfies this condition. Then,
\[ \mathfrak E(K,\eps)\geq \P \Big( \frac{T_0^2}{\log K} \leq \bar f(\eps), \widehat B=0, T_0^2 < T_{\sqrt\eps}^2 \wedge R_{b\sqrt\eps}^1 \Big). \numberthis\label{andisandHGTguys} \]
Moreover, analogously to the proof of \cite[Proposition 3.1]{C+19}, we obtain
\[ \limsup_{K \to \infty} \P \big( \{ \widehat B=0 \} \Delta \{ T_0^2 < T_{\eps}^2 \wedge R_{b\sqrt\eps}^1  \}  \big)=o_\eps(1), \numberthis\label{undefinedsymmdiffHGTguys} \]
where $\Delta$ denotes symmetric difference, and
\[ \limsup_{K \to \infty} \P \big( \{ \widehat B=0 \} \Delta \{ T_0^{(\eps,2),2} < \infty \}  \big)=o_\eps(1). \]
Together with \eqref{andisandHGTguys}, these imply that
\[
\begin{aligned}
\liminf_{K \to \infty} \mathfrak E(K,\eps) \geq & \P \Big( \frac{\widehat
\lambda T_0^2}{\log K} \leq \bar f(\eps), T_0^{(\eps,2),2} < \infty \Big) + o_\eps(1) \\ \geq &  \P \Big( \frac{\widehat\lambda
T_0^{(\eps,2),2}}{\log K} \leq \bar f(\eps), T_0^{(\eps,2),2} < \infty \Big) + o_\eps(1)  \geq \P \Big( T_0^{(\eps,2),2} < \infty \Big) + o_\eps(1), 
\end{aligned}
\numberthis\label{secondlineHGTguys}
\]
where in the second inequality of \eqref{secondlineHGTguys} we used the coupling \eqref{mutantcouplingHGTguys}. Thus, using~\eqref{qineqHGTguys}, we arrive at~\eqref{extinctionlowerHGTguys}. Hence, we have derived \eqref{extinctionof2}. 

Let us move on to the case of mutant survival in the first phase of invasion and verify \eqref{survivallowerHGTguys}.  
Arguing analogously to \eqref{undefinedsymmdiffHGTguys}, we get
\[ \limsup_{K \to \infty} \P \big( \{ \widehat B=1 \} \Delta \{ T_{\sqrt \eps}^2  < T_0^2 \wedge R_{b\sqrt\eps}^1  \}  \big)=o_\eps(1). \] 
Hence,
\begin{equation}\label{beforesetsHGTguysfixation}
\begin{aligned}
\liminf_{K \to \infty} \mathfrak I^{j_0}(K,\eps) & = \liminf_{K \to \infty} \P \Big( \Big| \frac{T_{S_\beta^2}}{\log K} -\Big( \frac{1}{\widehat \lambda} - \frac{1}{\widetilde \lambda}  \Big) \Big| \leq \bar f(\eps), T_{S_\beta^2}<T_0^2,  
 T^2_{\sqrt \eps} < T_0^2 \wedge R_{b\sqrt\eps}^1 \Big) + o_\eps(1)
\end{aligned}
\end{equation} 
if $j_0=2$, and
\begin{equation}\label{beforesetsHGTguyscoex}
\begin{aligned}
\liminf_{K \to \infty} \mathfrak I^{j_0}(K,\eps) & = \liminf_{K \to \infty} \P \Big( \Big| \frac{T_{S_\beta^{\rm co}}}{\log K} - \frac{1}{\widehat \lambda}  \Big| \leq \bar f(\eps), T_{S_\beta^{\mathrm{co}}}<T_0^2 \wedge T_{S_\beta^{2}},  
 T^2_{\sqrt \eps} < T_0^2 \wedge R_{b\sqrt\eps}^1 \Big) + o_\eps(1)
\end{aligned}
\end{equation} 
if $j_0=\mathrm{co}$. 
For $\eps>0$ and $\beta>0$, we recall the set $\mathfrak B^3_\eps$ from \eqref{BHGTguys}.
Note that if $T^2_{\sqrt\eps}<\infty$, then for all small enough $\eps>0$, $\mathbf N^K_{T^2_{\sqrt\eps}} \in \mathfrak B^3_\eps$ for all sufficiently large $K$ (depending on $\eps$). From now on, we assume that $\eps>0$ is so small that it satisfies this condition.
Further, we define another set
\[ \mathfrak B^4_\beta:= [0,\beta/2] \times [0,\beta/2] \times [\bar n_2-\beta/2, \bar n_2+\beta/2] \]
and the stopping time
\[
\begin{aligned}
\widehat T''_{\beta}:=&\inf \Big\{ t \geq T_{\sqrt\eps} \colon \mathbf N^K_t \in \mathfrak B^4_\beta \Big\}.
\end{aligned}
\]
Note that for $j_0=2$, $\mathfrak B^4_\beta$ is in fact the analogue of $S_{\beta}^{\mathrm{co}}$ for $j_0=\mathrm{co}$: these two sets are small closed neighbourhoods of the state that the rescaled population process will reach with high probability conditional on survival of mutants. However, in case $j_0=2$, we additionally need to take into account the time of extinction of the trait 2 population. Thus, vaguely speaking, we want to show that with high probability the process has to pass through $\mathfrak B^3_\eps$ in order to reach $S_\beta^{j_0}$, whatever $j_0 \in \{ 2, \mathrm{co} \}$ is, and afterwards it has to go through $\mathfrak B^4_\beta$ in order to reach $S_\beta^2$ in case $j_0=2$
. Then, using to the Markov property, we can estimate $T_{S_\beta^{j_0}}$ by estimating $T^2_{\sqrt\eps}$, $\widehat T''_{\beta}-T^2_{\sqrt\eps}$, and $T_{S_\beta^{1}}-\widehat T''_{\beta}$ in case $j_0=2$, and by estimating $T^2_{\sqrt\eps}$ and $T_{S_{\beta}^{\mathrm{co}}}-T^2_{\sqrt\eps}$ in case $j_0 = \mathrm{co}$.

In case $j_0=2$, \eqref{beforesetsHGTguysfixation} implies that 
\begin{align*}
    \liminf_{K \to \infty}\mathfrak I^2(K,\eps)& \geq   \P \Big( \Big| \frac{T_{S_\beta^2}}{\log K} -\Big( \frac{1}{\widehat \lambda} - \frac{1}{\widetilde \lambda} \Big) \Big| \leq \bar f(\eps), T^2_{\sqrt \eps} < T_0^2 \wedge R_{b\sqrt\eps}^1, \widehat T''_{\beta}<T_{S_\beta^2}, T_{S_\beta^2}<T_0^2 \Big) + o_\eps(1) \\
     & \geq \P \Big( \Big| \frac{T^2_{\sqrt\eps}}{\log K} -\frac{1}{\widehat \lambda} \Big| \leq \frac{\bar f(\eps)}{3}, \Big| \frac{\widehat T''_{\beta}-T^2_{\sqrt\eps}}{\log K} \Big| \leq \frac{\bar f(\eps)}{3}, \Big| \frac{T_{S_{\beta}^2}-\widehat T''_{\beta}}{\log K}+ \frac{1}{\widetilde\lambda}\Big| \leq \frac{\bar f(\eps)}{3},  \\
    & \qquad   T^2_{\sqrt \eps} < T_0^2 \wedge R_{b\sqrt\eps}^1 , \widehat T''_{\beta}<T_{S_\beta^2}, T_{S_\beta^2}<T_0^2 \Big) + o_\eps(1) ,
\end{align*}
whereas in case $j_0=\mathrm{co}$, \eqref{beforesetsHGTguyscoex} implies that
\begin{align*}
    \liminf_{K \to \infty}\mathfrak I^{\mathrm{co}} & (K,\eps) \geq    \P \Big( \Big| \frac{T_{S_\beta^{\mathrm{co}}}}{\log K} - \frac{1}{\widehat \lambda}  \Big| \leq \bar f(\eps), T^2_{\sqrt \eps} < T_0^2 \wedge R_{b\sqrt\eps}^1, T_{S_\beta^{\mathrm{co}}}<T_0^2 \wedge T_{S_\beta^2} \Big) + o_\eps(1) \\
    \geq &  \P \Big( \Big| \frac{T^2_{\sqrt\eps}}{\log K} -\frac{1}{\widetilde \lambda} \Big| \leq \frac{\bar f(\eps)}{3}, \Big| \frac{T_{S_\beta^\mathrm{co}} -T^2_{\sqrt\eps}}{\log K} \Big| \leq \frac{\bar f(\eps)}{3},  T^2_{\sqrt \eps} < T_0^2 \wedge R^1_{b\sqrt\eps} , T_{S_\beta^\mathrm{co}}<T_0^2 \wedge T_{S_\beta^1} \Big) + o_\eps(1) ,
\end{align*}
Note that for $\beta>0$ small enough and $\eps>0$ sufficiently small chosen accordingly, $R_{b\sqrt\eps}^1 \leq T_{S_\beta^{j_0}}$ almost surely, whatever $j_0$ is. Hence, applying the strong Markov property at times $T^2_{\sqrt\eps}$ and $\widehat T''_ {\beta}$ implies 
\[ \begin{aligned}
    \liminf_{K \to \infty} \mathfrak I^2(K,\eps)&\geq \liminf_{K \to \infty} \Big[ \P \Big( \Big| \frac{T^2_{\sqrt\eps}}{\log K} -\frac{1}{\widehat \lambda} \Big| \leq \frac{\bar f(\eps)}{3},  T^2_{\sqrt\eps} < T_0^2 \wedge R_{b\sqrt\eps}^1 \Big) \\
    & \qquad \times \inf_{\begin{smallmatrix}\mathbf n=(n_{1a},n_{1d},n_{2}) \colon \mathbf n \in \mathfrak B^3_\eps\end{smallmatrix}} \P \Big(  \Big| \frac{\widehat T''_{\beta}-T^2_{\sqrt\eps}}{\log K} \Big| \leq \frac{\bar f(\eps)}{3}, \widehat T''_{\beta} < T_0^2 \Big| \mathbf N^K_0=\mathbf n \Big) \\
    & \qquad \times \inf_{\mathbf n \in \mathfrak B^4_\beta} \P \Big(\Big| \frac{T_{S_{\beta}^2}-\widehat T''_{\beta}}{\log K}+\frac{1}{\widetilde \lambda}\Big| \leq \frac{\bar f(\eps)}{3}, T_{S_\beta^2} < T_0^2 \Big| \mathbf N_0^K = \mathbf n\Big) \Big]+o_{\eps}(1) \end{aligned} \numberthis\label{productformfixation2} \]
in case $j_0=2$, and 
\[ \begin{aligned}
    \liminf_{K \to \infty} &\mathfrak I^{\mathrm{co}}(K,\eps)\geq \liminf_{K \to \infty} \Big[ \P \Big( \Big| \frac{T^2_{\sqrt\eps}}{\log K} -\frac{1}{\widetilde \lambda} \Big| \leq \frac{\bar f(\eps)}{3}, T^2_{\sqrt\eps} < T_0^2 \wedge R_{b\sqrt\eps}^1 \Big) \\
    & \qquad \times \inf_{\begin{smallmatrix}\mathbf n \in \mathfrak B^3_\eps\end{smallmatrix}} \P \Big(  \Big| \frac{T_{S_\beta^\mathrm{co}} -T^2_{\sqrt\eps}}{\log K} \Big| \leq \frac{\bar f(\eps)}{3}, T_{S_\beta^{\mathrm{co}}} < T_0^2 \wedge T_{S_\beta^2} \Big| \mathbf N^K_0=\mathbf n \Big) \Big]\\
   \end{aligned} \numberthis\label{productformcoex2} \] in case $j_0=\mathrm{co}$. 
Hence, it remains to show that the right-hand side of \eqref{productformfixation2} respectively \eqref{productformcoex2} is close to $1-q_2$ as $K \to \infty$, given that $\eps>0$ is small. Let us first analyse the first term (which appears on the right-hand side of both equations) and show that 
\[ \liminf_{K \to \infty} \P \Big( \Big| \frac{T^2_{\sqrt\eps}}{\log K} -\frac{1}{\widehat \lambda} \Big| \leq \frac{\bar f(\eps)}{3}, T^2_{\sqrt\eps}<T_0^2, T^2_{\sqrt\eps} < T_0^2 \wedge R_{b\sqrt\eps}^1 \Big) \geq 1-q_2+o_{\eps}(1). \numberthis\label{firsttermHGTguys} \]
Now, we have the following
\[
\begin{aligned}
  \P \Big( \Big|  & \frac{T^2_{\sqrt\eps}}{\log K} -\frac{1}{\widehat \lambda} \Big| \leq \frac{\bar f(\eps)}{3}, T^2_{\sqrt\eps} < T_0^2 \wedge R_{b\sqrt\eps}^1 \Big) = \P  \Big( \Big| \frac{T^2_{\sqrt\eps}}{\log K} -\frac{1}{\widehat \lambda} \Big| \leq \frac{\bar f(\eps)}{3}, T^2_{\sqrt\eps} < T_0^2 \wedge R_{b\sqrt\eps}^1 \Big)
\\ &\geq  \P \Big(  \Big| \frac{\widehat\lambda T^2_{\sqrt\eps}}{\log K}-1 \Big|\leq \frac{\bar f(\eps)}{6},   T^2_{\sqrt\eps} < T_0^2 \wedge R_{b\sqrt\eps}^1 \Big) 
\\& \geq  \P \Big( \frac{\widehat\lambda T^2_{\sqrt\eps}}{\log K} - \frac{\widehat \lambda T^2_\eps}{\log K}  \leq \frac{\bar f(\eps)}{6}, \Big| \frac{\widehat\lambda T^2_{\sqrt\eps}}{\log K}-1 \Big|\leq \frac{\bar f(\eps)}{6},   T^2_{\sqrt\eps} < T_0^2 \wedge R_{b\sqrt\eps}^1 \Big). 
\end{aligned}
\]
Note that for any three events $A_1,A_2,A_3$, the following holds
\[ \P(A_1\cap A_2 \cap A_3) \geq \P(A_2 \cap A_3) - \P(A_1^c \cap A_3) -\P(A_2^c \cap A_3). \]
Applying this to the previous chain of inequalities, we obtain
\[
\begin{aligned}
\P &\Big(  \Big| \frac{T^2_{\sqrt\eps}}{\log K} -\frac{1}{\widetilde \lambda} \Big| \leq \frac{\bar f(\eps)}{3}, T^2_{\sqrt\eps} < T_0^2 \wedge R^1_{b\sqrt\eps} \Big) 
\\   \geq  \P &\Big(  \Big| \frac{\widehat\lambda T^2_{\sqrt\eps}}{\log K}-1 \Big|\leq \frac{\bar f(\eps)}{6},   T^2_{\sqrt\eps} < T_0^2 \wedge R^1_{b\sqrt\eps} \Big) - \P \Big( \frac{\widehat\lambda T^2_{\sqrt\eps}}{\log K} \geq \frac{\bar f(\eps)}{6}, T^1_{\sqrt\eps} < T_0^2 \wedge R_{b\sqrt\eps}^1  \Big).
\end{aligned}
\numberthis\label{ABCDHGTguys}
\]
Proposition~\ref{prop-firstphaseHGTguys} implies that the first term on the right-hand side of \eqref{ABCDHGTguys} satisfies
\[ \liminf_{K \to \infty} \P \Big(  \Big| \frac{\widehat\lambda T^2_{\sqrt\eps}}{\log K}-1 \Big|\leq \frac{\bar f(\eps)}{6},   T^2_{\sqrt\eps} < T_0^2 \wedge R^1_{b\sqrt\eps} \Big) \geq 1-q_2-o_\eps(1). \]
Finally, thanks to \cite[Lemma A.2]{C+19}, the second term on the right-hand side of \eqref{ABCDHGTguys} satisfies
\[ \limsup_{K \to \infty} \P \Big( \frac{\widehat\lambda T^2_{\sqrt\eps}}{\log K} \geq \frac{\bar f(\eps)}{6}, T^2_{\sqrt\eps} < T_0^2 \wedge R^1_{b\sqrt\eps}  \Big)=o_\eps(1). \]
This implies \eqref{firsttermHGTguys}.

Now, let us handle the second phase of invasion, both in the case $j_0=2$ (fixation of trait 2, which corresponds to an almost-fixation already by the end of the second phase) and in the case $j_0=\mathrm{co}$ (convergence to the coexistence equilibrium). 
For $\mathbf m=(m_{1a},m_{1d},m_{2}) \in [0,\infty)^3$, let $\mathbf n^{(\mathbf m)}$ denote the unique solution of the dynamical system \eqref{3dimHGT} with initial condition $\mathbf m$. Thanks to the continuity of flows of this dynamical system with respect to the initial condition (cf.~\cite[Theorem 1.1]{DLA06}) and thanks to the convergence provided by Proposition~\ref{prop-convergenceHGTguys}, it follows that if $\beta>0$ is sufficiently small, then there exist $\eps_0,\delta_0>0$ such that for all $\eps \in (0,\eps_0)$ and $\delta \in (0,\delta_0)$, there exists $\widehat t_{\beta,\delta,\eps}>0$ such that for all $t>\widehat t_{\beta,\delta,\eps}$, one has   
\[ \Big\Vert \mathbf n^{(\mathbf n^0)}(t)-(0,0,\bar n_2) \Big\Vert \leq \frac{\beta}{4} \]
in case $j_0=2$, and
\[ \Big\Vert \mathbf n^{(\mathbf n^0)}(t)-(n_{1a},n_{1d},n_2) \Big\Vert \leq \frac{\beta}{4} \]
in case $j_0=\mathrm{co}$, 
for any initial condition $n^0=(n_{1a}^0,n_{1d}^0,n_{2}^0) \in \mathfrak B^3_\eps$. 

Now, using \cite[Theorem 2.1, p.~456]{EK}, we conclude that for all sufficiently small $\beta>0$ and accordingly chosen $\eps<\eps_0$, the following hold. If $j_0=2$, then
\[ \lim_{K \to \infty} \P\Big( \widehat T''_\beta-T^2_{\sqrt\eps} \leq \widehat t_{\beta,\delta,\eps} \Big| \mathbf N^K_0 \in \mathfrak B^3_\eps \Big) =1-o_\eps(1), \]
whereas if $j_0=\mathrm{co}$, then
\[ \lim_{K \to \infty} \P\Big( T_{S_\beta^{\mathrm{co}}}-T^2_{\sqrt\eps} \leq \widehat t_{\beta,\delta,\eps} \Big| \mathbf N^K_0 \in \mathfrak B^3_\eps \Big) =1-o_\eps(1). \]
Thus, for $j_0=2$, the second term on the right-hand side of \eqref{productformfixation2} is close to 1 when $K$ tends to $\infty$, $\beta$ is small and $\eps>0$ is small enough chosen according to $\beta$. The same holds for the second term on the right-hand side of \eqref{productformcoex2} in case $j_0=\text{co}$. In the latter case, together with \eqref{firsttermHGTguys} we have obtained
\[ \liminf_{K \to \infty} \mathfrak I^{\mathrm{co}}(K,\eps) \geq 1-q_2-o_\eps(1), \]
which implies \eqref{coexprob2} and \eqref{coexistenceof2}. 

Finally, let us consider the third term on the right-hand side of \eqref{productformfixation2} (in the case $j_0=2$ of fixation of trait 2). Proposition~\ref{prop-thirdphaseHGTguys} implies that there exists $\beta_0>0$ (denoted as $\eps_0$ in Proposition~\ref{prop-thirdphaseHGTguys}) such that for all $\beta<\beta_0$,
\[ \lim_{K \to \infty} \P \Big(\Big| \frac{T_{S_{\beta}^2}-\widehat T''_{\beta}}{\log K}+ \frac{1}{\widetilde \lambda} \Big| \leq \frac{f(\eps)}{3} \Big| \mathbf N_0^K \in \mathfrak B^4_{\beta} \Big) = 1-o_\eps(1). \]
Combining \eqref{firsttermHGTguys} with the convergence of the second and the third term on the right-hand side of \eqref{productformfixation2} to 1, we obtain
\[ \liminf_{K \to \infty} \mathfrak I^1(K,\eps) \geq 1-q_2-o_\eps(1), \]
which implies \eqref{invasionprob2} and \eqref{invasionof2}.

\section{Invasion of trait 1 against trait 2: Proof of Theorems~\ref{thm-invasionof1} and \ref{thm-failureof1}}\label{sec-dormantproofs}
 In this section we investigate the case when the resident trait is 2, which in particular has no dormant state, and the mutant trait is 1, which exhibits dormancy. This situation is similar to the one analysed in \cite{BT19}, with the additional effect of HGT that is beneficial for the residents and harmful for the mutants. We will be able to use methods of \cite[Section 4.1]{BT19}, which in turn use various techniques of \cite[Section 3.1]{C+19}, but some coupling arguments have to be altered in order to handle HGT. Of course, many arguments for this invasion direction are also similar to our proofs in Section~\ref{sec-HGTproofs} about the reverse one, whence we will omit some details of the proofs in the present section. Still, the treatment of the dynamical system~\eqref{3dimHGT} requires some additional care compared to the HGT-free case.
\subsection{The first phase of invasion: mutant growth or extinction}\label{sec-phase1dormants}
For $\eps>0$, we recall the stopping time $ R_\eps^2$ from \eqref{R2epsdef},
which is the first time that the resident population leaves the closed $\eps$-ball around the equilibrium $\bar n_2=\smfrac{\lambda_2-\mu}{C}$ (for $\lambda_2>\mu$). Note that this stopping time also depends on $K$. Now, the following holds about the first phase of invasion.
\begin{prop}\label{prop-firstphasedormants}
Assume that $\lambda_2>\mu$ and $\widetilde \lambda \neq 0$. 
Let $K \mapsto m_2^K$ be a function from $(0,\infty)$ to $[0,\infty)$ such that $m_2^K \in \smfrac{1}{K} \N_0$ and $\lim_{K \to \infty} m_2^K=\bar n_2$. Then there exists a function $f \colon (0,\infty) \to (0,\infty)$ tending to zero as $\eps \downarrow 0$ such that for any $\xi \in [1/2,1]$,
\[ \limsup_{K \to \infty} \Big| \P \Big( T_{\eps^\xi}^1 < T_0^1 \wedge R^2_{2\eps^\xi}, \Big| \frac{T_{\eps^\xi}^1}{\log K} - \frac{1}{\widetilde \lambda} \Big| \leq \bar f(\eps) \Big| \mathbf N_0^K=\big(\frac{1}{K},0,m_2^K\big)  \Big)-(1-q_2) \Big| = o_\eps(1) \numberthis\label{invasiontimedormants} \]
and
\[ \limsup_{K \to \infty} \Big| \P\Big(T_0^1 < T_{\eps^\xi}^1 \wedge R^1_{2\eps^\xi}~ \Big|  \mathbf N_0^K=\big(\frac{1}{K},0,m_2^K\big)  \Big) - q_2 \Big|=o_{\eps}(1) \numberthis\label{secondofpropdormants} \]
\end{prop}
This proposition is the analogue of \cite[Proposition 4.1]{BT19}. We will now explain what has to be altered in the proof of \cite{BT19} in order to handle the additional HGT terms. The first step of the proof is to verify the following lemma.
\begin{lemma}\label{lemma-residentsstaydormants}
Under the assumptions of Proposition~\ref{prop-firstphasedormants}, there exists a positive constant $\eps_0$ such that for any $\xi \in [1/2,1]$ and $0<\eps \leq \eps_0$,
\[ \limsup_{K \to \infty} \P \big(R^2_{2\eps^\xi} \leq T^1_{\eps^\xi} \wedge T_0^1 \big)=0. \]
\end{lemma} 
Let us now explain how to obtain Lemma~\ref{lemma-residentsstaydormants} via an adaptation of the proof of \cite[Lemma 4.2]{BT19}. 
\begin{proof}
We prove the lemma via coupling the rescaled population size $N_{2,t}^K$ with two birth-and-death processes, $N_{2,t}^-$ and $N_{2,t}^+$, on time scales where the mutant population size remains still small compared to $K$. More precisely, following \cite[Section 3.1.2]{C+19},
\[ N_{2,t}^- \leq N_{2,t}^K \leq N_{2,t}^+, \qquad \text{a.s.} \qquad \forall t \leq T_0^1 \wedge T_{\eps^\xi}^1. \numberthis\label{firstresidentcouplingdormants} \]

In order to satisfy \eqref{firstresidentcouplingdormants}, the processes $N_2^-=(N_{2,t}^-)_{t \geq 0}$ and $N_2^+=(N_{2,t}^+)_{t \geq 0}$ can be defined with the following birth and death rates
\begin{align*}
N_{2,t}^- \colon \quad & \frac{i}{K} \to \frac{i+1}{K} \quad \text{at rate} \quad i\lambda_2, \\
                 & \frac{i}{K} \to \frac{i-1}{K} \quad \text{at rate} \quad i \big( \mu+C \frac{i}{K} + C \eps^\xi \big).
\end{align*}
and
\begin{align*}
N_{2,t}^+ \colon  \quad & \frac{i}{K} \to \frac{i+1}{K} \quad \text{at rate} \quad i\lambda_2+\tau \eps^\xi, \\
                 & \frac{i}{K} \to \frac{i-1}{K} \quad \text{at rate} \quad i \big( \mu+C \frac{i}{K} \big).
\end{align*}
Using this coupling, which is based on the same principle as the one satisfying \eqref{mutantcouplingHGTguys}, the proof can be continued analogously to the one of \cite[Lemma 4.2]{BT19}. Following the arguments of that proof, one arrives at the point from where the lemma follows as soon as one shows that for $V>0$ small enough, one has
\[ \lim_{K \to \infty} \e^{-KV} \E\big( R_{2\eps^\xi}^1 \wedge T_0^1 \wedge T_{\eps^\xi}^1 \big) =0. \numberthis\label{FWbound} \]
Now, we can again verify \eqref{FWbound} similarly to \cite[Section 3.1.2]{C+19} and hence also to the proof of Lemma~\ref{lemma-residentsstayHGTguys} (note that in \cite{BT19} only the case when the approximating branching process $((\widehat N_{1a}(t),\widehat N_{1d}(t)))_{t \geq 0}$ is supercritical was handled, which allowed for a simpler proof). Since we have
\[ \E \big[R_{2\eps^\xi}^2 \wedge T_0^1 \wedge T_{\eps^\xi}^1 \big] \leq \E \Big[\int_0^{R_{2\eps^\xi}^2 \wedge T_0^1 \wedge T_{\eps^\xi}^1} K N_{1,t}^K \d t \Big], \]
it suffices to show that there exists $\widetilde C>0$ such that 
\[ \E \Big[\int_0^{R^2_{2\eps^\xi} \wedge T_0^1 \wedge T_{\eps^\xi}^1} K N_{1,t}^K \d t \Big] \leq \widetilde C \eps^\xi K. \numberthis\label{expectedstoppingtimedormants} \]
To this end, it is enough to show that there exists a function $g \colon (\smfrac{1}{K} \N_0)^3 \to \R$ defined as
\[ g(n_{1a},n_{1d},n_{2})=\gamma_1 n_{1a} + \gamma_2 n_{1d} \numberthis\label{gamma1gamma2dormants} \]
such that
\[ \Lcal g(\mathbf N_{t}^K) \geq N_t^K, \numberthis\label{largegeneratordormants}\]
where $\Lcal$ is the infinitesimal generator (cf.~Appendix~\ref{sec-generator}) of $(\mathbf N_t^K)_{t \geq 0}$ and we write $N_t^K=N_{1a,t}^K+N_{1d,t}^K$.
Indeed, if \eqref{largegeneratordormants} holds, then by Dynkin's formula,
\[ 
\begin{aligned}
\E \Big[\int_0^{R_{2\eps^\xi}^2 \wedge T_0^1 \wedge T_{\eps^\xi}^1} K N_{1,t}^K \d t \Big] & \leq  \E \Big[\int_0^{R_{2\eps^\xi}^2\wedge T_0^1 \wedge T_{\eps^\xi}^1} K\Lcal g(\mathbf N_{t}^K)\d t \Big] =\E\big[ K g(\mathbf N_{R_{2\eps^\xi}^2\wedge T_0^2 \wedge T_{\eps^\xi}^1}^K)-Kg(\mathbf N_{0}^K) \big] \\
& \leq (\gamma_1 \vee \gamma_2) \eps^\xi K - ( \gamma_1 \wedge \gamma_2 )
\end{aligned}
\]
follows, which implies \eqref{expectedstoppingtimedormants}, independently of the signs of $\gamma_1$ and $\gamma_2$. Here, Dynkin's formula can indeed be applied because $ \E [R_{2\eps^\xi}^2\wedge T_0^1 \wedge T_{\eps^\xi}^1]$ is finite. That holds because given our initial conditions, with positive probability the single initial active trait 1 individual dies due to natural death within a unit length of time before any event of the process $\mathbf N_t^K$ occurs, and hence already $T_0^1$ is stochastically dominated by a geometric random variable, which has all moments. We now apply the infinitesimal generator $\mathcal L$ to the function $g$ introduced in \eqref{gamma1gamma2dormants} once again, which yields
\[ \Lcal g(\mathbf N_t^K) = N_{1a,t}^K \big[ (\lambda_1-\mu-C (N_{2,t}^K+N_{1a,t}^K)-\tau N_{2,t}^K)\gamma_1 +pC (N_{2,t}^K+N_{1a,t}^K) \gamma_2 \big] + N_{1d,t}^K \big[ \sigma  \gamma_1 -(\kappa\mu+\sigma)  \gamma_2 \big]. \] 
Hence, according to \eqref{gamma1gamma2dormants}, it suffices to show that there exist $\gamma_1,\gamma_2 \in \R$ such that the following system of inequalities is satisfied:
\begin{align}
    (\lambda_1-\mu-C (N_{1a,t}^K+N_{2,t}^K)-\tau N_{2,t}^K)\gamma_1 +pC (N_{1a,t}^K+N_{2,t}^K) \gamma_2 & >1 \label{first-gamma1gamma2dormants} 
\end{align}
and
\begin{align}
     \sigma \gamma_1 -(\kappa\mu+\sigma)\gamma_2 & >1. \label{second-gamma1gamma2dormants}
\end{align}
We claim that as long as $t \leq R_{2\eps^\xi}^2 \wedge T_0^1 \wedge T_{\eps^\xi}^1$, the matrix
\[ J^K(t) = \begin{pmatrix} \lambda_1-\mu-C N_{1a,t}^K -(C+\tau) N_{2,t}^K & p C (N_{1a,t}^K + N_{2,t}^K) \\ \sigma & -\kappa\mu-\sigma \end{pmatrix}\]
is entrywise close to the mean matrix $J$ introduced in \eqref{JdefHGT}, namely (noting that the two matrices have the same second row), there exists a constant $\widetilde C>0$ such that
\[ \big| \big( J^K(t)-J \big)_{j} \big| \leq \widetilde C \eps^{\xi}, \qquad  \forall j\in \{1,2\}. \numberthis\label{JJclosedormants} \]
Indeed, as long as $t \leq R_{2\eps^\xi}^2 \wedge T_0^1 \wedge T_{\eps^\xi}^1$, for $j=1$ we have  that
\[ |(\lambda_1-\mu-C N_{1a,t}^K -(C+\tau) N_{2,t}^K)-(\lambda_1-\mu-(C+\tau)\bar n_2) | \leq 2(C+\tau)\eps^\xi + C \eps^\xi, \]
whereas for $j=2$ we have 
\[ |pC(N_{1a,t}^K+N_{2,t}^K)-pC\bar n_2| \leq p C(2\eps^\xi+\eps^\xi). \]
This implies \eqref{JJclosedormants} for eg.~$\widetilde C = 3C+2\tau$. 

Let us now choose $(\gamma_1,\gamma_2)$. 
Thanks to the assumption that $\lambda_1>\mu$, given that $\eps>0$ is sufficiently small, $J^K(t) + ((C+\tau) \bar n_2+2(\kappa\mu+\sigma))\mathrm{Id} $ is a matrix with positive entries, where $\mathrm{Id}$ denotes the $2\times 2$ identity matrix. Hence, writing $u_0=(C+\tau) \bar n_2+2(\kappa\mu+\sigma)$, it follows from the Perron--Frobenius theorem that there exists a strictly positive right eigenvector $\widetilde \Gamma=(\widetilde \gamma_1,\widetilde \gamma_2)$ of $J+u_0\mathrm{Id}$ corresponding to the eigenvalue $\widetilde \lambda +u_0$. Then we have 
\[ (J+u_0 \mathrm{Id}) \widetilde \Gamma^T = (\widetilde \lambda + u_0)\widetilde \Gamma^T, \]
and thus also $J\widetilde \Gamma^T = \widetilde \lambda \widetilde \Gamma^T$. Since by assumption $\widetilde \lambda \neq 0$ and $\widetilde \Gamma$ has two positive coordinates, we obtain that
\[ \Gamma: = (\gamma_1,\gamma_2):=2(\widetilde\lambda( \widetilde \gamma_1 \wedge \widetilde \gamma_2))^{-1} \widetilde \Gamma\]
is well-defined, and it solves 
\[ J \Gamma^T = \widetilde \lambda \Gamma^T, \numberthis\label{GammaEVeq} \] further, $\widetilde\lambda \gamma_i \geq 2$ holds for all $i \in \{1,2\}$. Now, using \eqref{GammaEVeq} and \eqref{JJclosedormants}, we obtain
\[
\begin{aligned}
\big| \gamma_1 & (\lambda_1-\mu-C N_{1a,t}^K -(C+\tau) N_{2,t}^K)+\gamma_2 pC(N_{1a,t}^K+N_{2,t}^K) - \widetilde \lambda \gamma_1 \big| \\ & =  \big| \gamma_1 (C N_{1a,t}^K +(C+\tau) N_{2,t}^K-(C+\tau)\bar n_2) + \gamma_2 pC (N_{1a,t}^K+N_{2,t}^K-\bar n_2)\big| \leq (|\gamma_1|+|\gamma_2|) \widetilde C \eps^\xi. 
\end{aligned}\]
Finally, since $\widetilde \lambda \gamma_1 \geq 2$, it follows that if $\eps>0$ is small enough, then as long as $t \leq R_{2\eps^\xi}^2 \wedge T_0^1 \wedge T_{\eps^\xi}^1$, we have
\[ \gamma_1(\lambda_1-\mu-C N_{1a,t}^K -(C+\tau) N_{2,t}^K) + \gamma_2(p C (N_{1a,t}^K + N_{2,t}^K)) \geq 1. \]
This (together with the fact that $J$ and $J^K(t)$ have the same second row for any positive $K$ and $t$) implies \eqref{largegeneratordormants}, and hence the proof of Lemma~\ref{lemma-residentsstaydormants} is concluded. 
%
\end{proof}
\begin{proof}[Proof of Proposition~\ref{prop-firstphasedormants}.] In what follows, similarly to \cite[Section 4.1]{BT19}, we consider our population process on the event
\[ A_\eps := \{ T_0^1 \wedge T_{\eps^\xi}^1 < R_{2\eps^\xi}^2 \} \]
for sufficiently small $\eps>0$. On this event, the invasion or extinction of the mutant population will happen before the resident population substantially deviates from its equilibrium size. We couple on $A_\eps$ the process $(KN_{1a,t}^K,KN_{1d,t}^K)$ with two two-type branching processes $(N_{1a,t}^{\eps,-}, N_{1d,t}^{\eps,-})$ and $(N_{1a,t}^{\eps,+}, N_{1d,t}^{\eps,+})$ on $\N_0^2$ (these again depend on $K$, but we omit that from the notation) such that almost surely, for any $t <  t_\eps:=T_0^1 \wedge T_{\eps^\xi}^1 \wedge R_{2\eps^\xi}^2$ and $\upsilon \in \{ a,d \}$,
\[ \begin{aligned}
N_{1\upsilon,t}^{\eps,-} & \leq \widehat N_{1\upsilon}(t)\leq N_{1\upsilon,t}^{\eps,+}, \\
N_{1\upsilon,t}^{\eps,-} & \leq K N_{1\upsilon,t}^K \leq N_{1\upsilon,t}^{\eps,+},
\end{aligned} \numberthis\label{upsilondefdormants}
\]
where the approximating branching process $((\widehat N_{1a}(t),\widehat N_{1d}(t)))_{t \geq 0}$ was defined in Section~\ref{sec-phase13}. 
We claim that in order to satisfy \eqref{upsilondefdormants}, these processes can be defined with the following jump rates:
\[ \begin{aligned}
(N_{1a,t}^{\eps,-},N_{1d,t}^{\eps,-}) \colon \quad & (i,j) \to (i+1,j) & \text{at rate} & \quad i \lambda_1, \\
& (i,j) \to (i-1,j)& \text{at rate}  &\quad i(\mu+(1-p)C(\eps^\xi+\bar n_2+2\eps^\xi)+5pC\eps^\xi \\ & & & \qquad +\tau (\bar n_2+2\eps^\xi)), \\
& (i,j) \to (i-1,j+1) & \text{at rate}& \quad ipC(\bar n_2-2\eps^\xi), \\
& (i,j) \to (i+1,j-1) & \text{at rate} &\quad j\sigma, \\
& (i,j) \to (i,j-1) & \text{at rate} &\quad j\kappa\mu,
\end{aligned}
\]
and
\[ \begin{aligned}
(N_{1a,t}^{\eps,+},N_{1d,t}^{\eps,+}) \colon \quad & (i,j) \to (i+1,j) & \text{at rate} & \quad i \lambda_1, \\
& (i,j) \to (i-1,j)& \text{at rate}  &\quad i(\mu+(1-p)C(\bar n_2-2\eps^\xi)-5pC\eps^\xi), \\
& (i,j) \to (i-1,j+1) & \text{at rate}& \quad ipC(\bar n_2+2\eps^\xi+\eps^\xi), \\
& (i,j) \to (i+1,j-1) & \text{at rate} &\quad j\sigma, \\
& (i,j) \to (i,j-1) & \text{at rate} &\quad j\kappa\mu.
\end{aligned}
\]
The idea of this coupling is similar to the one of the coupling satisfying \eqref{firstresidentcouplingHGTguys}. Given this coupling, the proof of Proposition~\ref{prop-firstphasedormants} can be completed analogously to the one of \cite[Proposition 4.1]{BT19} (which uses further arguments of \cite{C+19}), noting that for the mutant population of trait 1, HGT is just a special case of death by competition. Here, we also note that although the case when $\widetilde \lambda <0$ was not considered in \cite{BT19}, in this case the same proof techniques can be applied as for $\widetilde\lambda>0$. 
\end{proof}

\subsection{The second phase of invasion: Lotka--Volterra phase}\label{sec-phase2dormants}
Assume that \eqref{firstinmutantlessfit2ineq} holds. On the event $\{ T_{\sqrt \eps}^1 < T_0^1 \wedge R_{2\sqrt\eps}^2 \} \subset A_{\eps}$, after time $T_{\eps}^1$ the total mutant (trait 1) population has size close to $\eps K$. Note that Proposition~\ref{prop-firstphasedormants} does not tell about the proportion of dormant and active mutants at time $T_\eps^1$. Similarly to \cite[Section 4.2.2]{BT19}, we show that with high probability, there exists a point in time in the interval $[T_\eps^1,T_{\sqrt \eps}^1]$ such that at this time, the total mutant population size is still comparable to $\eps K$, and the proportion of the active and dormant mutants is near the equilibrium proportion $((\widehat N_{1a}(t),\widehat N_{1d}(t)))_{t \geq 0}$. We verify this in Section~\ref{sec-KestenStigum}, using arguments similar to \cite{C+19,BT19}. Starting from such a proportion of active and dormant mutants, the limiting dynamical system \eqref{3dimHGT} will converge to $(n_{1a},n_{1d},n_2)$ in the case \eqref{secondinmutantlessfit2ineq} of stable coexistence and to $(\bar n_{1a}, \bar n_{1d},0)$ in the case \eqref{reversesecondin} of mutant fixation; we will prove this in Section~\ref{sec-ODEconvergencedormants}.
\subsubsection{Kesten--Stigum arguments}\label{sec-KestenStigum}
If~ \eqref{firstinmutantlessfit2ineq} holds, then $\widetilde \lambda$ is positive. Hence, the Kesten--Stigum theorem (see eg.~\cite[Theorem 2.1]{GB03}) ensures that
we have
\[ \Big( \frac{\widehat N_{1a}(t)}{\widehat N_{1a}(t)+\widehat N_{1d}(t)},\frac{\widehat N_{1d}(t)}{\widehat N_{1a}(t)+\widehat N_{1d}(t)}\Big) \underset{K \to \infty}{\longrightarrow} (\pi_{1a},\pi_{1d}) \]
almost surely, conditional on the survival of the approximating branching process $(\widehat N_{1a}(t),\widehat N_{1d}(t))_{t \geq 0}$,
where $(\pi_{1a},\pi_{1d})$ is the positive left eigenvector of the matrix $J$ defined in \eqref{JdefHGT} associated to $\widetilde \lambda$ such that $\pi_{1a}+\pi_{1d}=1$, which can be computed explicitly according to \eqref{lambdaHGT}.  We have the following proposition. 
\begin{prop}\label{prop-secondphasedormants}
Assume that~\eqref{firstinmutantlessfit2ineq} holds (in other words, $\widetilde\lambda>0$) with $\lambda_2>\mu$. Then there exists $\widehat C>0$ sufficiently large such that for $\delta>0$ such that $\pi_{1a} \pm \delta \in (0,1)$, under the same assumptions as Proposition~\ref{prop-firstphasedormants},
\[ \numberthis\label{mutantproportionsdormants}
\begin{aligned}
& \liminf_{K \to \infty} \P \Big( \exists t \in \big[ T_\eps^1, T^1_{\sqrt \eps}\big], \frac{\eps K}{\widehat C} \leq KN_{1a,t}^K+KN_{1d,t}^K \leq \sqrt \eps K, \\ & \qquad \pi_{1a}-\delta < \frac{ N_{1a,t}^K}{ N_{1a,t}^K+N_{1d,t}^K} < \pi_{1a}+\delta \Big| T^1_{\sqrt \eps} < T_0^1 \wedge R_{2\sqrt\eps}^1 \Big) \geq 1-o_\eps(1).
\end{aligned}
\]
\end{prop}
The proof of this proposition is very similar to the one of~\cite[Proposition 4.4]{BT19}, which employs some arguments of the proof of~ \cite[Proposition 3.2]{C+19}. Compared to the setting of~\cite{BT19}, only the additional HGT terms have to be handled, which can be treated very similarly to the terms for death by competition. We refrain from presenting further details here. 

\subsubsection{Convergence of the dynamical system}\label{sec-ODEconvergencedormants}
Next, we show that for $\widetilde\lambda>0$, starting from an initial condition where the dormant/active proportion of mutants is sufficiently balanced, the system converges towards the coexistence equilibrium $(n_{1a},n_{1d},n_2)$ under the assumption \eqref{firstinmutantlessfit2ineq} and to the monomorphic equilibrium $(\bar n_{1a},\bar n_{1d},0)$ of trait 1 under the assumption \eqref{reversefirstin}. The proof of this lemma uses ideas from the proof of \cite[Lemma 4.7]{BT19}, however, the extension of that proof is not immediate because of the additional HGT terms and possible coexistence.
\begin{lemma}\label{lemma-1invades2LV}
Assume that~\eqref{firstinmutantlessfit2ineq} holds with $\lambda_2>\mu$. Let us consider the system of ODEs \eqref{3dimHGT}. Then, if for $t=0$ we have 
\[ \frac{pC(n_{1a}(t)+n_{2}(t))}{\kappa\mu+\sigma} > \frac{n_{1d}(t)}{n_{1a}(t)}, \numberthis\label{proportioncond1} \]
further,
\[ \frac{n_{1d}(t)}{n_{1a}(t)} > \frac{\mu-\lambda_1+C(n_2(t)+n_{1a}(t))+\tau n_2(t)}{\sigma}, \numberthis\label{proportioncond2} \]
and
\[ \qquad n_2(t) \geq 0, n_{1a}(t),n_{1d}(t)>0, \numberthis\label{+0cond}\]
then under condition \eqref{reversesecondin}, the following holds:
\[ \lim_{t \to \infty} (n_{1a}(t),n_{1d}(t), n_{2}(t))=(\bar n_{1a},\bar n_{1d},0),
\numberthis\label{goodlimitfixation1} \]
whereas under condition \eqref{secondinmutantlessfit2ineq}, there exists $\eps>0$ such that if $(n_{1a}(0),n_{1d}(0),n_2(0))$ additionally satisfies \[
n_{1a}(0)+n_{2}(0) \leq \sqrt \eps \text{ and } \bar n_2-2\sqrt\eps \leq n_2(0) \leq \bar n_2 + 2\sqrt\eps, \numberthis\label{initialconddormants} \] then we have
\[ \lim_{t \to \infty} (n_{1a}(t),n_{1d}(t), n_{2}(t))=(n_{1a},n_{1d},n_2).
\numberthis\label{goodlimitcoexistence1} \]
\end{lemma}
\begin{proof}
Note that \eqref{proportioncond1} is equivalent to the condition that $\dot n_{1d}(t)>0$ and \eqref{proportioncond2} is equivalent to the condition that $\dot n_{1a}(t)>0$. Hence, as long as \eqref{proportioncond1} and \eqref{proportioncond2} hold, $t \mapsto n_{1a}(t)$ and $t \mapsto n_{1d}(t)$ are strictly increasing. 

Let us assume that \eqref{proportioncond1}, \eqref{proportioncond2}, and \eqref{+0cond} are satisfied for $t=0$. We claim that then both of these inequalities hold for all $t>0$ unless for some $t>0$,
\[ 
\frac{pC(n_{1a}(t)+n_{2}(t))}{\kappa\mu+\sigma} = \frac{n_{1d}(t)}{n_{1a}(t)} = \frac{\mu-\lambda_1+C(n_2(t)+n_{1a}(t))+\tau n_2(t)}{\sigma}. \numberthis\label{1a1dequilibrium}
\]
Indeed, let us assume that for some $t>0$, $(n_{1a}(t),n_{1d}(t),n_{2}(t))$ lies on the boundary of the set
\[ \begin{aligned}
 U=\{ & (\widehat n_{1a},\widehat n_{1d},\widehat n_{2}) \in [0,\infty) \times (0,\infty) \times (0,\infty) \colon \text{\eqref{proportioncond1},  \eqref{proportioncond2}, and \eqref{+0cond} hold for }\\ &  (n_{1a}(t),n_{1d}(t), n_{2}(t))= (\widehat n_{1a},\widehat n_{1d},\widehat n_{2}) \}
\end{aligned} \numberthis\label{propcondset1} \]
with $n_{1a},n_{1d}>0$. Then one of the following conditions holds:
    \begin{enumerate}[(i)]
    \item\label{first-bdry1} $\dot n_{1d}(t)=0$, $\dot n_{1a}(t)>0$,
    \item\label{second-bdry1} $\dot n_{1a}(t)=0$, $\dot n_{1d}(t)>0$,
    \item\label{third-bdry1} $\dot n_{1a}(t)=\dot n_{1d}(t)=0$.
    \end{enumerate}
In case \eqref{first-bdry1} we have
\[ \frac{\d}{\d t} \Big[ \frac{n_{1d}(t)}{n_{1a}(t)} \Big]  =\frac{-\dot n_{1a}(t) n_{1d}(t)}{n_{1a}(t)^2} < 0. \]
The case \eqref{second-bdry1} yields
\[ \frac{\d}{\d t} \Big[ \frac{n_{1d}(t)}{n_{1a}(t)} \Big] =\frac{\dot n_{1d}(t) n_{1a}(t)}{n_{1a}(t)^2} > 0. \]
We conclude that $s \mapsto (n_{1a}(s),n_{1d}(s),n_2(s))$ cannot enter the interior of the complement of the open set  $U$ defined in \eqref{propcondset1}. Hence, in case the solution reaches the boundary of the set $U$, after an arbitrarily short time, either the solution will be situated in $U$ again or condition~\eqref{third-bdry1} will still be satisfied.

Since solutions of \eqref{3dimHGT} with coordinatewise nonnegative initial conditions are bounded (depending only on the initial condition), we infer from from the previous consideration, using 
monotonicity, that $t \mapsto (n_{1a}(t),n_{1d}(t),n_2(t))$ converges along a subsequence to some $(n_{1a}^*,n_{1d}^*,n_2^*)$ such 
that for $(n_{1a}(t),n_{1d}(t),n_2(t))=(n_{1a}^*,n_{1d}^*,n_2^*)$, the inequalities \eqref{proportioncond1} and \eqref{proportioncond2} are replaced by the 
corresponding equalities (whereas \eqref{+0cond} is still satisfied). But this convergence implies that at $(n_{1a}^*,n_{1d}^*,n_2^*)$, the time derivative of the $n_2$-coordinate of \eqref{3dimHGT} must be equal to zero. This can only happen if $(n_{1a}^*,n_{1d}^*,n_2^*)$ is an equilibrium of \eqref{3dimHGT} with $n_{1a}^*$ and $n_{1d}^*$ positive and $n_2^*$ nonnegative, i.e., either $(n_{1a}^*,n_{1d}^*,n_2^*)=(n_{1a},n_{1d},n_2)$ or $(n_{1a}^*,n_{1d}^*,n_2^*)=(\bar n_{1a},\bar n_{1d},0)$.  

In case \eqref{reversesecondin} holds, only the latter case is possible because the coexistence equilibrium $(n_{1a},n_{1d},n_2)$ does not exist. Hence, thanks to the boundedness of the solution, $t \mapsto (n_{1a}(t),n_{1d}(t),n_2(t))$ must converge along all subsequences to $(\bar n_{1a},\bar n_{1d},0)$, which implies \eqref{goodlimitfixation1}. 

In case \eqref{secondinmutantlessfit2ineq} is satisfied, our goal is to show that only  $(n_{1a}^*,n_{1d}^*,n_2^*)=(n_{1a},n_{1d},n_2)$ is possible.  We claim that under \eqref{firstinmutantlessfit2ineq} we have (regardless of whether $\lambda_2>\mu$)
\[ \bar n_{1a} > n_{1a}+n_2 > \widetilde n_2, \numberthis\label{equilibriumordering} \]
where we recall the notation $\widetilde n_2 = \frac{\lambda_2-\mu}{C}$. Indeed, by part~\eqref{compeq} of Lemma~\ref{lemma-effectivecomp}, we have
\[ 0 = \lambda_2-\mu-C \widetilde n_2 = \lambda_2-\mu - C(n_{1a}+n_{2}) + \tau n_{2}, \]
which implies the second inequality in \eqref{equilibriumordering}. In order to verify the first inequality in~\eqref{equilibriumordering}, we note that in case $(n_{1a},n_{1d},n_2)$ exists as a coordinatewise positive equilibrium, we have according to \eqref{coexeq} that
\[ n_{1a}+n_2 = \frac{(\lambda_2-\lambda_1)(\kappa\mu+\sigma)}{Cp\sigma-(\kappa\mu+\sigma)\tau} \]
and
\[ \bar n_{1a} = \frac{(\lambda_1-\mu)(\kappa\mu+\sigma)}{C(\kappa\mu+(1-p)\sigma)}. \]
Now, since by assumption~\eqref{mutantlessfit2ineq} holds, thanks to Lemma~\ref{lemma-lessconditions} this also implies that $Cp\sigma-(\kappa\mu+\sigma)<0$. Given this and using elementary computations, we conclude that $n_{1a}+n_2<\bar n_{1a}$ is equivalent to \eqref{secondinmutantlessfit2ineq}, which holds by assumption. Hence, the claim follows.


Let now $\eps>0$ so small that 
\[ \bar n_{1a} > n_{1a}+n_2 > \widetilde n_2 + 3\sqrt\eps; \]
such $\eps$ exists thanks to \eqref{equilibriumordering}. Under the condition \eqref{initialconddormants} that $n_{1a}(0)+n_{1d}(0) \leq \sqrt \eps$ and $n_2(0) \leq \widetilde n_2 + 2\sqrt\eps=\bar n_2+2\sqrt\eps$ (where we used that $\lambda_2>\mu$), we have that $n_{1a}(0) + n_{2}(0) \leq \widetilde n_2 + 3\sqrt \eps$. Assume now for a contradiction that $(n_{1a}^*,n_{1d}^*,n_2^*)=(\bar n_{1a},\bar n_{1d},0)$. Then we have
\[ \liminf_{t \to \infty} n_{1a}(t)+n_{2}(t) > n_{1a}+n_2. \]
Hence, thanks to the continuity of the solutions of \eqref{3dimHGT} and Bolzano's theorem, there exists $t>0$ such that
\[ n_{1a}(t) + n_2(t)=n_{1a}+n_2. \numberthis\label{goodactivesum} \]
At the same time, we still have that
\[ \frac{pC(n_{1a}(t)+n_{2}(t))}{\kappa\mu+\sigma} \geq \frac{n_{1d}(t)}{n_{1a}(t)} \geq \frac{\mu-\lambda_1+C(n_2(t)+n_{1a}(t))+\tau n_2(t)}{\sigma}. \numberthis\label{bothnonstrict} \]
If both inequalities hold with an equality, then we immediately have $(n_{1a}(t),n_{1d}(t),n_2(t))=(n_{1a},n_{1d},n_2)$, and hence $(n_{1a}(s),n_{1d}(s),n_2(s))=(n_{1a},n_{1d},n_2)$ for all $s \geq t$, which contradicts the assumption that $(n_{1a}^*,n_{1d}^*,n_2^*)=(\bar n_{1a},\bar n_{1d},0)$. Else, there exists $\delta>0$ such that for all $s \in (t,t+\delta)$, both inequalities of \eqref{bothnonstrict} hold with a strict inequality after we replace $s$ with $t$ everywhere. This implies that $ n_{2}(s) < n_2$ for all $s \in (t,t+\delta)$, from which together with \eqref{goodactivesum} it follows that $n_{1a}(t)>n_{1a}$. Now, we have
\[ \dot n_{2}(s)=n_2(s) (\lambda_2-\mu-C(n_{1a}(s)+n_2(s))+\tau n_{1a}(s)) = n_2(s) (\lambda_2-\mu-C(n_{1a}+n_2)+\tau n_{1a}(t))>0, \numberthis\label{n2not0} \]
where we used that $\lambda_2-\mu-C(n_{1a}+n_2)+\tau n_{1a}=0$ since $(n_{1a},n_{1d},n_2)$ is an equilibrium. It follows that $r \mapsto n_2(t+r)$ is increasing on $(0,\delta)$. Since the right-hand side of~\eqref{n2not0} is strictly positive, one can easily improve these arguments in order to see that if $n_2(0)>0$ (which is true thanks to~\eqref{initialconddormants} for all $\eps>0$ sufficiently small), then $\liminf_{s \to \infty} n_2(s)>0$. This contradicts the convergence of $s \mapsto (n_{1a}(s),n_{1d}(s),n_2(s))$ to $(\bar n_{1a},\bar n_{1d},0)$ along a subsequence. We conclude that along all subsequences, the solution of \eqref{3dimHGT} converges to $(n_{1a},n_{1d},n_2)$, and hence the lemma follows.
\end{proof}
The last lemma of this section ensures that the state $(\pi_{1a},\pi_{1d},\bar n_{2})$ that the population process approaches according to Proposition~\ref{prop-secondphasedormants} is an admissible initial condition for Lemma~\ref{lemma-1invades2LV}.
\begin{lemma}\label{lemma-goodstartdormants}
Assume that \eqref{firstinmutantlessfit2ineq} holds with $\lambda_2>mu$. Let $\widehat C$ be chosen according to Proposition~\ref{prop-secondphasedormants}, further, $t>0$ and $n_{1a}(t),n_{1d}(t),n_{2}(t)>0$ such that $n_{2}(t) \in (\bar n_{2}-2\sqrt\eps,\bar n_{2}+2\sqrt\eps), n_{1a}+n_{1d} \in (\eps/\widehat C,\sqrt\eps)$, and $\smfrac{n_{1d}(t)}{n_{1a}(t)} = \smfrac{\pi_{1d}}{\pi_{1a}}$. Then, under the assumptions of the proposition, if $\eps>0$ is sufficiently small, then $(n_{1a}(t),n_{1d}(t),n_{2}(t))$ satisfies \eqref{proportioncond1}, \eqref{proportioncond2}, and \eqref{+0cond}. 
\end{lemma}
\begin{proof}
The fact that \eqref{+0cond} is satisfied for $\eps>0$ sufficiently small follows simply from the assumption $\lambda_2>\mu$. Moreover,
since $(\pi_{1a},\pi_{1d})$ is a left eigenvector of $J$ corresponding to the eigenvalue $\widetilde\lambda$, we have
\begin{align*}
    (\lambda_1-\lambda_2-\frac{\tau}{C}(\lambda_2-\mu))+\sigma \frac{\pi_{1d}}{\pi_{1a}}& =\widetilde\lambda, \\
    p(\lambda_2-\mu)-(\kappa\mu+\sigma)\frac{\pi_{1d}}{\pi_{1a}} & = \widetilde \lambda \frac{\pi_{1d}}{\pi_{1a}}
\end{align*}
Hence, since $\widetilde\lambda>0$, given that $\eps>0$ is small enough, the bounds \eqref{proportioncond1} and \eqref{proportioncond2} follow from estimating $\pi_{1d}/\pi_{1a}$ from above respectively from below according to the conditions of the lemma. Since these computations are elementary, we refrain from presenting the details here.
\begin{comment}{
we obtain
\[ 
\begin{aligned}
\frac{\pi_{1d}}{\pi_{1a}} & = \frac{\widetilde\lambda-\lambda_1+\lambda_2+\frac{\tau}{C}(\lambda_2-\mu)}{\sigma}=\frac{\widetilde\lambda-\lambda_1+\mu+C \big(\frac{\lambda_2-\mu}{C}\big)+\tau\big(\frac{\lambda_2-\mu}{C}\big)}{\sigma} \\ & > \frac{-\lambda_1+\mu+C\Big(\frac{\lambda_2-\mu}{C}+3\sqrt\eps\Big)+\tau \Big(\frac{\lambda_2-\mu}{C}+2\sqrt\eps\Big)}{\sigma}\geq \frac{\mu-\lambda_1+C(n_{1a}(t)+n_{2}(t))+\tau n_2(t)}{\sigma}
\end{aligned}\]
and
\[ \frac{\pi_{1d}}{\pi_{1a}} =\frac{pC\Big(\frac{\lambda_2-\mu}{C}\Big)}{\widetilde\lambda+\kappa\mu+\sigma} < \frac{pC\Big(\frac{\lambda_2-\mu}{C}-2\sqrt\eps\Big)}{\kappa\mu+\sigma}  \leq \frac{pC(n_{1a}(t)+n_{2}(t))}{\kappa\mu+\sigma},\]
as asserted.}\end{comment}
\end{proof}

\subsection{The third phase of invasion: extinction of the resident population}\label{sec-phase3dormants}

Let us assume that \eqref{firstinmutantlessfit2ineq} holds. After the second phase, the rescaled process $\mathbf N_t^K$ is close to the state $(n_{1a},n_{1d},n_2)$ under condition \eqref{secondinmutantlessfit2ineq} and to $(0,0,\bar n_{2})$ under condition \eqref{reversesecondin}. In the first case, there is no third phase; more precisely, for the stochastic population process, the coexistence equilibrium $(n_{1a},n_{1d},n_2)$ is only metastable, but with high probability, the process stays close to it for a time that is at least exponential in $K$ (see the beginning of Section~\ref{sec-phase3HGTguys} for details and corresponding references). In contrast, in the second case, there is a third phase, which takes $O(\log K)$ time units and ends with the extinction of the formerly resident trait 2 with high probability after a time that is in logarithmic in $K$, while trait 1 stays close to its two-coordinate equilibrium. More precisely, we have the following proposition, where we recall the set $S_\beta^1$ from \eqref{Sbeta1def} and the stopping time $T_{S_\beta^1}$ from \eqref{TSbetaidef}. 
\begin{prop}\label{prop-thirdphasedormants}
Assume that \eqref{firstinmutantlessfit2ineq} and \eqref{reversesecondin} hold with $\lambda_2>\mu$. There exist $\eps_0,C_0,c_0>0$ such that for all $\eps \in (0,\eps_0)$, if there exists $\eta \in (0,1/2)$ that satisfies
\[ \big| N^K_{1a}(0)-\bar n_{1a} \big| \leq \eps~\text{ and }~\big| N^K_{1d}(0) - \bar n_{1d} | \leq \eps ~\text{ and }~\eta \eps/2 \leq N^K_2(0) \leq \eps/2, \]
then
\begin{align*}
    & \forall E>(\mu+C\bar n_{1a}-\lambda_1)^{-1}+C_0\eps, \quad & \P(T_{S_{c_0\eps}^1} \leq E\log K) \underset{K \to \infty}{\longrightarrow} 1, \\
    & \forall 0 \leq E < (\mu+C\bar n_{1a}-\lambda_1)^{-1}-C_0\eps, \quad & \P(T_{S_{c_0\eps}^1} \leq E \log K) \underset{K \to \infty}{\longrightarrow} 0.
\end{align*}
\end{prop}
\begin{proof}[Sketch of proof]
This proposition is the analogue of Proposition~\ref{prop-thirdphaseHGTguys} for the third phase of the invasion with the roles of the two traits interchanged, in the case of fixation of mutants. On the other hand, the inital conditions in Proposition~\ref{prop-thirdphasedormants} are very similar to the ones in Proposition~\ref{prop-firstphaseHGTguys}, which investigates the first phase of invasion in the reverse invasion direction. The essential difference is that while Proposition~\ref{prop-thirdphaseHGTguys} only tells about the case when \eqref{firstinmutantlessfit2ineq} and \eqref{reversesecondin}, ie.\ when we expect that fixation of trait 1 holds, according to the properties of the dynamical system~\eqref{3dimHGT}. In particular, we have seen that under these conditions we have $\widetilde\lambda < 0$. 

Now, the proof of Proposition~\ref{prop-thirdphasedormants} proceeds as follows. First, similarly to the proof of Lemmas~\ref{lemma-residentsstayHGTguys} and \ref{lemma-residentsstaydormants}, we first show that the rescaled mutant population size vector $(N_{1a,t}^K,N_{1d,t}^K)$ stays close to its equilibrium $(\bar n_{1a},\bar n_{1d})$ for long times, given that the resident population is small. This can be verified via coupling this two-type population size process between two birth-and death processes on $[0,T^2_\eps]$. The coupled processes can be chosen analogously to the proof of Lemma~\ref{lemma-residentsstayHGTguys}, with the only difference being that in their definition, $\sqrt \eps$ must be replaced by $\eps$ everywhere.  Now, analogously to the proof of Lemmas~\ref{lemma-residentsstayHGTguys} and \ref{lemma-residentsstaydormants}, we can use results by Freidlin and Wentzell about exit of jump processes from a domain \cite[Section 5]{FW84} to show that with high probability, we can find $V>0$ and $c_0>0$ such that the following holds: 
\[ \lim_{K \to \infty} \P\big( R_{c_0\eps}^1 \leq \e^{KV} \wedge T^2_\eps)=0. \numberthis\label{mutantsstayinequilibriumdormants} \]
Now, we can find two branching processes $N_2^{\eps,\leq}=(N_{2,t}^{\eps,\leq})_{t \geq 0}$ and $N_2^{\eps,\geq}=(N_{2,t}^{\eps,\geq})_{t \geq 0}$ such that 
\[ N_{2,t}^{\eps,\leq} \leq K N_{2,t}^K \leq N_{2,t}^{\eps,\geq} \numberthis\label{subcriticalcouplingdormants} \]
almost surely on the time interval
\[ I_\eps^K = \big[ 0, R_{c_0\eps}^1  \wedge  T^2_\eps \big]. \]
Indeed, in order that \eqref{subcriticalcouplingdormants} is satisfied, the processes $N_{2}^{\eps,\geq}$ and $N_{2}^{\eps,\leq}$ can be defined with the following rates and initial conditions:
\[ \begin{aligned}
N_{2}^{\eps,\leq} \colon & i \to i+1 & \text{ at rate } & i(\lambda_1+\tau\eps  (\bar n_{1a}-c_0\eps)), 
& i \to i-1 & \text{ at rate } i \big(\mu + C \big(\bar n_{1a} + (c_0+1)\eps\big)\big),
\end{aligned}
\]
started from $\lfloor \eta \eps K \rfloor$, and
\[ \begin{aligned}
N_{2}^{\eps,\geq} \colon & i \to i+1 & \text{ at rate } & i(\lambda_1+\tau \eps(\bar n_{1a}+(c_0+1)\eps)),
& i \to i-1 & \text{ at rate } i (\mu + C (\bar n_{1a} - c_0\eps)),
\end{aligned}
\]
started from  $\lfloor \eps K \rfloor +1$. This coupling is based on the same principle as the one fulfilling \eqref{mutantcouplingHGTguys}.

For all $\eps>0$ sufficiently small, both of these branching processes are subcritical, cf.~Section~\ref{sec-phase13}. The growth rates of these three processes are $\lambda_2-\mu-C \bar n_{1a} \pm O(\eps)=\widehat \lambda + O(\eps)$. From this, analogously to \cite[Section 3.3]{C+19}, we derive that the extinction time of these processes started from $[\lfloor \eta K \eps \rfloor, \lfloor \eps K \rfloor +1]$ is of order $(-\widehat \lambda+O(\eps))\log K$. This in turn follows from the fact that for a single-type branching process $\mathfrak N=(\mathfrak N(t))_{t \geq 0}$ with birth rate $\Bcal>0$ and death rate $\Dcal>0$ that is subcritical (i.e., $\Bcal<\Dcal$), the assertions \eqref{extinction2typelower} and \eqref{extinction2typeupper} are applicable. Here, one considers the branching process as a two-type branching process where the second component stays zero for all times, and the Lyapunov exponent $\bar \lambda$ equals $(\Dcal-\Bcal)^{-1}$. 
Further, if $\mathfrak N(0)=\lfloor \eta K \eps \rfloor$, then for all small enough $\eps>0$,
\[ \lim_{K \to \infty} \P\Big( \mathcal S_0^{\Ncal} > K \wedge \mathcal S_{\lfloor \eps K \rfloor }^{\mathcal N} \Big) = 0, \numberthis\label{diebeforeyougrowupdormants} \]
analogously to~\eqref{diebeforeyougrowupHGTguys}. Now for $C \geq 0$ we can estimate as follows
\[ 
\begin{aligned}
\P(T_0^2 < C \log K) - \P\big( \mathcal S_0^{N_1^{\eps,\leq}}  < C \log K \big)
\leq & \P\big( T_0^2 > T_\eps^2 \wedge K \big) + \P\big( T_\eps^2 \wedge K > R_{c_0\eps}^1  \big) \\
\leq & \P\big(\mathcal S_0^{N_1^{\eps,\geq}} > \mathcal S_{\lfloor \eps K \rfloor}^{N_1^{\eps,\geq}} \wedge K \big) +\P\big( T_\eps^2 \wedge K > R_{c_0\eps}^1 \big).
\end{aligned} \numberthis\label{strangebounddormants} \]
The first inequality in~\eqref{strangebounddormants} can be verified analogously to~\eqref{proofofspeciation}.
Once $\eps>0$ is small enough, the second term in the last line tends to zero as $K \to \infty$ according to \eqref{mutantsstayinequilibriumdormants} and so does the first one according to \eqref{diebeforeyougrowupdormants}. We conclude that
\[ \limsup_{K \to \infty} \P\big( T_0^2 < E\log K \big) \leq \lim_{K \to \infty} \P\big( \mathcal S_0^{N_1^{\eps,\leq}} \leq E \log K \big) \]
and
\[ \liminf_{K \to \infty} \P\big( T_0^2 < E\log K \big) \geq \lim_{K \to \infty} \P\big( \mathcal S_0^{N_1^{\eps,\geq}} \leq E\log K \big), \]
which implies the proposition.
\end{proof}
\subsection{Proof of Theorems~\ref{thm-invasionof1} and \ref{thm-failureof1}}
Putting together Propositions~\ref{prop-firstphasedormants}, \ref{prop-secondphasedormants}, and \ref{prop-thirdphasedormants}, we now prove our main results, employing some arguments from \cite[Section 3.4]{C+19}, similarly to \cite[Section 3.4]{BT19}, handling the additional effect of HGT and the eventual convergence to the coexistence equilibrium. Throughout the proof we will assume that $\lambda_2>\mu$. Our proof strongly relies on the coupling \eqref{upsilondefdormants}. More precisely, we define a Bernoulli random variable $B$ as the indicator of nonextinction
\[ B:= \mathds 1 \{ \forall t>0 \colon \widehat N_{1a}(t)+\widehat N_{1d}(t)>0 \} \]
of the approximating branching process $((\widehat N_{1a}(t), \widehat N_{1d}(t)))_{t \geq 0}$ defined Section~\ref{sec-phase13}, which is initially coupled with $((K N_{1a,t}^K,K N_{1d,t}^K))_{t \geq 0}$ according to \eqref{upsilondefdormants}. 
Let $f$ be the function defined in Proposition~\ref{prop-firstphasedormants}. Throughout the rest of the proof, we can assume that $\eps>0$ is so small that $f(\eps) <1$.

Our goal is to show that
\[ \liminf_{K \to \infty} \mathcal E(K,\eps) \geq q_1-o(\eps) \numberthis\label{extinctionlowerdormants}\]
holds for 
\[ \mathcal E(K,\eps):= \P \Big( \frac{T_0^1}{\log K} \leq f(\eps), T_0^1 < T_{S_\beta^1} \wedge T_{S_\beta^{\rm co}}, B=0 \Big).\]
Further, we want to show that in case $q_1<1$, 
\[ \liminf_{K \to \infty} \mathcal I^{j_0}(K,\eps) \geq 1-q_1-o(\eps), \numberthis\label{survivallowerdormants}\]
where we define
\[ \mathcal I^{1}(K,\eps):=\P \Big( \Big| \frac{T_{S_\beta^1} }{\log K}-\Big( \frac{1}{\widetilde \lambda} - \frac{1}{\widehat \lambda} \Big)\Big| \leq f(\eps), T_{S_\beta^1} < T_0^1 , B=1 \Big) \]
and
\[ \mathcal I^{\mathrm{co}}(K,\eps):=\P \Big( \Big| \frac{T_{S_\beta^{\mathrm{co}}}}{\log K}- \frac{1}{\widetilde \lambda}  \Big| \leq f(\eps), T_{S_\beta^{\mathrm{co}}} < T_0^1 \wedge T_{S_\beta^1}, B=1 \Big), \]
further, we put $j_0=\mathrm{co}$ in case $(n_{1a},n_{1d},n_2)$ exists as a coordinatewise positive equilibrium (this corresponds to either the case when both \eqref{firstinmutantlessfit2ineq} and \eqref{secondinmutantlessfit2ineq} hold; cf.~also Lemmas~\ref{lemma-coexistence} and~\ref{lemma-lessconditions}) and $j_0=1$ otherwise (which corresponds to the case when \eqref{secondinmutantlessfit2ineq} and \eqref{reversefirstin} hold). 
Throughout the proof, $\beta>0$ is to be understood as sufficiently small; we will comment on how small it should be in the individual cases.

The assertions~\eqref{extinctionlowerdormants} and~\eqref{survivallowerdormants} together will imply Theorem~\ref{thm-invasionof1} and the equation \eqref{extinctionof1} in Theorem~\ref{thm-failureof1}. The other assertion of Theorem~\ref{thm-failureof1}, equation \eqref{lastoftheoremof1}, is a consequence of \eqref{secondofpropdormants}.

Let us start with the case of mutant extinction in the first phase of invasion and verify \eqref{extinctionlowerdormants}. Clearly, we have 
\[ \mathcal E(K,\eps)\geq \P \Big( \frac{T_0^1}{\log K} \leq f(\eps), T_0^1 < T_{S_\beta^1} \wedge T_{S_\beta^{\rm co}}, B=0, T_0^1 < T_{\eps}^1 \wedge R_{2\eps}^2 \Big). \]
Now, considering our initial conditions, one can choose $\beta>0$ sufficiently small such that for all sufficiently small $\eps>0$ we have
\[ T_{\eps}^1 \wedge R_{2\eps}^2 < T_{S_\beta^1} \wedge T_{S_\beta^{\rm co}}, \]
almost surely. We assume further on during the proof that $\beta$ satisfies this condition. Then,
\[ \mathcal E(K,\eps)\geq \P \Big( \frac{T_0^1}{\log K} \leq f(\eps), B=0, T_0^1 < T_{\eps}^1 \wedge R_{2\eps}^2 \Big). \numberthis\label{andisanddormants} \]
Moreover, similarly to the proof of Proposition~\ref{prop-firstphasedormants} with $\xi=1$, we obtain
\[ \limsup_{K \to \infty} \P \big( \{ B=0 \} \Delta \{ T_0^1 < T_{\eps}^1 \wedge R_{2\eps}^2  \}  \big)=o_\eps(1), \numberthis\label{undefinedsymmdiffdormants} \]
where we recall that $\Delta$ denotes symmetric difference, and
\[ \limsup_{K \to \infty} \P \big( \{ B=0 \} \Delta \{ T_0^{(\eps,+),1} < \infty \}  \big)=o_\eps(1), \]
where we recall the coupling~\eqref{upsilondefdormants} and define the stopping time
\[ T_0^{(\eps,+),1} = \inf \{ t >0 \colon N^{(\eps,+)}_{1a}(t)+N^{(\eps,+)}_{1d}(t) =0 \}. \]
Using \eqref{andisanddormants}, an analogue of \eqref{secondlineHGTguys} implies that
\begin{align*}
\liminf_{K \to \infty} \mathcal E(K,\eps) \geq
\P \Big( T_0^{(\eps,+),1} < \infty \Big) + o_\eps(1).
\end{align*}
Thus, employing an analogue of \eqref{qineqHGTguys}, we obtain \eqref{extinctionlowerdormants}, which implies \eqref{extinctionof1}. 

Let us continue with the case of mutant survival in the first phase of invasion and verify \eqref{survivallowerdormants}. 
Arguing analogously to \eqref{undefinedsymmdiffdormants} but for $\xi=1/2$, we get
\[ \limsup_{K \to \infty} \P \big( \{ B=1 \} \Delta \{ T_{\sqrt \eps}^1  < T_0^1 \wedge R_{2\sqrt\eps}^2  \}  \big)=o_\eps(1). \] 
Thus,
\begin{equation}\label{beforesetsdormantsfixation}
\begin{aligned}
\liminf_{K \to \infty} \mathcal I^{j_0}(K,\eps) & = \liminf_{K \to \infty} \P \Big( \Big| \frac{T_{S_\beta^1}}{\log K} -\Big( \frac{1}{\widetilde \lambda} - \frac{1}{\widehat\lambda} \Big) \Big| \leq f(\eps), T_{S_\beta^1}<T_0^1,  
T^1_{\sqrt \eps} < T_0^1 \wedge R_{2\sqrt\eps}^2 \Big) + o_\eps(1)
\end{aligned}
\end{equation} 
if $j_0=1$, and
\begin{equation}\label{beforesetsdormantscoex}
\begin{aligned}
\liminf_{K \to \infty} \mathcal I^{j_0}(K,\eps) & = \liminf_{K \to \infty} \P \Big( \Big| \frac{T_{S_\beta^{\rm co}}}{\log K} - \frac{1}{\widetilde \lambda}  \Big| \leq f(\eps), T_{S_\beta^{\mathrm{co}}}<T_0^1 \wedge T_{S_\beta^{1}},  
 T^1_{\sqrt \eps} < T_0^1 \wedge R_{2\sqrt\eps}^2 \Big) + o_\eps(1)
\end{aligned}
\end{equation} 
if $j_0=\mathrm{co}$. 
For $\eps>0,\beta>0$ and $\delta$ satisfying the conditions of Proposition~\ref{prop-secondphasedormants}, we introduce the sets
\[ \begin{aligned}
\mathfrak B^1_\eps &:= [\pi_{1a}-\delta,\pi_{1a}+\delta] \times [\eps/\widehat C,\sqrt \eps] \times [\bar n_2-2\sqrt\eps,\bar n_2+2\sqrt\eps], \\
\mathfrak B^2_\beta &:= [\bar n_{1a}-(\beta/2),\bar n_{1a}+(\beta/2)] \times [\bar n_{1d}-(\beta/2),\bar n_{1d}+(\beta/2)] \times [0,\beta/2],
\end{aligned}
\]
and the stopping times
\[
\begin{aligned}
T'_\eps:=& \inf \Big\{ t \geq 0 \colon \Big( \frac{N_{1a,t}^K}{N_{1a,t}^K+N_{1d,t}^K}, N_{1a,t}^K+N_{1d,t}^K, N_{2,t}^K \Big)\in \mathfrak B^1_\eps \Big\}, \\
T''_{\beta}:=&\inf \Big\{ t \geq T'_\eps \colon \mathbf N^K_t \in \mathfrak B^2_\beta \Big\}.
\end{aligned}
\]
Note that for $j_0=1$, $\mathfrak B^2_\beta$ is in fact the analogue of $S_{\beta}^{\mathrm{co}}$ for $j_0=\mathrm{co}$: these are small closed neighbourhoods of the metastable state that the rescaled population process will reach with high probability conditional on survival of mutants. However, in case $j_0=1$, we additionally need to take into account the time of extinction of the trait 2 population. Thus, informally speaking, we want to show that with high probability the process has to pass through $\Bcal^1_\eps$ in order to reach $S_\beta^{j_0}$, whatever $j_0 \in \{ 1, \mathrm{co} \}$ is, and afterwards it has to go through $\mathfrak B^2_\beta$ in order to reach $S_\beta^1$ in case $j_0=1$
. Then, thanks to the Markov property, we can estimate $T_{S_\beta^{j_0}}$ by estimating $T'_\eps$, $T''_{\beta}-T'_\eps$, and $T_{S_\beta^{1}}-T''_{\beta}$ in case $j_0=1$, and by estimating $T'_\eps$ and $T_{S_{\beta}^{\mathrm{co}}}-T'_\eps$ in case $j_0 = \mathrm{co}$.

In case $j_0=1$, \eqref{beforesetsdormantsfixation} implies that 
\begin{align*}
    \liminf_{K \to \infty} \mathcal I^1(K,\eps) & \geq  \P \Big( \Big| \frac{T_{S_\beta^1}}{\log K} -\Big( \frac{1}{\widetilde \lambda} - \frac{1}{\widehat \lambda} \Big) \Big| \leq f(\eps), T^1_{\sqrt \eps} < T_0^1 \wedge R_{2\sqrt\eps}^2, T''_{\beta}<T_{S_\beta^1}, T_{S_\beta^1}<T_0^1 \Big) + o_\eps(1) \\
    & \geq   \P \Big( \Big| \frac{T'_\eps}{\log K} -\frac{1}{\widetilde \lambda} \Big| \leq \frac{f(\eps)}{3}, \Big| \frac{T''_{\beta}-T'_\eps}{\log K} \Big| \leq \frac{f(\eps)}{3}, \Big| \frac{T_{S_{\beta}^1}-T''_{\beta}}{\log K}+ \frac{1}{\widehat\lambda}\Big| \leq \frac{f(\eps)}{3},  \\
    & \qquad   T^1_{\sqrt \eps} < T_0^1 \wedge R_{2\sqrt\eps}^2 , T''_{\beta}<T_{S_\beta^1}, T_{S_\beta^1}<T_0^1 \Big) + o_\eps(1) ,
\end{align*}
whereas in case $j_0=\mathrm{co}$, \eqref{beforesetsdormantscoex} implies that
\begin{align*}
    \liminf_{K \to \infty}\mathcal I^{\mathrm{co}} & (K,\eps) \geq    \P \Big( \Big| \frac{T_{S_\beta^{\mathrm{co}}}}{\log K} - \frac{1}{\widetilde \lambda} \Big| \leq f(\eps), T^1_{\sqrt \eps} < T_0^1 \wedge R_{2\sqrt\eps}^2, T_{S_\beta^{\mathrm{co}}}<T_0^1 \wedge T_{S_\beta^1} \Big) + o_\eps(1) \\
    \geq &  \P \Big( \Big| \frac{T'_\eps}{\log K} -\frac{1}{\widetilde \lambda} \Big| \leq \frac{f(\eps)}{3}, \Big| \frac{T_{S_\beta^\mathrm{co}} -T'_\eps}{\log K} \Big| \leq \frac{f(\eps)}{3},  T^1_{\sqrt \eps} < T_0^1 \wedge R_{2\sqrt\eps}^2 , T_{S_\beta^\mathrm{co}}<T_0^1 \wedge T_{S_\beta^1} \Big) + o_\eps(1) ,
\end{align*}
Note that for $\beta>0$ sufficiently small and $\eps>0$ sufficiently small chosen accordingly, $R_{2\sqrt\eps}^2 \leq T_{S_\beta^{j_0}}$ almost surely, whatever $j_0$ is. Hence, the strong Markov property applied at times $T'_\eps$ and $T''_ {\beta}$ implies 
\[ \begin{aligned}
    \liminf_{K \to \infty} \mathcal I^1(K,\eps)&\geq \liminf_{K \to \infty} \Big[ \P \Big( \Big| \frac{T'_\eps}{\log K} -\frac{1}{\widetilde \lambda} \Big| \leq \frac{f(\eps)}{3}, T'_\eps<T_0^1, T^1_{\sqrt\eps} < T_0^1 \wedge R_{2\sqrt\eps}^2 \Big) \\
    & \qquad \times \inf_{\begin{smallmatrix}\mathbf n=(n_{1a},n_{1d},n_{2}) \colon \big(\frac{n_{1a}}{n_{1a}+n_{1d}},n_{1a}+n_{1d},n_2\big) \in \mathfrak B^1_\eps\end{smallmatrix}} \P \Big(  \Big| \frac{T''_{\beta}-T'_\eps}{\log K} \Big| \leq \frac{f(\eps)}{3}, T''_{\beta} < T_0^1 \Big| \mathbf N^K_0=\mathbf n \Big) \\
    & \qquad \times \inf_{\mathbf n \in \mathfrak B^2_\beta} \P \Big(\Big| \frac{T_{S_{\beta}^1}-T''_{\beta}}{\log K}+\frac{1}{\widehat \lambda}\Big| \leq \frac{f(\eps)}{3}, T_{S_\beta^1} < T_0^1  \Big| \mathbf N_0^K = \mathbf n\Big) \Big]+o_{\eps}(1) \end{aligned} \numberthis\label{productformfixation1} \]
in case $j_0=1$, and 
\[ \begin{aligned}
    \liminf_{K \to \infty} &\mathcal I^{\mathrm{co}}(K,\eps)\geq \liminf_{K \to \infty} \Big[ \P \Big( \Big| \frac{T'_\eps}{\log K} -\frac{1}{\widetilde \lambda} \Big| \leq \frac{f(\eps)}{3}, T'_\eps<T_0^1, T^1_{\sqrt\eps} < T_0^1 \wedge R_{2\sqrt\eps}^2 \Big) \\
    & \qquad \times \inf_{\begin{smallmatrix}\mathbf n=(n_{1a},n_{2},n_{2}) \colon \big(\frac{n_{1a}}{n_{1a}+n_{1d}},n_{1a}+n_{1d},n_1\big) \in \mathfrak B^1_\eps\end{smallmatrix}} \P \Big(  \Big| \frac{T_{S_\beta^\mathrm{co}} -T'_\eps}{\log K} \Big| \leq \frac{f(\eps)}{3}, T_{S_\beta^{\mathrm{co}}} < T_0^1 \wedge T_{S_\beta^1} \Big| \mathbf N^K_0=\mathbf n \Big) \Big]\\
   \end{aligned} \numberthis\label{productformcoex1} \] in case $j_0=\mathrm{co}$.
It remains to show that the right-hand side of \eqref{productformfixation1} respectively \eqref{productformcoex1} is close to $1-q_1$ as $K \to \infty$ if $\eps$ is small. We begin with the first term (note that this one appears on the right-hand side of both equations). We want to verify that 
\[ \liminf_{K \to \infty} \P \Big( \Big| \frac{T'_\eps}{\log K} -\frac{1}{\widetilde \lambda} \Big| \leq \frac{f(\eps)}{3}, T'_\eps<T_0^1, T^1_{\sqrt\eps} < T_0^1 \wedge R_{2\sqrt\eps}^2 \Big) \geq 1-q_1+o_{\eps}(1). \numberthis\label{firsttermdormants} \]
This can be done analogously to \cite[Proof of (3.61)]{C+19}, where Proposition~\ref{prop-secondphasedormants} plays the role of \cite[Proposition 3.2]{C+19}. 

Now, let us treat the second phase of invasion, both in the case $j_0=1$ (fixation of trait 1, which corresponds to an almost-fixation already by the end of the second phase) and in the case $j_0=\mathrm{co}$ (convergence to the coexistence equilibrium). 
For $\mathbf m=(m_{1a},m_{1d},m_{2}) \in [0,\infty)^3$, let us recall that $\mathbf n^{(\mathbf m)}$ denotes the unique solution of the dynamical system \eqref{3dimHGT} with initial condition $\mathbf m$. Thanks to the continuity of flows of this dynamical system with respect to the initial condition (cf.~\cite[Theorem 1.1]{DLA06}) and thanks to the convergence provided by Lemma~\ref{lemma-1invades2LV}, we deduce that if $\beta>0$ is small enough, then there exist $\eps_0,\delta_0>0$ such that for all $\eps \in (0,\eps_0)$ and $\delta \in (0,\delta_0)$, there exists $t_{\beta,\delta,\eps}>0$ such that for all $t>t_{\beta,\delta,\eps}$,
\[ \Big\Vert \mathbf n^{(\mathbf n^0)}(t)-(\bar n_{1a},\bar n_{1d},0) \Big\Vert \leq \frac{\beta}{4} \]
holds in case $j_0=1$, and
\[ \Big\Vert \mathbf n^{(\mathbf n^0)}(t)-(n_{1a},n_{1d},n_2) \Big\Vert \leq \frac{\beta}{4} \]
holds in case $j_0=\mathrm{co}$, 
for any initial condition $n^0=(n_{1a}^0,n_{1d}^0,n_{2}^0)$ such that $(n_{1a}^0/(n_{1a}^0+n_{1d}^0),n_{1a}^0+n_{1d}^0,n_2^0) \in \mathfrak B^1_\eps$. Indeed, because of Lemma~\ref{lemma-goodstartdormants}, $n^0$ satisfies \eqref{proportioncond1}, \eqref{proportioncond2}, and \eqref{+0cond} in case $n_{1a}^0/(n_{1a}^0+n_{1d}^0)$ is equal to $\pi_{1a}$, and for all sufficiently small $\eps>0$, the same follows by continuity for all $n^0=(n_{1a}^0,n_{1d}^0,n_{2}^0)$ such that  $(n_{1a}^0/(n_{1a}^0+n_{1d}^0),n_{1a}^0+n_{1d}^0,n_2^0) \in \mathfrak B^1_\eps$.

Now, using \cite[Theorem 2.1, p.~456]{EK}, we conclude that for all sufficiently small $\beta>0$ and accordingly chosen $\eps<\eps_0$, the following hold. If $j_0=1$, then
\[ \lim_{K \to \infty} \P\Big( T''_\beta-T'_\eps \leq t_{\beta,\delta,\eps} \Big|\Big( \frac{N_{1a,0}^K}{N_{1a,0}^K+N_{1d,0}^K}, N_{1,0}^K, N_{2,0}^K \Big) \in \mathfrak B^1_\eps \Big) =1-o_\eps(1), \]
whereas if $j_0=\mathrm{co}$, then
\[ \lim_{K \to \infty} \P\Big( T_{S_\beta^{\mathrm{co}}}-T'_\eps \leq t_{\beta,\delta,\eps} \Big|\Big( \frac{N_{1a,t}^K}{N_{1a,t}^K+N_{1d,t}^K}, N_{1,t}^K, N_{2,t}^K \Big) \in \mathfrak B^1_\eps \Big) =1-o_\eps(1). \]
Thus, for $j_0=1$, the second term on the right-hand side of \eqref{productformfixation1} is close to 1 when $K$ tends to $\infty$, $\beta$ is small and $\eps>0$ is small enough chosen according to $\beta$. Similarly, for $j_0=\mathrm{co}$, the second term on the right-hand side of \eqref{productformcoex1} is close to 1 when $K$ tends to $\infty$, $\beta$ is small and $\eps>0$ is small enough chosen according to $\beta$. In the latter case, together with \eqref{firsttermdormants} we have obtained
\[ \liminf_{K \to \infty} \mathcal I^{\mathrm{co}}(K,\eps) \geq 1-q_1-o_\eps(1), \]
which implies \eqref{coexprob1} and \eqref{coexistenceof1}. 

Finally, we investigate the third term on the right-hand side of \eqref{productformfixation1} (in the case $j_0=1$ of fixation of trait 1). Proposition~\ref{prop-thirdphasedormants} implies that there exists $\beta_0>0$ (denoted as $\eps_0$ in Proposition~\ref{prop-thirdphasedormants}) such that for all $\beta<\beta_0$, for $\eps>0$ sufficiently small,
\[ \lim_{K \to \infty} \P \Big(\Big| \frac{T_{S_{\beta}^1}-T''_{\beta}}{\log K}+ \frac{1}{\widehat \lambda} \Big| \leq \frac{f(\eps)}{3} \Big| \mathbf N_0^K \in \mathfrak B^2_{\beta} \Big) = 1-o_\eps(1). \]
Combining \eqref{firsttermdormants} with the convergence of the second and the third term on the right-hand side of \eqref{productformfixation1} to 1, we obtain
\[ \liminf_{K \to \infty} \mathcal I^1(K,\eps) \geq 1-q_1-o_\eps(1), \]
which implies \eqref{invasionprob1} and \eqref{invasionof1}.

\section*{Acknowledgements}
The authors thank F.~Nie and C.~Smadi for interesting discussions and comments, and T.~Paul for pointing out an error in an earlier version of the manuscript. 
We also thank the referees and the editor of {\em Theoret.\ Pop.\ Biol.} for valuable comments.
JB was supported by DFG Priority Programme 1590 ``Probabilistic Structures in Evolution’’ and Berlin Mathematics Research Center MATH+. AT was supported by DFG Priority Programme 1590 ``Probabilistic Structures in Evolution’’.


\begin{thebibliography}{WWWW98}


\bibitem[AN72]{AN72}
{\sc K. B. Athreya} and {\sc P. E. Ney},
\emph{Branching processes},
Springer (1972).
\bibitem[B04]{B04}
{\sc N.Q. Balaban, J. Merrin, R. Chait,  L. Kowalik, S. Leibler}
Bacterial Persistence as a Phenotypic Switch.
\emph{Science} {\bf 305} (2004).
\bibitem[B20]{B19}
{\sc A. Bovier.} 
Stochastic models for adaptive dynamics. Scaling limits and diversity. To appear in {\em Probabilistic Structures in Evolution}, A. Baake and A. Wakolbinger, Eds. (2020).
\bibitem[BCFMT16]{BCFMT16}
{\sc S. Billiard, P.  Collet, R. Ferrière, S. Méléard} and {\sc V. C. Tran,}
The effect of competition and horizontal trait inheritance on invasion, fixation, and polymorphism. 
\emph{J. Theoret. Biol.} {\bf 411}, 48–-58 (2016). 
\bibitem[BCFMT18]{BCFMT18}
{\sc S. Billiard, P.  Collet, R. Ferrière, S. Méléard,} and {\sc V. C. Tran,}
Stochastic dynamics for adaptation and evolution of microorganisms,
\emph{Journal of the European Mathematical Society},
pages 527--552, special issue for the Proceedings ECM2016 (2018). 
\bibitem[BEGKW15]{BEGKW15}
{\sc J. Blath}, {\sc B. Eldon}, {\sc A. González Casanova},  {\sc N. Kurt} and {\sc M. Wilke-Berenguer},
Genetic variability under the seedbank coalescent. \emph{Genetics}, {\bf 200:3}, 921--934 (2015).
\bibitem[BGKW16]{BGKW16} {\sc J. Blath}, {\sc A. González Casanova}, {\sc N. Kurt}, and {\sc M. Wilke-Berenguer},  A new coalescent for seed-bank models, {\em Ann. Appl. Probab.} \textbf{26:2}, 857--891 (2016). 
\bibitem[BGKW20]{BGKW20} {\sc J. Blath}, {\sc A. González Casanova}, {\sc N. Kurt}, and {\sc M. Wilke-Berenguer}, The seed bank coalescent with simultaneous switching, \emph{Electron.~J.~Probab.}, {\bf 25}, paper no.~27, 21 pp. (2020).
\bibitem[BP14]{BP14} {\sc F. Baumdicker} and {\sc P. Pfaffelhuber}. The infinitely many genes model with horizontal gene transfer, \emph{Electron.~J.~Probab.}, {\bf 19:115}, 1-27 (2014).
\bibitem[BT20]{BT19}
{\sc J. Blath} and {\sc A. Tóbiás},
Invasion and fixation of microbial dormancy traits under competitive pressure,  \emph{Stoch. Process. Their Appl.}, {\bf 130:12}, 7363--7395 (2020).
\bibitem[B84]{B84}
{\sc M.G. Bulmer},
Delayed Germination of Seeds: Cohen’s Model Revisited
\emph{Theor. Pop. Biol.} {\bf 26}, 367--377 (1984).
\bibitem[C06]{C06}
{\sc N. Champagnat},
A microscopic interpretation for adaptive
dynamics trait substitution sequence models,
{\em Stochastic Process. Appl.}, {\bf 116:8}, 1127--1160 (2006).
\bibitem[Coh66]{Co66}
{\sc D.~Cohen},
\newblock Optimizing reproduction in a randomly varying environment.
\newblock {\em J. Theoret. Biol.}, {\bf 12}, 119--129 (1966).
\bibitem[CCLS17]{C+16}
{\sc C. Coron}, {\sc M. Costa}, {\sc H. Leman}, and {\sc C. Smadi},
A stochastic model for speciation by mating preferences,
\emph{J.~Math.~Biol.}, {\bf 76}, 1421--1463 (2018). 
\bibitem[CCLLS21]{C+19}
{\sc C. Coron}, {\sc M. Costa}, {\sc F. Laroche}, {\sc H. Leman}, and {\sc C. Smadi},
Emergence of homogamy in a two-loci stochastic population model,
\emph{ALEA, Lat. Am. J. Probab. Math. Stat.} {\bf 18}, 469-–508 (2021).
\color{black} \bibitem[CMT21]{CMT19}
{\sc N. Champagnat}, {\sc S. Méléard}, and {\sc V. C. Tran},
Stochastic analysis of emergence of evolutionary cyclic behavior
in population dynamics with transfer, {\em Ann. Appl. Probab.}, to appear, see also:
\emph{arXiv:1901.02385} (2021).
\bibitem[DLA06]{DLA06}
{\sc F. Dumortier}, {\sc J. Llibre}, and {\sc J. C. Artés},
Qualitative Theory of Planar Differential Systems. Springer (2006). 
\bibitem[E85]{E85}
{\sc S. Ellner},
Germination Strategies in 
 Randomly Varying  Environments. I. Logistic-type 
Models.
\emph{Theor. Pop. Biol.} {\bf 28}, 50--79 (1985).
\bibitem[EK86]{EK}
{\sc S. N. Ethier} and {\sc T. G. Kurtz},
\emph{Markov processes. Characterization and convergence},
Wiley Series in Probability and Mathematical Statistics: Probability and Mathematical Statistics. John Wiley \& Sons Inc., New York (1986). 
\bibitem[FW84]{FW84}
{\sc M. Freidlin} and {\sc A. D. Wentzell},
{\em Random perturbations of dynamical systems}, Grundlehren der Mathematischen Wissenschaften (Fundamental Principles of Mathematical Sciences), volume 260,
Springer (1984). \color{black}
\bibitem[GB03]{GB03}
{\sc H-O. Georgii} and {\sc E. Baake},
Supercritical multitype branching processes: the ancestral types of typical individuals.
\emph{Adv. App. Probab.}, {\bf 35:4}, 1090--1110 (2003). 
\bibitem[GB14]{GB14}
{\sc C. Gyles} and {\sc P. Boerlin},
Horizontally transferred genetic elements and their role in pathogenesis of bacterial disease.
\emph{Veterinary Pathology}, {\bf 51:2}, 328--334 (2014). 
\bibitem[KKL01]{KKL01} {\sc I. Kaj}, {\sc S. Krone}, and {\sc M. Lascoux}, {\rm Coalescent theory for seed bank models.} {\it J. of Appl. Probab.}, {\bf 38:2}, 285--300 (2001). 
\bibitem[KW12]{KW12}
{\sc E. V. Koonin} and {\sc Y.~I. Wolf},
\newblock Evolution of microbes and viruses: a paradigm shift in evolutionary
  biology?
\newblock {\em Frontiers in cellular and infection microbiology}, {\bf 2}, 119 (2012).
\color{blue}
\begin{comment}{\color{gray} \bibitem[KSVHJ11]{KSVHJ11}
\textsc{F. C. Klebaner}, \textsc{S. Sagitov}, \textsc{V. A. Vatutin}, \textsc{P. Haccou}, and \textsc{P. Jagers},
Stochasticity in the adaptive dynamics of
evolution: the bare bones, \emph{Journal of
Biological Dynamics}, \textbf{5:2}, 147--162 (2011). \color{black}}\end{comment}
\color{black}
\bibitem[KL+05]{KL+05}
\textsc{E. Kussell, R. Kishony, N.Q. Balaban, S. Leibler},
Bacterial Persistence: A Model of Survival in Changing Environments.
\emph{Genetics} {\bf 169:4}, 8 pp (2005).
\bibitem[L10]{L10}
\textsc{K. Lewis,}
Persister Cells.
\emph{Annu. Rev. Microbiol.} {\bf 64}, 357–-372 (2010).
\bibitem[LT46]{LT46}
{\sc J. Lederberg} and {\sc E. L. Tatum}.
\newblock Gene recombination in escherichia coli.
\newblock {\em Nature}, {\bf 158}:558 (1946).
\bibitem[LJ11]{LJ11}
\textsc{J. T. Lennon} and {\sc S. E. Jones},
Microbial seed banks: the ecological and evolutionary implications of dormancy. \emph{Nat. Rev. Microbiol.} \textbf{9:2}, 119--130 (2011). 
\bibitem[LdHWB20]{LdHWB20}
\textsc{J. T. Lennon} and {\sc den Hollander, F. T. W.} and {\sc Wilke Berenguer, M.} and {\sc Blath, J.},
Principles of seed banks: complexity emerging from dormancy,
\emph{arXiv:2012.00072}, (2020).
\bibitem[OLG00]{OLG00}
{\sc H. Ochman, J. G. Lawrence}, and {\sc E. A. Groisman},
Lateral gene transfer and the nature of bacterial innovation.
\emph{Nature} {\bf 405}, 299--304 (2000).
\bibitem[RC87]{RC87}
{\sc D. B. Roszak} and {\sc R. R. Colwell},
Survival strategies of bacteria in the natural environment. 
\emph{Microbiol Rev.} 51:365–-379 (1987). 
\bibitem[SL18]{SL18}
{\sc W.~R.~Shoemaker} and {\sc J.~T.~Lennon},
Evolution with a seed bank: The population genetic consequences of
  microbial dormancy,
\textit{Evol.~Appl.} \textbf{11}:60--75 (2017).
\bibitem[SD73]{SD73}
{\sc A. S. Sussman} and {\sc H. A. Douthit}, 
Dormancy in microbial spores. 
\emph{Annu. Rev. Plant Physiol.} {\bf 24}:311–352 (1973). 
\bibitem[TLLPS11]{T11} {\sc A.~Tellier}, {\sc S.J.Y.~Laurent}, {\sc H.~Lainer}, {\sc P.~Pavlidis}, and {\sc W.~Stephan}, Inference of 
seed bank parameters in two wild tomato species using ecological and 
genetic data. \emph{Proceedings of the National Academy of Sciences of the
U.S.A.}~\textbf{108:41}, 17052--17057 (2011).
\bibitem[V15]{V15} {\sc J. Vandermeer}, Some complications of the elementary forms of competition in a source/sink and metacommunity context: the role of intranstive loops. \emph{arXiv:1502.05225} (2015). 
\bibitem[W04]{W04} 
{\sc R. J. Whittington}, {\sc D. J. Marshall}, {\sc P. J. Nicholls}, {\sc I. B. Marsh}, and {\sc L. A. Reddacliff}, 
Survival and dormancy of Mycobacterium avium subsp. paratuberculosis in the environment. 
\emph{Appl Environ Microbiol.} {\bf 70}, 2989--3004 (2004).
\end{thebibliography}
\end{document}